\documentclass[a4paper,11pt,dvips]{amsart}

\usepackage{float}
\usepackage{fullpage}
\usepackage{multirow,yfonts,amsthm}
\usepackage{dsfont}
\usepackage[english]{babel}
\usepackage{amssymb, amsfonts}
\usepackage{amsmath}
\usepackage{stmaryrd}
\usepackage[all]{xy}
\usepackage{graphicx}
\usepackage{chngpage}
\usepackage[toc,page]{appendix} 
\usepackage{subfigure}
\usepackage{xcolor}
\usepackage{hyperref}
\usepackage{accents}
\usepackage{nicefrac}
\usepackage{ upgreek }
\usepackage{rotating}
\usepackage{enumerate}

\definecolor{gris}{gray}{0.45}

\newtheorem{thm}{Theorem}[subsection]
\newtheorem{lem}[thm]{Lemma}
\newtheorem{prop}[thm]{Proposition}
\newtheorem{prop/def}[thm]{Proposition/Definition}

\theoremstyle{definition}

\theoremstyle{remark}
\newtheorem{rem}[thm]{Remark}

\theoremstyle{plain}

\DeclareMathOperator{\Hom}{Hom}

\newcommand {\N}{\mathbb{N}}

\newcommand{\R}{\mathbb{R}}
\newcommand{\C}{\mathbb{C}} 
\newcommand{\D}{\mathbb{D}} 
 
\newcommand{\Z}{\mathbb{Z}}

 \definecolor{grass}{rgb}{0.14,0.72,0.2}

  \newcommand{\comE}[1]{ }         
\newcommand{\chign}[1]{\chi_{g,n}(#1)}
\newcommand{\MCG}{\Gamma_{g,n}}
\newcommand{\fMCG}{\hat{\Gamma}_{g,n}}

\newcommand{\fMCGplus}{\hat{\Gamma}_{g,n}^{\bullet}}
\newcommand{\MCGg}{\Gamma_{g}^1}

\newcommand{\liftMCG}{\hat{\Gamma}_{g,n}^\circ}
\newcommand{\Bn}{B_n}
\newcommand{\T}{\mathcal T}
\newcommand{\Tgn}{\mathcal{T}_{g,n}}
\newcommand{\forggn}{\pi_{g,n}}
\newcommand{\tautgn}{\mathrm{taut}_{\hat{\star}}}
\newcommand{\taut}{\mathrm{taut}}
\newcommand{\tautgnplusun}{\mathrm{taut}_{\hat{\star}^{\bullet}}}
\newcommand{\Mgn}{\mathcal{M}_{g,n}}
\newcommand{\Mgnplusun}{\mathcal{M}_{g,n+1}}
\newcommand{\classan}{{\mathrm{class}^+}}
\newcommand{\class}{\mathrm{class}}

\newcommand{\gf}{\Lambda_{g,n}}
\newcommand{\GL}{\mathrm{GL}_r\C}
\renewcommand{\emptyset}{\varnothing}
\usepackage{amsmath,tikz}  

\newcommand{\triang}{\mathrm{Upp}}
\newcommand{\GLtw}{\mathrm{GL}_2\mathbb{C}}
\newcommand{\Aff}{\mathrm{Aff}(\mathbb{C})}
\DeclareMathOperator{\card}{\mathrm{card}}

\newcommand{\intint}[2]{\llbracket #1, #2\rrbracket}
\newcommand{\cardn}{\card}

\newcommand{\bigslant}[2]{{\raisebox{.0em}{$#1$}\raisebox{-.1em}{$/$}\raisebox{-.2em}{$ #2 $}}}

\newcommand{\morph}{\mathfrak{f}}
\newcommand{\familyCurve}{\mathcal{C}}
\newcommand{\familyBundle}{\mathcal{E}}
\newcommand{\initialBundle}{E}
\newcommand{\trivvb}{\mathbb{V}}

\newcommand{\Bgn}{\mathcal{B}_{g,n}}
\newcommand{\Bgnplusun}{\mathcal{B}_{g,n+1}}
\newcommand{\Rgn}{\mathcal{R}_{g,n}}

\newtheorem{bigthm}{Theorem}

\newcommand{\logD}{\mathrm{log}\, D}

 \newcommand{\clF}{\mathrm{cl}}
\newtheorem*{thmA1}{Theorem A1}
\newtheorem*{thmA2}{Theorem A2}
\newtheorem*{thmB1}{Theorem B1}
\newtheorem*{thmB2}{Theorem B2} 
\newcommand{\SLtwoZ}{\mathrm{SL}_2\mathbb{Z}}

\author[G. Cousin]{Ga\"el Cousin}
\address{GMA-IME\\Universidade Federal Fluminense\\
Campus do Gragoat\'a\\
Niter\'oi\\RJ \\ Brazil}
\email{gcousin@id.uff.br }
\author[V. Heu]{Viktoria Heu}
\address{ IRMA\\
 7 rue Ren\'e Descartes\\
 67084 Strasbourg\\France}
\email{heu@math.unistra.fr}
\begin{document}
\title{Algebraic isomonodromic deformations and \\the mapping class group}
 
\begin{abstract}The  germ of the universal isomonodromic deformation of a logarithmic connection on a stable $n$-pointed genus $g$ curve always exists in the analytic category. The first part of this article investigates under which conditions it is the analytic germification of an algebraic isomonodromic deformation. 
Up to some minor technical conditions, this turns out to be the case if and only if the monodromy of the connection has finite orbit under the action of the mapping class group.  The second part of this work studies the dynamics of this action in the particular case of reducible rank $2$ representations and genus $g>0$, allowing to classify all finite orbits. Both of these results extend recent ones concerning the genus $0$ case.
\end{abstract}
  \thanks{The authors would like to warmly thank Gwena\"el Massuyeau for  discussions around the mapping class group. We are also grateful to the anonymous referee for useful suggestions.
This work took place at IRMA and LAREMA. It was supported by ANR-13-BS01-0001-01,
ANR-13-JS01-0002-01 and Labex IRMIA}
\maketitle
\tableofcontents    
 
\section{Introduction}\label{sec action fg}
\textbf{The mapping class group. } 
Let $g$ and $n$ be nonnegative integers. Let $\Sigma_g$ be a compact oriented real surface of genus $g$,  let $y^n=(y_1, \ldots, y_n)$ be a sequence of $n$ distinct points in $\Sigma_g$. We shall denote by $Y^n:=\{y_1, \ldots, y_n\}$ the corresponding (unordered) set of points.    
The (pure) mapping class group of $(\Sigma_g,y^n)$ is defined to be the set of orientation preserving homeomorphisms $h$ of $\Sigma_g$ such that $h(y_i)=y_i$ for all $i\in \intint{1}{n}:=\left\{ k \in \mathbb{Z}~\middle|~1\leq k\leq n\right\}$, quotiented by isotopies:
$$\MCG:=\bigslant{\mathrm{Homeo}_+(\Sigma_g, y^n)}{\{\mbox{isotopies relative to $Y^n$}\}\, .}$$
We can also consider homeomorphisms of $\Sigma_g$ that preserve the set $Y^n$, but do not necessarily preserve the labelling of the punctures. This leads to the full mapping class group
$$\fMCG:=\bigslant{\mathrm{Homeo}_+(\Sigma_g,Y^n)}{\{\mbox{isotopies relative to $Y^n$}\}\, .}$$
Note that we have an exact sequence of groups $$1\longrightarrow \MCG\longrightarrow \fMCG \longrightarrow\mathfrak{S}_n \longrightarrow 1,$$
where $\mathfrak{S}_n$ denotes the symmetric group of degree $n$. In particular, $\MCG$ is a subgroup of $\fMCG$ of finite index $n!$~.
Let now $y_0 \in \Sigma_g\setminus Y^n$ be a point.
We denote the fundamental group of $\Sigma_g\setminus Y^n$ with respect to the base point $y_0$ by 
  \begin{equation}\gf:=\pi_1(\Sigma_g \setminus Y^n,y_{0})\, .\end{equation}
  The composition $\alpha \, .\, \alpha'$ of two paths $\alpha, \alpha'\in \gf$ shall denote the usual concatenation ``first $\alpha$, then $\alpha'$''. 
For any group $G$, we may consider the space $\Hom(\gf, G)$ of representations as well as the set of representations modulo conjugation, which we shall denote 
\begin{equation}\chign{G}:= \bigslant{\Hom(\gf, G)}{G}\, .\end{equation}

\textbf{The mapping class group acts on $\chign{G}$. } 
Define the groups of orientation preserving homeomorphisms $h$ of $\Sigma_g$ such that $h(y_0)=y_0$ and
$h(y^n)=y^n$, respectively $h(Y^n)=Y^n$, modulo isotopy:
 \[\begin{array}{rll}\Gamma_{g,n+1}&:=&\bigslant{\mathrm{Homeo}_+(\Sigma_g,y^n, y_0)}{\{\mbox{isotopies relative to $Y^n \cup\{y_0\}$}\} \, , }\vspace{.3cm}\\
 \fMCG ^{\bullet}&:=&\bigslant{\mathrm{Homeo}_+(\Sigma_g, Y^n, y_0)}{\{\mbox{isotopies relative to $Y^n \cup\{y_0\}$}\}\, .} \end{array}\]
 Now $\fMCG ^{\bullet}$ naturally acts on the fundamental group 
$\gf$: for $h \in \fMCG^{\bullet}$ and $\alpha \in \gf$, we set 
$$\mathfrak{a}(h)(\alpha) :=h_*\alpha\, .$$ \textit{Via} the forgetful maps $\Gamma_{g,n+1}\to \MCG$ and $\fMCG ^{\bullet}\to \fMCG$ we obtain a commutative diagram
\[ \hspace{3 cm} \begin{xy}\xymatrix{\Gamma_{g,n+1} \ar@{->>}[d]\ar@{^{(}->}[r]&\fMCG ^{\bullet} \ar[r]^{\mathfrak{a} \textrm{~~~~~}}\ar@{->>}[d]& \mathrm{Aut}(\gf)\ar@{->>}[d]&\\ \MCG \ar@{^{(}->}[r]&\fMCG \ar[r]& \mathrm{Out}(\gf):\ar@{=}[r]&\bigslant{\mathrm{Aut}(\gf)}{\mathrm{Inn}(\gf)\, .}
}
\end{xy}\]
Indeed, any element $h\in \mathrm{Homeo}_+(\Sigma_g, y^n)$ may be lifted to an element $h_0\in \mathrm{Homeo}_+(\Sigma_g,y^n, y_0)$. Let $h_1\in \mathrm{Homeo}_+(\Sigma_g,y^n, y_0)$ be another representative. Then they are the extremities of an isotopy $(h_t)_{t\in [0,1]}$ relative to $Y^n$. We have a loop $\gamma \in \gf$ defined by $\gamma(t)=h_t(y_0)$. Then for any $\alpha \in \gf$, we have $\mathfrak{a}(h_1)(\alpha)=\gamma^{-1}\, .\, \mathfrak{a}(h_0)(\alpha)\, .\, \gamma$\, . 

In particular, for any group $G$, the mapping class group $\fMCG$ acts on the set $\chign{G}$, and this action lifts to an action of 
 $\fMCG ^{\bullet}$ on the space $\Hom(\gf, G)$. More precisely, for all $\rho\in \Hom(\gf, G), h\in \fMCG ^{\bullet}$ and $\alpha \in \gf$, we define
 \begin{equation} (h\cdot \rho )(\alpha): = \rho(\mathfrak{a}(h^{-1})(\alpha))\, .\end{equation}
 \vspace{.3cm}
 
\textbf{Application to isomonodromic deformations and a dynamical study. }
In this paper, we establish two results about finite orbits of the mapping class group action on $\chign{G}$ for $G=\GL$. These results and their respective proofs can be read independently. In Theorem  \ref{algebrization thm},  which will be stated in Section \ref{SecIntroAlg} and proven in Part \ref{partAlg}, we relate such finite orbits to the existence of an  algebraic universal isomonodromic deformation of a logarithmic connection over a curve, whose monodromy belongs to that orbit.  This motivates Theorem  \ref{mainthm dynamics},  which will be stated in Section \ref{SecIntroDyn}   
and proven in Part   \ref{partDyn}, classifying conjugacy classes of reducible rank 2 representations with finite orbit. To that end, we introduce a specific presentation of $\gf$ and explicit \textit{formulae} for the mapping class group action.  

\begin{rem}Recall that a  representation $\rho \in \Hom(\gf , \GL)$ is called \emph{irreducible} if the only subvector spaces of $\C^r$ that are stable under $\mathrm{Im}(\rho)$ are $\{0\}$ and $\C^r$. A \emph{semisimple} representation is a direct sum of irreducible representations. 
 \end{rem}

   \renewcommand\thesubsection{\thesection.\Alph{subsection}}
 \setcounter{subsection}{0} 
 \subsection{Algebraization of universal isomonodromic deformations}\label{SecIntroAlg}
We need to introduce some additional vocabulary before stating our main result, which can be seen as a criterion under which a GAGA-type theorem holds for isomonodromic deformations. In order to avoid having to introduce each definition twice, we adopt the $\upomega$-notation described in Table \ref{omeganotation}.  

\begin{table}[H]
$\begin{array}{|c|c|c|} \hline 
\upomega & \upomega -\textrm{manifold} &\upomega -\textrm{open}\\
\hline 
\textrm{analytic}& \textrm{\begin{tabular}{c}connected Hausdorff \\ complex analytic manifold\end{tabular}}& \textrm{open with respect to  Euclidean topology}\\
\hline 
\textrm{algebraic}& \textrm{ \begin{tabular}{c}smooth irreducible\\ quasi-projective variety over~$\C$ \end{tabular}}& \textrm{open with respect to Zariski topology}\\ 
\hline
\end{array}$
\vspace{.2cm}
\caption{ \label{omeganotation}}
\end{table}
 \vspace{-.3cm}

\textbf{Logarithmic connections. } Let $X$ be a  $\upomega$-manifold and let $D$ be a (possibly empty) reduced normal crossing divisor on $X$. Denote by $D^1, \ldots, D^n$ the irreducible components of $D$. A \emph{logarithmic $\upomega$-connection} of rank $r$ over $X$ with polar divisor 
$D$ is a pair $(E,\nabla)$, where $E\to X$ is a $\upomega$-vector bundle of rank $r$ over $X$, whose sheaf of sections we also denote by ${E}$, and $\nabla$ is a $\C$-linear morphism 
$$\nabla : E \to E\otimes \Omega^1_X(\logD)\, , $$
which satisfies the Leibniz rule
$$\nabla (f\cdot e) =f\cdot \nabla(e)+e\otimes \mathrm{d}f\, $$
for any $f\in \mathcal{O}_X(\Delta)\, , e\in E(\Delta)$, where $\Delta \subset X$ is any $\upomega$-open subset. 
 We require $D$ to be minimal in the sense that for any $i \in \intint{1}{n}$, $\nabla$ does not factor through \[E\otimes \Omega^1(\mathrm{log}(D-D^i))\hookrightarrow E\otimes \Omega^1_X(\logD).\]
Such a logarithmic connection $(E, \nabla)$ is called \emph{flat} if its curvature $\nabla^2$ 
is zero.

We are particularly interested in the case where $X$ is a smooth projective curve (a compact Riemann surface). Since then $X$ is of complex dimension one, any logarithmic connection over $X$ is automatically flat. Moreover, since then $X$ is projective, any analytic logarithmic connection over $X$ is isomorphic to the analytification of a unique algebraic logarithmic  connection over $X$ by one of Serre's GAGA theorems \cite[Prop. $18$]{GAGA}. \vspace{.3cm}

\textbf{Monodromy. } 
The notion of the monodromy representation of a flat connection varies slightly in the literature. For introductory and technical purposes, let us give the definition we are going to use. This definition can only be formulated in  the analytic category; in the algebraic case the monodromy representation is defined \textit{via} analytification. Let $X$ and $D$ be as above ($X$ has arbitrary dimension). Denote $X^0:=X\setminus D$.  Let $(E,\nabla)$ be an  analytic  logarithmic connection over $X$ with polar divisor $D$. Assume moreover that this analytic connection is flat, which is equivalent to it being integrable, \textit{i.e.}   $\mathcal{S}:=\ker (\nabla|_{X^0})$ is a locally constant sheaf of rank $r$ over $X^0$.   
Let $\Sigma$ and $Y\subset \Sigma$ be topological spaces such that there is a homeomorphism 
$$\Phi : (\Sigma,Y)\stackrel{\sim}{\to}  (X,D)\, .$$  
Fix such a homeomorphism and fix a point ${y}_0\in \Sigma\setminus Y$.  Denote ${x}_0:=\Phi({y}_0)$. For any path $\gamma : [0,1] \to \Sigma \setminus Y$, the pull back $(\Phi\circ \gamma)^*\mathcal{S}$ 
is locally constant and thus isomorphic to a constant sheaf. Hence $\gamma$ defines an isomorphism $\gamma(\mathcal{S}) : \mathcal{S}_{\gamma(1)}\to \mathcal{S}_{\gamma(0)}$. This isomorphism is invariant by homotopy relative to $\{\gamma(0)\, , \gamma(1)\}$ and satisfies $\gamma_1\, .\, \gamma_2(\mathcal{S})=\gamma_1(\mathcal{S})\circ \gamma_2(\mathcal{S})$ for any pair of paths $(\gamma_1,\gamma_2)$. We obtain a representation $\pi_1(\Sigma\setminus Y,  {y}_0) \to \mathrm{GL}(\mathcal{S}_{ {x}_0})$. \textit{Via} an  isomorphism $\mathcal{S}_{x_0}\to \C^r$, one deduces a (non-canonical) representation
$\rho_{\nabla} \in \Hom(\pi_1(\Sigma,  {y}_0), \GL) $ and
a canonical conjugacy class of representation $$[\rho_{\nabla}]\in  \bigslant{\Hom(\pi_1(\Sigma\setminus Y, {y}_0), \GL)}{\GL}\, . $$ We refer to  $\rho_{\nabla}$ as the \emph{monodromy representation} and to  $[\rho_{\nabla}]$ as the \emph{monodromy} of $(E, \nabla)$ with respect to $\Phi$. Conversely, given $\Phi$, given a conjugacy class of representation $[\rho]\in  \Hom(\pi_1(\Sigma\setminus Y,  {y}_0), \GL)/{\GL}$ and a compatible choice of \emph{mild transversal models} (see Section \ref{Sec mild trans}), there is a flat  logarithmic analytic connection $(E, \nabla)$ over $X$, unique up to isomorphism, inducing these transversal models and such that $[\rho_\nabla]=[\rho]$ (see \cite[Th. $3.3$]{cousinisom}, adapted from  \cite[Prop. $5.4$]{MR0417174}). In our work, the use of the marking $\Phi$ is essential, as we wish to compare the monodromies of connections over various homeomorphic curves.
\vspace{.3cm}

\textbf{Isomonodromy. } 
 Let $C$ be a smooth projective curve of genus $g$, let $D
 $ be a reduced divisor of degree $n$ on $C$. Let $(\initialBundle, \nabla_0)$ be a logarithmic connection over $C$ with polar divisor~$D$. 
 
 A $\upomega$-\emph{isomonodromic deformation of}  $(C, \initialBundle,\nabla_0)$ 
 consists in the following data:  
\begin{itemize}
\item a $\upomega$-family $(\kappa : \familyCurve\to T, \mathcal{D})$ of $n$-pointed smooth curves of genus $g$ (see Section \ref{SecSetup}); 
  \item a flat logarithmic $\upomega$-connection $(\familyBundle,\nabla)$ over $\familyCurve$ with polar divisor $\mathcal{D}$;
  \item a point $t_0$  in $T$ ; we denote $\familyCurve_{t_0}:=\kappa^{-1}(\{t_0\})$ and $\mathcal{D}_{t_0}:=\mathcal{D}|_{\familyCurve_{t_0}}$; and
\item  an isomorphism of pointed curves  with logarithmic connections
 $$(\psi,\Psi):   ((C, D),(\initialBundle,\nabla_0))\stackrel{\sim}{\to}((\familyCurve_{t_0}, \mathcal{D}_{{t_0}}),(\familyBundle,\nabla)|_{\familyCurve_{t_0}})\, .$$
  \end{itemize}

Why are such deformations called isomonodromic? Again we have to work in   the analytic category. Up to shrinking $T$ to a sufficiently small polydisc $\Delta$ containing $t_0$, the family $\kappa : (\familyCurve,\mathcal{D})\to \Delta$ is topologically trivial. Hence there is a homeomorphism
$$\Phi : (\Sigma_g,Y^n)\times \Delta \stackrel{\sim}{\to} (\familyCurve,\mathcal{D})$$ commuting with the natural projections to $\Delta$.
Now for any $t\in \Delta$, the morphism $$\pi_1(\Sigma_g\setminus Y^n, y_0) \longrightarrow \pi_1( (\Sigma_g\setminus Y^n)\times \Delta , (y_0,t) )\, , $$ induced by the inclusion of the fiber at $t$,  is an isomorphism. On the other hand, $(\familyBundle,\nabla)|_{\familyCurve_{t}}$ is a logarithmic connection over $\familyCurve_t$ with polar divisor $\mathcal{D}_{t}$. By flatness of $\nabla$, its monodromy representation with respect to $\Phi|_{t}$ and the base point $y_0$ can be identified with the monodromy representation of $(\familyBundle,\nabla)$ with respect to $\Phi$ and the base point $(y_0,t)$. For $t=t_0$, this means we can identify the monodromy representation of $(\familyBundle,\nabla)$ over $\familyCurve$ with respect to $\Phi$ with the monodromy representation of $(\initialBundle,\nabla_0)$ over $C$ with respect to $\psi^{-1}\circ \Phi|_{t_0}$. 
In that sense, we may say that with respect to some continuous ``base point section'' $t\mapsto (y_0,t)$, the monodromy representation along a germ of isomonodromic deformation is constant and given by the monodromy representation of $(\initialBundle,\nabla_0)$. More generally, one can say that an isomonodromic deformation induces a topologically locally trivial family of monodromy representations, leading to a phenomenon of monodromy of the monodromy representation. The latter will become  tangible in Section \ref{SecMonMon}. \vspace{.3cm}

\textbf{Statement of Theorem  \ref{algebrization thm}. }
Following \cite[Th. $3.4$]{MR2667785}(see also \cite{Malgrange, Krichever}), any triple $(C, \initialBundle,\nabla_0)$ as before admits a universal analytic isomonodromic deformation, which is unique up to unique isomorphism, and whose parameter space $T$ is the Teichm\"uller space $\T_{g,n}$. This universal analytic isomonodromic deformation satisfies a universal property with respect to germs of analytic isomonodromic deformations of $(C, \initialBundle,\nabla_0)$
.
 A universal algebraic isomonodromic deformation of  $(C, \initialBundle,\nabla_0)$, if it exists, would be  an algebraic isomonodromic deformation whose analytic germification is isomorphic to the germification of the universal analytic isomonodromic deformation of  $(C, \initialBundle,\nabla_0)$. In Section \ref{SecDefDef}, we give an alternative definition and state a \hyperlink{propunivalg}{universal property of universal algebraic isomonodromic deformations}. Our main result is the following.    
\begingroup
\setcounter{bigthm}{0} 
\begin{bigthm} \label{algebrization thm} Let $C$ be a smooth complex projective curve of genus $g$. Let $D$ be a set of $n$ distinct points in $C$ and let $\Phi : (\Sigma_g , Y^n)\to(C, D)$ be an orientation preserving homeomorphism.
Let $(\initialBundle,\nabla_0)$ be an algebraic logarithmic connection of rank $r$ over $C$ with polar divisor $D$ and monodromy $[\rho]\in \chign{\GL}$ with respect to $\Phi$.  
 Assume that $2g-2+n>0$ and that $\nabla_0$ is mild. If $r>2$, then assume further that $\rho$ is semisimple. The following are equivalent:
\begin{enumerate}
\item \label{algitem alg}There is a universal algebraic isomonodromic deformation of $(C, \initialBundle, \nabla_0)$.\vspace{.1cm}
\item \label{algitem finite}The orbit $\MCG \cdot [\rho]$ in $\chign{\GL}$ is finite.
\end{enumerate}
\end{bigthm} \endgroup
\begin{rem} 
Note that the orbit $\MCG \cdot [\rho]$ in $\chign{\GL}$ does not depend on the choice of $\Phi$. 
 \end{rem}
\begin{rem}
As we will see in Section~\ref{Sec mild trans}, the mildness assumption on $\nabla_0$ can always be achieved after a suitable birational gauge transformation. In this sense, under the mentioned restriction on the monodromy, Theorem~\ref{algebrization thm}  asserts that any finite branching universal isomonodromic deformation can be parametrized by algebraic functions. 
\end{rem}

We prove this theorem by adapting the proof for the special case of genus $g=0$, which has been established in \cite{cousinisom}.
 The main ingredients of the proof of Theorem \ref{algebrization thm} are: the logarithmic Riemann-Hilbert correspondence (see Section \ref{Sec mild trans}); the introduction of a base point section for a family of punctured curves and the splitting of the fundamental group of the total space of the family (see Section \ref{SecSplitting}), together with its relation to the mapping class group (see Section \ref{SecNotBul}).  
Both implications to be proven appear as special cases of stronger results: \hyperlink{thA1}{Theorem A1} and \hyperlink{thA2}{Theorem A2}, respectively. 
We give their statements and proofs in Section \ref{SecFinPreuve}. 

The statement of Theorem \ref{algebrization thm} is natural in the following sense. As we recall in Section  \ref{SecModuli},  the (algebraic) moduli space $\Mgn$ of stable smooth $n$-pointed genus $g$ curves is the quotient of the (analytic) Teichm\"uller space $\Tgn$ by the natural action of $\MCG$ .  Intuitively, a universal algebraic isomonodromic deformation should be the quotient of the universal analytic isomonodromic deformation with respect to a sufficiently large subgroup of $\MCG$ that fixes $[\rho]$.

  \setcounter{subsection}{1} 

\subsection{Dynamical study of finite orbits in the reducible rank $2$ case}\label{SecIntroDyn}  Since the pure mapping class group is a  finite index subgroup of the full mapping class group,  for any representation $\rho \in \Hom(\gf, G)$, the conjugacy class  $[\rho]\in \chi_{g,n}(G)$ has finite orbit under $\MCG$ if and only if it has finite orbit under $\fMCG$. Note that the size of $\fMCG\cdot[\rho]$ equals the size of the set of conjugacy classes of $m$-tuples
$$\bigslant{{\left\{ (\rho'(s_1), \ldots, \rho'(s_m))~\middle|~\rho'\in \Hom(\gf, G)\textrm{ and }{[}\rho'{]}\in \fMCG\cdot[\rho] \right\}}}{G} \, , $$ where $\{s_1, \ldots, s_m\}$ is a set of generators of $\gf$. 
We introduce a specific presentation 
 $$\gf= \left\langle\alpha_1, \beta_1, \ldots, \alpha_g, \beta_g,\gamma_1,\ldots,\gamma_n ~\vert~ [\alpha_1, \beta_1]\cdots[\alpha_g, \beta_g]\gamma_1\cdots\gamma_n=1\right\rangle\, $$ and a subgroup 
$$\liftMCG =\left\langle  \tau_{1}, \ldots, \tau_{3g+n-2}, \sigma_1, \ldots, \sigma_{n-1}\right\rangle\,  $$   of $\fMCG^\bullet$ which, as such, acts on $\Hom(\gf, G)$, and which is sufficiently large in the sense that the $\liftMCG$-orbit of $[\rho]\in \chi_{g,n}(G)$ equals its $\fMCG$-orbit.  Moreover, the  action of 
     $\liftMCG $ on $\gf$ can be explicitely described (see Section \ref{Sec explicit}). Table \ref{A} summarizes  the explicit action of the generators $ \tau_{1}, \ldots, \tau_{3g+n-2}, \sigma_1, \ldots, \sigma_{n-1}$ of $\liftMCG$ on the generators $ \alpha_1, \beta_1, \ldots,$ $ \alpha_g, \beta_g,\gamma_1,\ldots,\gamma_n$ of $\gf$. Here we only indicate the action on those of our generators of  $\gf$  that are not fixed by the action of the generator of $\liftMCG$ under consideration.

 \begin{table}[htbp]
$\begin{array}{|ll| lcl l|}
\hline
\tau_{2k}\,, &k\in\intint{1}{g} & \alpha_k & \mapsto& \alpha_k \beta_k&\\
\hline
\tau_{2k-1}\,, & k\in\intint{1}{g} & \beta_k & \mapsto& \beta_k\alpha_k&\\
\hline
\tau_{2g+k}\,,& k\in \intint{1}{g-1}&\alpha_{k+1}&\mapsto& \Theta_k^{-1} \alpha_{k+1}&  \\
&&  \alpha_{k} &\mapsto& \alpha_{k}\Theta_k&\\
&&\beta_{k}&\mapsto& \Theta_k^{-1}\beta_{k}\Theta_k& \\
&&\textrm{where}&\Theta_k=&\alpha_{k+1}\beta_{k+1}^{-1}\alpha_{k+1}^{-1}\beta_{k}&\\
\hline
\tau_{3g-1+k}\,,  & k\in\intint{1}{n-1}   & \alpha_g & \mapsto& \alpha_g \Xi_k& \\
&&\beta_g & \mapsto& \Xi_k^{-1}\beta_g\Xi_k&\\
&&\gamma_i & \mapsto& \Xi_k^{-1} \gamma_i \Xi_k\,,&i\in\intint{1}{k} \\
&   &\textrm{where}& \Xi_k =& (\gamma_1\ldots \gamma_k)^{-1}\beta_g&\\
\hline
\sigma_{k}\,, &k\in\intint{1}{n-1}  &\gamma_k & \mapsto& \gamma_k\gamma_{k+1}\gamma_k^{-1}& \\
&&\gamma_{k+1}&\mapsto& \gamma_k&\\  \hline
\end{array}
$\\$ $
\caption{Action of $\liftMCG$ on $\gf$. \label{A}}
\end{table}

  We then apply this explicit description of the mapping class group action to the specific study of finite $\MCG$-orbits on $\chi_{g,n}(\GLtw)$  that correspond to reducible representations. 
  For $g=0$, this study has been completely carried out in \cite{cousin2016finite}. 
 In this special case, the study can be reduced to linear dynamics. To explain this, recall that any reducible representation $\rho \in \mathrm{Hom}(\gf, \GLtw)$ is conjugated to the tensor product of a character  $\rho_{\C^*}\in \Hom(\gf, \C^*)$ and an affine representation $\rho_{\mathrm{Aff}}\in \Hom(\gf, \mathrm{Aff}(\mathbb{C}))$:  $$[\rho]=[ \rho_{\C^*}\otimes \rho_{\mathrm{Aff}}]\, .$$ Then, $[\rho]$ has finite orbit under $\MCG$ in $\chign{\GLtw}$ if and only if $[\rho_{\C^*}]$ and $[\rho_{\mathrm{Aff}}]$ have finite orbit under $\MCG$ in $\chign{\C^*}$ and $\chign{\mathrm{Aff}(\C)}$ respectively. For $g=0$, the pure mapping class group acts trivially on $\chign{\C^*}$ and on the linear part of $\rho_{\mathrm{Aff}}$. Hence in the special case $g=0$, the study of finite orbits reduces to the study of a certain linear action on the translation part of $\rho_{\mathrm{Aff}}$.  
 
 For $g>0$, the study of finite orbits of conjugacy classes of reducible $\GLtw$-representations also reduces to the case of scalar and affine representations, but the linear part of  $\rho_{\mathrm{Aff}}$ is no longer invariant and there is no effective means to reduce the study to linear dynamics. However, Table~\ref{A} allows to study the orbits explicitely.  In the case $g=1$ and $n>0$, we find a particular type of representations whose conjugacy classes have finite orbit under $\MCG$, namely the representations $\rho_{\mu, \mathbf{c}}  \in \Hom (\gf ,\GLtw)$ defined by $$\rho_{\mu, \mathbf{c}} (\alpha_1):=\begin{pmatrix} \mu  & 0 \\ 0 & 1\end{pmatrix} \quad  \rho_{\mu, \mathbf{c}}(\beta_1):=\begin{pmatrix} 1 & -\frac{1}{\mu-1} \\ 0 & 1\end{pmatrix} \quad  \rho_{\mu, \mathbf{c}}(\gamma_i):=\begin{pmatrix} 1  & c_i \\ 0 & 1\end{pmatrix} \quad \forall i\in \intint{1}{n}\,  $$ where $\mu\in \C^*\setminus\{1\}$ is a root of unity and $\mathbf{c}=(c_1, \ldots, c_{n})\in \C^{n}$ with $\sum_{i=1}^{n}c_i=1$. Note that the condition $\sum_{i=1}^{n}c_i=1$ is necessary for $\rho_{\mu, \mathbf{c}}$ to be well-defined.  
The complete classification,  for every $g>0$ and $n\geq 0$, of reducible rank-$2$ representations with finite $\MCG$-orbit   is the following.

 \begingroup
\setcounter{bigthm}{1} 
\begin{bigthm}\label{mainthm dynamics} Let $g>0, n\geq 0$.
Let $\rho \in \Hom (\gf ,\GLtw)$ be a reducible representation. Consider its conjugacy class $[\rho] \in \chi_{g,n}(\GLtw)$. Then the orbit $\MCG \cdot [\rho]$ is finite if and only if one of the following conditions is satisfied.\vspace{.2cm}
\begin{enumerate}
\item \label{case1}The representation $\rho$ is a direct sum of scalar representations with finite images.\vspace{.2cm}
\item \label{case2}We have $g=1$, $n>0$,  there are a root of unity $\mu\in \C^*\setminus\{1\}$ , $\mathbf{c}=(c_1, \ldots, c_{n})\in \C^{n}$ with $\sum_{i=1}^{n}c_i=1$ and a scalar representation $\lambda$ with finite image such that $$[\rho] \in \MCG\cdot [\lambda \otimes \rho_{\mu, \mathbf{c}}]\, .\vspace{.2cm}$$ 
 \end{enumerate}
Moreover, if the orbit $\MCG\cdot[\rho]$ is finite, we can give an estimate for its cardinality, which for $\rho=\lambda_1\oplus\lambda_2$ and $\rho=\lambda \otimes \rho_{\mu, \mathbf{c}}$ in the cases $(\ref{case1})$ and $(\ref{case2})$ respectively is  \vspace{.2cm}
\begin{enumerate}
\item  $\hspace{.18cm} \frac{1}{2}\cdot \max\left\{\cardn (\mathrm{Im}(\lambda_i))^{2g-1}~\middle|~i\in \{1,2\}\right\}~\leq~\cardn (\MCG \cdot [\rho])~\leq~\cardn (\mathrm{Im}(\rho))^{2g}$   and\vspace{.2cm}
\item  $\hspace{.45cm}\max\left\{N_2\, , \, \phi(N)(2N-\phi(N))N^{n'-1}\right\}~\leq~\cardn (\MCG\, \cdot\,  [\rho])~\leq~(N^2-1)N^{n'-1}N_2^{2}\, ,$\vspace{.2cm}\\
where  $\phi$ denotes the Euler totient function,  $n':=\cardn \left\{ i\in\intint{1}{n}~|~\rho(\gamma_i)\not \in \C^*I_2\right\}$, $N:=\mathrm{order}(\mu)$ and $N_2:=\cardn (\mathrm{Im}(\lambda))$.  \end{enumerate}
 \end{bigthm}
 \endgroup

The heart of the proof of Theorem \ref{mainthm dynamics} is the complete classification of finite $\fMCG$-orbits in $\chign{\Aff}$ under the full mapping class group (see the beginning of Section \ref{Sec Aff case} for details on how we proceed). In Section \ref{SecCalcul}, we deduce an explicit description of the finite $\MCG$-orbits for scalar and affine representations. The decomposition of reducible representations into a tensor product of such representations then yields the result (see Sections  \ref{SecReduc} and  \ref{Sec Proof A}).

  During the evaluation process of the present work, a classification complementary to our Theorem \ref{mainthm dynamics} appeared in the preprint \cite{2017arXiv170700071B}; it concerns  finite $\MCG$-orbits of \emph{irreducible} representations in
$\chign{\mathrm{SL}_2\mathbb{C}}$
for $g>0$ and asserts that, in that case as well, finite orbits correspond to finite representations if $g>1$ and a special class of infinite representations with finite orbits appears in the genus $1$ case.

 \renewcommand\thesubsection{\thesection.\arabic{subsection}}

 \part{Algebraization}\label{partAlg}

 \section{Universal isomonodromic deformations}
 In this section, we will recall some well known results about moduli spaces and universal families of curves. For a more detailed exposition, see for example \cite[Chap. 15
 ]{MR2807457} and \cite[Chap. 6]{Hubbard}. Then we turn to the existence of analytic and algebraic universal isomonodromic deformations of connections over curves, and their respective universal properties. The main purpose of this section is the precise setup of our notation and definitions. The reader might want to skip this section at first and come back to it when needed (references will be given).
 
  \subsection{Moduli spaces of curves
  } \label{SecModuli}
  We define \emph{a  curve of genus $g$}   to  be a \emph{smooth projective complex} curve  $C$ with $H^1(C,\Z)=\Z^{2g}$.
  From now on, we will assume \begin{equation}\label{condgn} 2g-2+n>0\, .\end{equation}

As a set, the Teichm\"uller space $\Tgn$ of $n$-pointed genus $g$ curves is the set of isomorphism classes $[C,D,\varphi]$ of triples $(C,D,\varphi)$, 
where  $C$ is a genus $g$ curve, $D=\{x_1, \ldots, x_n\}$ is a set of $n$ distinct points in $C$ and $\varphi$ is a Teichm\"uller structure, \textit{i.e.}  an orientation-preserving homeomorphism $\varphi :(\Sigma_g ,Y^n)\to(C,D)$. 
Two $n$-pointed genus $g$ curves with Teichm\"uller structure $(C,D,\varphi)$ and $(C',D',\varphi')$ are said to be isomorphic 
  if there exists an isomorphism of pointed curves $\psi : (C',D')\to (C,D) $ such that $[\varphi]=[\psi \circ \varphi']$, where $[\varphi]$ denotes the isotopy class of $\varphi$.
We have a natural action of $\MCG$ on $\Tgn$ given by \[[h]\cdot [C,D,\varphi ]:= [C,D,   \varphi\circ h^{-1}]\, ; ~~~~~ [h] \in \MCG\, ,\,  [C,D,\varphi]\in \Tgn\, .\]

The kernel of this action is finite. More precisely, we have (see \cite[Prop $4.11$ p. $189$]{MR2807457}):

\begin{lem}\label{lemme Kgn}
If the natural morphism $\MCG\rightarrow \mathrm{Aut}(\Tgn)$ has nontrivial kernel $K_{g,n}$, then $K_{g,n}\simeq \Z/2\Z$ and one of the following holds.
\begin{itemize}
\item $(g,n)=(2,0)$ and the non-trivial element of $K_{g,n}$ is the hyperelliptic involution of $\Sigma_2$.\vspace{.1cm}
\item $(g,n)=(1,1)$ and the non-trivial  element of $K_{g,n}$ is the order $2$ symmetry about the puncture, given, for $(\Sigma_1,y_1)=(\C/ \Z^2,0)$, by $z\mapsto -z$.
\end{itemize}\end{lem}
   \vspace{.3cm}
 
As a set, the moduli space $\Mgn$ of curves of genus $g$ with $n$ (labeled) punctures is the set of isomorphism classes $[C,\mathbf{x}]$ of pairs $(C,\mathbf{x})$, where $C$ is a genus $g$ curve and $\mathbf{x}=(x_1,\ldots, x_n)$ is a tuple of $n$ distinct points  in $C$. The isomorphisms are isomorphisms of pointed curves that \emph{respect the labellings} of the $n$-tuples. Notice that a Teichm\"uller structure $(C,D,\varphi)$ defines such a pair $(C,\mathbf{x})$, by setting $\mathbf{x}:=(\varphi(y_i))_{i\in \intint{1}{n}}$.
In this way, we obtain  a forgetful map  \begin{equation} \forggn  : \Tgn\rightarrow \Mgn\,\end{equation}
 whose fibers are globally fixed by the action of $\MCG/K_{g,n}$ on $\Tgn$.
Denote by $ \Rgn \subset \Tgn$ the set consisting in points with non-trivial stabilizer for the action of $\MCG/K_{g,n}$. The subset $ \Bgn:=\forggn (\Rgn)$ of $\Mgn$ characterizes pointed curves with automorphism groups not isomorphic to $K_{g,n}$. We say that these curves have \emph{exceptional automorphisms}.  

Recall that  $\Tgn$ has a natural structure of a complex analytic manifold, and $\Mgn$ has a natural structure of a complex quasi-projective variety (see \cite[chap. XIV]{MR2807457}). The set 
$\Bgn$ of curves with exceptional automorphisms is a Zariski closed subset of $\Mgn$ (see \cite[Rem. $5.13$ p. $202$ and Th. $6.5$ p. $207$]{MR2807457}) which is a proper subset (see \cite{MR0142552,MR2611511,MR1748288,MR2346015}). Moreover, the map $\forggn |_{\Tgn\setminus \Rgn} :\Tgn\setminus \Rgn \rightarrow \Mgn \setminus \Bgn$  is a non-branched analytic cover, with Galois group $\MCG/K_{g,n}$. For any point $\hat{\star}\in \Tgn$ projecting to $\star\in \Mgn$ we obtain a tautological morphism \begin{equation}\label{def phi} \tautgn  : \pi_1(\Mgn \setminus \Bgn,\star)\twoheadrightarrow \bigslant{\MCG}{K_{g,n}};\end{equation} such that for the lift $\hat{\gamma}$ in $\Tgn$ with $\hat{\gamma}(0)=\hat{\star}$ of a loop $\gamma$ corresponding to an element of $\pi_1(\Mgn \setminus \Bgn,\star)$ we have  $\hat{\gamma}(1)=\tautgn (\gamma)\cdot \hat{\star}.$  For another point $\hat{\star}'=[h]\cdot \hat{\star}$ we obtain $\taut_{[h]\cdot \hat{\star}}=h\cdot \tautgn \cdot h^{-1}$.

 \subsection{Families of pointed curves}\label{SecSetup}
Let $C$ be a curve of genus $g$ and let $D$ be a reduced divisor of degree $n$ on $C$. 
Recall that we always assume \eqref{condgn}; \textit{i.e.} the pointed curve $(C,D)$ is stable. 

A \emph{$\upomega$-family of $n$-pointed genus $g$ curves with central fiber $(C,D)$} is a datum $$\mathcal{F}_{(C,D)}=(\kappa : \familyCurve\to T,\mathcal{D}, t_0, \psi)\, ,$$ where

\begin{itemize}
\item $\kappa : \familyCurve\to T$ is a proper surjective smooth morphism of $\upomega$-manifolds;\vspace{.1cm}
\item $\mathcal{D}=\sum_{i=1}^n\mathcal{D}^i$ is a reduced divisor on $\familyCurve$ such that \vspace{.1cm}
\item there are pairwise disjoint sections $\sigma_1, \ldots, \sigma_n$ of $\kappa$ with $\sigma_i(T)=\mathcal{D}^i$;\vspace{.1cm}
\item $t_0\in T$ is a point and   \vspace{.1cm}
\item $\psi :(C,D)  \stackrel{\sim}{\to}    (\familyCurve_{t_0}, \mathcal{D}_{{t_0}}) $ is an isomorphism of $\upomega$-manifolds.   \end{itemize}\vspace{.3cm}

Here and in the following, we denote by $\familyCurve_t:=\kappa^{-1}(\{t\})$ the fiber of $\kappa$ at a parameter $t\in T$, and we denote $\mathcal{D}_{{t}}:=\mathcal{D}|_{\familyCurve_t}$ accordingly. 
When there is a smooth connected $\upomega$-neighborhood $\Delta$ of $t_0$ such that $\mathcal{F}_{(C,D)}|_{\Delta}$ satisfies a certain property, we may say that $\mathcal{F}_{(C,D)}$ satisfies this property \emph{up to shrinking}. 

A morphism $\morph : \mathcal{F}'_{(C,D)}\to \mathcal{F}_{(C,D)}$  is a pair $\morph=(\morph^a, \morph^b)$, where $\morph^a : \familyCurve'\to \familyCurve$ and $\morph^b : T'\to T$  are morphisms of $\upomega$-varieties such that the following diagram commutes (and in particular $\morph^b(t_0')=t_0$). 
\[\xymatrix
   {(C,D)\ar@{=}[r]^{\mathrm{id}}\ar@{^{(}->}[d]_{\psi'}&(C,D)\ar@{_{(}->}[d]^{\psi}\\
  (\familyCurve',\mathcal{D}')\ar[d]_{\kappa'} \ar[r]^{\morph^a} &(\familyCurve,\mathcal{D}) \ar[d]^{\kappa}    \\
  T'\ar[r]^{\morph^{b}} 
& T  
}\]
\begin{rem} Note that this definition implies that $(\familyCurve',\mathcal{D}')$ is isomorphic to the pullback $\morph^{b *}(\familyCurve,\mathcal{D})$ (the fibered product with respect to $\morph^b$ and $\kappa$). 
\end{rem}

Suppose now that we have a Teichm\"uller structure for $(C,D)$, given by an orientation preserving homeomorphism $\varphi : (\Sigma_g,Y^n)\to (C,D)$.
A $\upomega$-family of $n$-pointed genus $g$ curves  with Teichm\"uller structure with central fiber $(C,D,\varphi)$ is a datum $\mathcal{F}_{(C,D,\varphi)}^+=(\mathcal{F}_{(C,D)}, \Phi)$, where $\mathcal{F}_{(C,D)}$ is as above and $\Phi : (\Sigma_g,Y^n)\times T \stackrel{\sim}{\to} (\familyCurve,\mathcal{D})$ is a homeomorphism such that the following diagram commutes, where $\mathrm{pr}$ denotes the projection to the second factor.
\[\xymatrix
   {   (\Sigma_g,Y^n)\times \{t_0\} \ar[r]^{\hskip10pt \varphi}\ar@{^{(}->}[d] & (C,D)\ar@{_{(}->}[d]^{\psi}\\
 (\Sigma_g,Y^n)\times T \ar[dr]_{\mathrm{pr}} \ar[r]^{\hskip10pt \Phi} &(\familyCurve,\mathcal{D}) \ar[d]^{\kappa}    \\
 & T
}\]

 In particular, if we denote $$\Phi_t:=\Phi|_{(\Sigma_g,Y^n)\times \{t\}}\, , $$  then $\Phi_{t_0}=\psi\circ\varphi$.  Notice that by definition, a $\upomega$-family  with Teichm\"uller structure is topologically trivial.   

For a given $\varphi$ as above, up to shrinking, any analytic family $\mathcal{F}_{(C,D)}$  lifts to a family $\mathcal{F}_{(C,D, \varphi)}^{+}$ with Teichm\"uller structure. 

Let $\mathcal{F}_{(C,D, \varphi)}^+=(\mathcal{F}_{(C,D)}, \Phi)$ and ${\mathcal{F}'}_{(C,D, \varphi')}^{+}=(\mathcal{F}_{(C,D)}', \Phi')$ be two $\upomega$-families with Teichm\"uller structures. A morphism ${\mathcal{F}'}_{(C,D, \varphi)}^+\to \mathcal{F}_{(C,D, \varphi)}^+$  is a datum $\morph=(\morph^a, \morph^b)$ is as before, such that moreover, if we denote by $\morph^{\mathrm{top}}$  the continuous map defined by the diagram  
 \[\xymatrix
   {
 (\Sigma_g,Y^n)\times T'\ar[d]_{\Phi'}\ar[r]^{\morph^{\mathrm{top}}}&  (\Sigma_g,Y^n)\times T  \\
  (\familyCurve',\mathcal{D}')  \ar[r]^{\morph^a} &(\familyCurve,\mathcal{D}) \ar[u]^{\Phi^{-1}}\, ,      \\
}\]
then $\morph^{\mathrm{top}}$ is a fiberwise isotopy.

To a $\upomega$-family $\mathcal{F}_{(C,D)}$ [resp. $\upomega$-family with Teichm\"uller structure $\mathcal{F}_{(C,D, \varphi)}^+$] as before, one can associate a so-called  {$\upomega$-family $\mathcal{F}$} [resp. {$\upomega$-family with Teichm\"uller structure $\mathcal{F}^+$}] \emph{with non-specified central fiber}, by forgetting $(C,D)$ [resp. $(C,D,\varphi)$] and the marking $t_0, \psi$. 
 A morphism of 
$\upomega$-families with non-specified central fiber is a datum $\morph$ as above for a convenient choice of  marked central fibers.

\subsection{Universal families of pointed curves}\label{Sec univ fam}

 Let $\mathcal{F}^+_{(C,D, \varphi)}=(\mathcal{F}_{(C,D)}, \Phi)$ be a $\upomega$-family with Teichm\"uller structure. Then the classifying map
  $$ \classan(\mathcal{F}^+): \left\{ \begin{array}{ccc} T &\to &\Tgn  \\ t & \mapsto & [\familyCurve_t,\mathcal{D}_{{t}},\Phi_t] \end{array}\right.\,  $$
  is holomorphic with respect to the natural analytic manifold structure of  $\Tgn$. The Teichm\"uller space $\Tgn$ carries a universal family (see for example \cite[Chap. 6]{Hubbard}), which is an analytic family with Teichm\"uller structure
$\mathcal{F}_{g,n}^+=(\mathcal{F}_{g,n}, \Phi_{g,n})$ and non-specified central fiber, satisfying
$$\classan(\mathcal{F}_{g,n}^+)=\mathrm{id}_{\Tgn}\, .$$ The universal Teichm\"uller curve enjoys the following universal property:
If $\mathcal{F}^+=(\mathcal{F}, \Phi)$ is an analytic family with Teichm\"uller structure and non-specified central fiber, then there is a unique isomorphism
$$\morph^+ : \mathcal{F}^+ \stackrel{\sim}{\longrightarrow} \classan(\mathcal{F}^+)^*(\mathcal{F}_{g,n})$$ with $\morph^b=\mathrm{id}_T$.

  Let $\mathcal{F}_{(C,D)}$ 
  be a $\upomega$-family. Assume we have a \emph{labelling} $\mathbf{x}$ of $D$, \textit{i.e.} $\mathbf{x}=(x_i)_{i\in \intint{1}{n}}\in C^n$ and $D =\sum_{i=1}^nx_i$.  Then there is a well-defined labelling $\boldsymbol{\mathcal{D}}=(\mathcal{D}^i)_{i\in \intint{1}{n}}$ of $\mathcal{D}$ defined by $\mathcal{D}=\sum_{i=1}^n \mathcal{D}^i$ and $\psi(x_i)\in \mathcal{D}^i$ for all $i\in \intint{1}{n}$. 
  We then have a well-defined 
classifying map
   $$\class(\mathcal{F}): \left\{ \begin{array}{ccc} T &\to &\Mgn  \\ t & \mapsto & [\familyCurve_t,\boldsymbol{\mathcal{D}}_{{t}}] \end{array}\right. \, , $$
which is a morphism of $\upomega$-varieties with respect to the natural structure of $\upomega$-variety on $\Mgn$.  We say that the fiber $(\familyCurve_t, \mathcal{D}_{t})$ of $\mathcal{F}$ at $t\in T$ has \emph{exceptional automorphisms} if $\class(\mathcal{F})(t)\in \Bgn$. This notion does not depend on the choice of a labelling.
 
Although there is no universal family of curves over $\Mgn$ in the strict sense,  we can consider algebraic Kuranishi families. Let   $\mathcal{F}$ be a $\upomega$-family and let $t\in T$ be a parameter. Denote $\mathcal{F}|_{\Delta^{\mathrm{an}}}$ the analytic germification of $\mathcal{F}$ at $t$, which can be endowed with a Teichm\"uller structure $\Phi_{\Delta^{\mathrm{an}}}$. We say that $\mathcal{F}$ is \emph{Kuranishi at $t$} if $\classan(\mathcal{F}|_{\Delta^{\mathrm{an}}},\Phi_{\Delta^{\mathrm{an}}})$ is a local isomorphism. The notion of being Kuranishi at $t$ does not depend on the choice of $\Phi_{\Delta^{\mathrm{an}}}$. 
We say that $\mathcal{F}$ is \emph{Kuranishi} if it is Kuranishi at each $t\in T$. Notice that if $\mathcal{F}^{\mathrm{Kur}}$ is an algebraic Kuranishi family, then for any labelling, the classifying map $\class(\mathcal{F}^{\mathrm{Kur}})$ is dominant and has finite fibers. 

For any stable $n$-pointed genus $g$ curve $(C,D)$, there exists an algebraic Kuranishi family $\mathcal{F}^{\mathrm{Kur}}_{(C,D)}
$ with central fiber $(C,D)$
. 
Moreover, we have (see \cite[Rem. $6.9$, p. $208$]{MR2807457}):  
 \begin{prop}[Universal property of Kuranishi families] \label{kuranishi alg} \label{prop univ kuranishi alg} Let $(C,D)$ and $\mathcal{F}^{\mathrm{Kur}}_{(C,D)}$ be as above.
 Let $\mathcal{F}'_{(C,D)}$ be an algebraic family with central fiber $(C,D)$. Then there are
 \begin{itemize}
  \item an \'etale base change $p:(T'',t_0'')\rightarrow (T',t_0')$; denote $\mathcal{F}''_{(C,D)}:=p^*\mathcal{F}'_{(C,D)}$;\vspace{.1cm}
  \item a morphism $q : (T'',t_0'')\to (T,t_0)$ and \vspace{.1cm}
  \item an isomorphism  $\morph: \mathcal{F}''_{(C,D)}\stackrel{\sim}{\longrightarrow} q^*\mathcal{F}^{\mathrm{Kur}}_{(C,D)}$ with $\morph^b=id_{\Delta''}$. 
\end{itemize}
  \end{prop}
  \subsection{Universal isomonodromic deformations}\label{SecDefDef}
   Let again $(C,D)$ be a stable $n$-pointed genus $g$ curve.   Let $(\initialBundle, \nabla_0)$ be logarithmic $\upomega$-connection over $C$ with polar divisor $D$.
   \vspace{.3cm}
    
   \textbf{Isomonodromic deformations. }
   A  $\upomega$-\emph{isomonodromic deformation} of $(C, \initialBundle, \nabla_0)$ is a tuple $\mathcal{I}_{(C, \initialBundle, \nabla_0)}=(\mathcal{F}_{(C, D)}, \familyBundle, \nabla , \Psi)$, where
  \begin{itemize} 
  \item  $\mathcal{F}_{(C,D)}=(\kappa : \familyCurve\to T,\mathcal{D}, t_0, \psi)$ is a $\upomega$-family with central fiber $(C,D)$,\vspace{.1cm}
  \item $(\familyBundle,\nabla)$ is a flat logarithmic $\upomega$-connection over $\familyCurve$ with polar divisor $\mathcal{D}$ and\vspace{.1cm}
  \item $(\psi,\Psi) : (\initialBundle\to C, \nabla_0)\to (\familyBundle\to \familyCurve, \nabla)|_{\familyCurve_{t_0}}$ is an isomorphism of $\upomega$-logarithmic connections, \textit{i.e.} 
  $\Psi : \initialBundle\to \familyBundle|_{\familyCurve_{t_0}}$ is a $\upomega$-vector bundle isomorphism over $\psi : C\to \familyCurve_{t_0}$
  satisfying $\Psi^*\left(\nabla|_{\familyCurve_{t_0}}\right)=\nabla_0$. 
  \end{itemize} 
  
   Let $\mathcal{I}_{(C, \initialBundle, \nabla_0)}$ and $\mathcal{I}_{(C, \initialBundle, \nabla_0)}'$ be two $\upomega$-isomonodromic deformations of $(C, \initialBundle, \nabla_0)$. A morphism $\morph^{}: \mathcal{I}_{(C, \initialBundle, \nabla_0)}'\to \mathcal{I}_{(C, \initialBundle, \nabla_0)}$ is a datum $\morph^{ }=(\morph^a, \morph^b, \morph^{\mathrm{vb}})$, where $(\morph^a, \morph^b)$ is a morphism $\mathcal{F}_{(C,D)}'\to \mathcal{F}_{(C,D)}$ as in Section \ref{SecSetup}, and $ \morph^{\mathrm{vb}}$ is a morphism of $\upomega$-vector bundles over $\morph^a$ with ${\nabla'=\morph^{vb}}^*\nabla$.  
 
 An algebraic isomonodromic deformation $\mathcal{I}_{(C, \initialBundle, \nabla_0)}$ of $(C, \initialBundle, \nabla_0)$ as above is called \emph{regular} if moreover $(\familyBundle, \nabla)$ is regular (with respect to a suitable meromorphic structure at infinity). The definition of regularity can be found in  \cite[Th. $4.1$]{MR0417174}. Putting this regularity condition on $\mathcal{I}_{(C, \initialBundle, \nabla_0)}$ may be seen  as a way of standardizing algebraic isomonodromic deformations, as illustrated by the following statement.  

 \begin{lem}\label{LemStandardiser} If $(\initialBundle, \nabla_0)$ is mild and $\mathcal{I}_{(C, \initialBundle, \nabla_0)}$  is an algebraic  isomonodromic deformation of $(C, \initialBundle, \nabla_0)$, then the analytification of $\mathcal{I}_{(C, \initialBundle, \nabla_0)}$ is isomorphic to the analytification of a regular algebraic isomonodromic deformation $\mathcal{I}'_{(C, \initialBundle, \nabla_0)}$ of $(C, \initialBundle, \nabla_0)$. \end{lem}
   
   This lemma will be proven in Section \ref{Sec mild trans}, where we will also recall the notion of mildness, which is a minor technical condition. 
 \vspace{.3cm}
    
   \textbf{Analytic universal isomonodromic deformations. } Let $\varphi :(\Sigma_g,Y^n) \stackrel{\sim}{\to} (C,D)$ be an orientation preserving homeomorphism. 
  Consider the universal Teichm\"uller family $\mathcal{F}_{g,n}^+=(\mathcal{F}_{g,n}, \Phi_{g,n})$. We shall denote$$  \mathcal{F}_{g,n}=(\kappa_{g,n} : \mathcal{\familyCurve}\to \Tgn, \mathcal{D}) \, ;  \quad \Phi_{g,n} :(\Sigma_g,Y^n)\times \Tgn \stackrel{\sim}{\to} (\familyCurve, \mathcal{D})\, ; \quad t_0:=[C,D, \varphi]\in \Tgn 
  \, .$$ 
  By the definition of $\mathcal{F}_{g,n}^+$, we then have an isomorphism $\psi : (C,D)\stackrel{\sim}{\to} ({\familyCurve_{t_0}} , \mathcal{D}_{{t_0}})$.
  In particular,
  $\mathcal{F}^{\mathrm{Teich}}_{(C, D)}:=( \mathcal{F}_{g,n} , t_0, \psi)$ is an analytic family with central fiber $(C,D)$, which moreover is topologically trivial and has simply connected parameter space.   The inclusion $\Phi_{g,n}^{-1}\circ \psi\circ \varphi$ of the topological fiber at $t_0$ then defines an isomorphism 
   \begin{equation}\label{Teichmincl}\gf=\pi_1(\Sigma_g\setminus Y^n, y_0) \stackrel{\sim}{\to} \pi_1((\Sigma_g\setminus Y^n)\times \Tgn, (y_0,t_0))\, .\end{equation}
 Now let  $[\rho_{\nabla_0}]$  be the monodromy of $(\initialBundle, \nabla_0)$ with respect to $\varphi$.
 The representation $\rho_{\nabla_0}$ can then be trivially extended to a representation $\rho$ of $\pi_1((\Sigma_g\setminus Y^n)\times \Tgn, (y_0,t_0))$.
 It turns out that the conjugacy class of this ``extended representation'' is the monodromy representation, with respect to $\Phi_{g,n}$, of a certain flat logarithmic connection $(E, \nabla)$  over $ \mathcal{X}$ with polar divisor $\mathcal{D}$ such that the pullback $\psi^*(E,\nabla)$ restricted to $\mathcal{X}_{t_0}$ is canonically isomorphic to $(\initialBundle, \nabla_0)$. 
We obtain the \emph{universal analytic isomonodromic deformation} 
$$\mathcal{I}^{\mathrm{univ,\, an}}_{(C, \initialBundle, \nabla_0)}:=(\mathcal{F}^{\mathrm{Teich}}_{(C, D)}, \familyBundle, \nabla , \Psi^\mathrm{can})\, .$$
Its construction has been carried out in \cite{MR2667785}, using Malgrange's Lemma (see  \cite{Malgrange}) and the fact that $\Tgn$ is contractible by Fricke's Theorem. 
It satisfies the following universal property: if $\mathcal{I}'_{(C, \initialBundle, \nabla_0)}=(\mathcal{F}_{(C, D)}', \familyBundle', \nabla' , \Psi')$ is an analytic isomonodromic deformation of $(C,\initialBundle,\nabla_0)$, and if $\Delta'$ is a sufficiently small neighborhood of its central parameter $t_0'$, then there is a morphism $q : (\Delta',t_0')\to (\Tgn, t_0)$
and a canonical isomorphism $$\mathcal{I}'_{(C, \initialBundle, \nabla_0)}|_{\Delta'} \simeq q^{ *}\mathcal{I}_{(C, \initialBundle, \nabla_0)}^{\mathrm{univ,\, an}}\, .$$
The construction of this analytic universal isomonodromic deformation and the proof of its universal property rely on the fact that up to shrinking the parameter space, analytic families of curves are topologically trivial and have simply connected parameter space. This is of course no longer the case in the algebraic category, the extension of the monodromy representation of $(C, \initialBundle, \nabla_0)$ being the main challenge
.  
 \vspace{.3cm}
    
   \textbf{Algebraic universal isomonodromic deformations. }
   An  \emph{algebraic universal isomonodromic deformation} of $(C, \initialBundle, \nabla_0)$ is an algebraic isomonodromic deformation $\mathcal{I}^{\mathrm{univ, alg}}_{(C, \initialBundle, \nabla_0)}=(\mathcal{F}_{(C, D)}, \familyBundle, \nabla , \Psi)$, where
   $\mathcal{F}_{(C, D)}=\mathcal{F}^{\mathrm{Kur}}_{(C, D)}$ is an algebraic Kuranishi family with central fiber $(C,D)$. Note that an algebraic universal  isomonodromic deformation of $(C, \initialBundle, \nabla_0)$ does not need to exist; its existence is precisely the subject of Theorem \ref{algebrization thm}. When it does exist, it satisfies the following universal property, which will be proven in Section \ref{Sec proof UP}. 
        
   \begin{prop}[Universal property of universal algebraic isomonodromic deformations]\label{propuniviso}\hypertarget{propunivalg}{} Let $(C, \initialBundle, \nabla_0)$ and $\mathcal{I}^{\mathrm{univ, alg}}_{(C, \initialBundle, \nabla_0)}$ be as above. Let $\mathcal{I}'_{(C, \initialBundle, \nabla_0)}$ be another algebraic isomonodromic deformation of $(C, \initialBundle, \nabla_0)$. Assume that 
    \begin{itemize}
 \item  $(\initialBundle, \nabla_0)$ is mild;  \vspace{.1cm}
  \item the monodromy representation of $(\initialBundle, \nabla_0)$ is irreducible; \vspace{.1cm}
  \item $\mathcal{I}^{\mathrm{univ, alg}}_{(C, \initialBundle, \nabla_0)}$ and $\mathcal{I}'_{(C, \initialBundle, \nabla_0)}$ 
  are both regular. \end{itemize}
   Then there are
 \begin{itemize}
  \item an \'etale base change $p:(T'',t_0'')\rightarrow (T',t_0')$; \\
  denote $(\mathcal{F}''_{(C, D)}, \familyBundle'', \nabla'' , \Psi''):=p^*\mathcal{I}'_{(C, \initialBundle, \nabla_0)}$
  ; \vspace{.1cm}
  \item a flat algebraic connection $(L, \xi)$ of rank $1$ over $T''$ with empty polar divisor;     \vspace{.1cm}
  \item a morphism $q : (T'',t_0'')\to (T,t_0)$ and  \vspace{.1cm}
  \item an  isomorphism $\morph : (\mathcal{F}_{(C, D)}'', (\familyBundle'', \nabla'')\otimes \kappa''^*(L,\xi) , \Psi'')\stackrel{\sim}{\longrightarrow} q^*\mathcal{I}^{\mathrm{univ, alg}}_{(C, \initialBundle, \nabla_0)}$ with $\morph^b=id_{\Delta''}$. \end{itemize}
   \end{prop}

  \begin{rem} 
  It is not possible without further assumptions to prove a similar statement for initial connections $(\initialBundle, \nabla_0)$ with merely semisimple monodromy representations. \end{rem}

 \section{Fundamental groups   and the Riemann-Hilbert correspondence}  
In this section, we shall see that up to an \'etale base change, any algebraic family of pointed curves can be endowed with a section avoiding the punctures.  
 The existence of such a base point section allows us to decompose the fundamental group of the total space of the family of curves into an semi-direct product of the fundamental groups of the central fiber and the parameter space. Together with the logarithmic Riemann-Hilbert correspondence, this will be used to prove the universal property of universal algebraic isomonodromic deformations. 
 
 \subsection{Splitting of the fundamental group}\label{SecSplitting}
\begin{lem}[Existence of a base point section]\label{lemsection}
 Let $\mathcal{F}_{(C,D)}=(\kappa : \familyCurve\to T,\mathcal{D}, t_0, \psi)$ be an algebraic family of pointed curves with central fiber $(C,D)$ as in Section \ref{SecSetup}. Let $x_0$ be a point in $C\setminus D$. Then there are
 \begin{itemize}
 \item a Zariski open neighborhood $\Delta$ of $t_0$ in $T$ and  \vspace{.1cm}
  \item a finite \'etale cover $p:(\Delta',t_0')\rightarrow (\Delta,t_0)$
\end{itemize} such that for $\mathcal{F}'_{(C,D)}=(\kappa' : \familyCurve'\to \Delta',\mathcal{D}', t_0',\psi')$, defined by $\mathcal{F}'_{(C,D)}:=p^*\mathcal{F}_{(C,D)}$, there exists a section $\sigma$ of $\kappa'$ with values in $\familyCurve'\setminus \mathcal{D}'$ such that $\sigma(t_0')=\psi'(x_0)$.

\end{lem}

\begin{proof} Since $\familyCurve$ is embedded in some projective space $\mathbb{P}^N$, by Bertini's Theorem, there exists a hyperplane $H$ of $\mathbb{P}^N$ which intersects $\familyCurve_{t_0}$ transversely, is disjoint from $\mathcal{D}_{{t_0}}$ and satisfies $\psi(x_0)\in H$.  Since $H$ is ample, we have $\mathrm{deg}(\familyCurve_t\cap H)>0$ for each $t\in T$. In particular, $H\cap \familyCurve_t \neq \emptyset$ for each parameter $t\in T$. 
By irreducibility of $T$, there exists an irreducible component $T'$ of $\familyCurve\cap H$ such that $\kappa(T')=T$ and $\psi(x_0)\in T'$. Now $\kappa|_{T'}: T'\to T$ is a connected finite ramified covering. Denote by $Z_1\subset T$ its branching locus. Further, denote by $Z_2$ the adherence of $\kappa(T'\cap \mathcal{D})$. By construction, 
$Z:=Z_1\cup Z_2$ is a Zariski closed proper subset of $T$ not containing $t_0$. Denote   
$ \Delta :=T\setminus Z$ and 
 $$\Delta':=\kappa^{-1}(\Delta)\cap T'\, .$$ We now have $t_0':=\psi(x_0)\in  \Delta'$ and 
$$p:=\kappa|_{\Delta'} : (\Delta',t_0')\rightarrow (\Delta,t_0)$$ is a connected finite \'etale cover. Consider the algebraic family $\mathcal{F}'_{(C,D)} :=p^*\mathcal{F}_{(C,D)}$. By definition of the pullback, its total space $\familyCurve'$ is given by a fibered product
$$\familyCurve'=\{(x,t')\in \familyCurve|_{\kappa^{-1}(\Delta)}\times \Delta' ~|~ \kappa(x)=p(t')\}$$
and we have $\kappa':\familyCurve'\to \Delta'\, ;~(x,t')\mapsto t'\, .$ On the other hand, $\Delta'$ is a subset of $\familyCurve|_{\kappa^{-1}(\Delta)}$ by construction and we can define a section $\sigma$ of $\kappa'$ by $$\sigma :\Delta'\to \familyCurve'\, ;~t'\mapsto (t',t')\, .$$  
Since moreover $\Delta'\cap \mathcal{D}=\emptyset$ by the choice of $Z_2$, we have $\sigma(\Delta')\cap \mathcal{D}'=\emptyset$. We conclude by noticing $\sigma(t_0')=(\psi(x_0),t_0')=\psi'(x_0)$.
 \end{proof}

To fix the notation, let us recall the definition of (inner) semi-direct products.

Let $G$ be a group and $A$ a subgroup. Assume we have a group $\widetilde{B}$ fitting into a split short exact sequence of groups, as follows.     \[ \xymatrix{
    \{1\} \ar[r]& A \ar[r] & G\ar[r] & \widetilde{B}  \ar@/_1pc/[l]_{\sigma}   \ar[r]&\{1\}\, 
    } \]
Assume further that the map $A\to G$ in that sequence is defined by the inclusion map. 
    Then $A$ is a normal subgroup of $G$; for $B:=\sigma(\widetilde{B})$ we have a natural morphism  
 $\eta \in \Hom(B, \mathrm{Aut}(A))$ defined by $\eta(b)(a)=b\cdot a\cdot b^{-1}$ for all $a\in A\, , b\in B;$ we have a group
  $A\rtimes_{\eta}B$ defined as the set $A\times B$ endowed with the   group law 
$$(a\, ,b)\cdot (a'\, , b') = (a\cdot \eta(b)(a') \, , b\cdot b')\, ,  $$
and the natural morphism $A\rtimes_{\eta}B\to G$ defined by $(a,b)\mapsto a\cdot b$ is bijective,  allowing us to identify
$G=A\rtimes_{\eta}B.$


 \begin{lem}[Splitting]\label{LemSplit} Let $\mathcal{F}_{(C,D)}=(\kappa : \familyCurve\to T,\mathcal{D}, t_0, \psi)$  be an algebraic family as in Section \ref{SecSetup}.  Let $\sigma : T\to \familyCurve$ be a section of $\kappa$ such that $\sigma(T)\subset \familyCurve^0=\familyCurve\setminus \mathcal{D}$. Denote $C^0:=C\setminus D$ and $x_0:=\psi^{-1}(\sigma(t_0))$. Then  \begin{equation}\label{eqSplit} \pi_1(\familyCurve^0,\sigma(t_0)) = \psi_*\pi_1(C^0,x_0)\rtimes_{\eta} \sigma_*\pi_1(T,t_0)\, ,\end{equation}
 where for all $\gamma \in \pi_1(C^0,x_0)$ and $ \beta \in \pi_1(T,t_0)$ we have $$\begin{array}{rcl}\eta(\sigma_*\beta ) (\psi_*\gamma) &=& \sigma_*\beta \cdot \psi_*\gamma  \cdot \sigma_*\beta^{-1}
 \, . \end{array}$$  \end{lem}
 \begin{proof}  Since $\sigma$ takes values in $\familyCurve^0$, we have a  morphism of fundamental groups $\sigma_* : \pi_1(T,t_0)\to \pi_1(\familyCurve^0,\sigma(t_0))$. From the embedding of the central fiber, we get the morphism $\psi_*:  \pi_1(C^0,x_0)  \to \pi_1(\familyCurve^0,\sigma(t_0))\,.$ 
    Consider now the family of $n$-punctured curves  given by $\kappa : \familyCurve^0\rightarrow T$. This family is a topologically locally trivial fibration and the fiber over $t_0$ identifies, \textit{via} $\psi$, with $C^0$. Hence we have a long homotopy exact sequence 
  \[\cdots \longrightarrow\pi_2(\familyCurve^0 ,\sigma(t_0))\stackrel{\kappa^*}{\longrightarrow} \pi_2(T ,t_0) \longrightarrow \pi_1(C^0 ,\sigma(t_0)) \stackrel{\psi^*}{\longrightarrow} \pi_1(\familyCurve^0 ,\sigma(t_0))\stackrel{\kappa^*}{\longrightarrow} \pi_1(T ,t_0)\longrightarrow \{1\}.   \]
 
 The maps ${\sigma}_* : \pi_j(T,t_0)\rightarrow \pi_j(\familyCurve^0,\tau(t_0))$ are sections for the corresponding $\kappa_*$ and we may derive the following split short exact sequence:
 
  $ \xymatrix{\textrm{~}\hspace{.7cm} & 
    \{1\} \ar[r]&\pi_1(C^0,x_0) \ar[r]^{\hskip-5pt \psi_*}& \pi_1(\familyCurve^0,\tau(t_0)) \ar[r]_-{\kappa_*}& \pi_1(T,t_0) \ar@/^-1.2pc/[]!<-1ex,-2ex>;[l]!<+4ex,-1ex>_-{{\sigma}_*}  \ar[r]&\{1\}\, .
    } $ 
 \end{proof}

 Given a decomposition \eqref{eqSplit}, the monodromy representation of the flat connection underlying an isomonodromic deformation  can be seen as an extension of the monodromy representation of the initial connection. When does such an extension exist, and is it somehow unique? Again we need a little group theory.

  \begin{lem}[Extension of representations] \label{LemExt} 
  Let $G=A\rtimes_{\eta}B$ be as before and let $\rho_A \in \Hom(A, \GL)$ be a representation. 
  \begin{itemize}
   \item There exists a representation $\rho  \in \Hom(G, \GL)$ such that $\rho|_A=\rho_A$ if and only if there exists a representation $\rho_B \in \Hom(B, \GL)$ such that for all $(a,b)\in A\times B$ 
  we have $$\rho_A(b\cdot a \cdot b^{-1})=\rho_B(b)\cdot \rho_A(a)\cdot \rho_B(b^{-1})\, . $$ 
  \item 
   Let $\rho, \rho' \in  \Hom(G, \GL)$ be representations such that  $\rho|_A=\rho'|_A=\rho_A$. Assume that $\rho_A$ is irreducible. Then there is $\lambda \in  \Hom(B,\C^*)$ such that 
  $$\rho=\lambda\otimes \rho'\, .$$
  \end{itemize}
 \end{lem}
 
 The proof of this lemma is elementary and will be left to the reader. A similar statement can be found in \cite[Lem. 1]{cousinisom}.


 \subsection{Logarithmic Riemann-Hilbert correspondence}\label{Sec mild trans}
Let us briefly recall some notions and results from \cite{cousinisom}, allowing to construct isomonodromic deformations from extensions of monodromy representations. 

Denote by $\mathbb{D}$ the unit disc around $0$ in the complex line and denote by $\trivvb$ the trivial vector bundle of rank $r$ over $\mathbb{D}$. 
A (logarithmic) \emph{transversal model} is an analytic logarithmic connection $(\trivvb, \xi)$ over $\mathbb{D}$ with polar locus $\{0\}$. It is called a \emph{mild transversal model}  if any automorphism of the locally constant sheaf $\mathrm{ker}(\xi|_{\mathbb{D}\setminus \{0\}})$ is obtained by the restriction to $\mathbb{D}\setminus \{0\}$ of an automorphism of the sheaf $\oplus_{i=1}^r\mathcal{O}_{\mathbb{D}}$ of holomorphic sections of $\trivvb$. 
Let us recall some examples.
\begin{itemize}\itemsep0pt
\item If  $(\trivvb, \xi)$ is a model such that its monodromy admits only one Jordan block for each  eigenvalue, then $(\trivvb, \xi)$ is mild. 
\item If  $(\trivvb, \xi)$ is \emph{resonant} (its residue admits two eigenvalues that differ by a non-zero integer) and has diagonalizable monodromy, then $(\trivvb, \xi)$ is not mild. 
\item If  $(\trivvb, \xi)$ is non-resonant, then $(\trivvb, \xi)$ is mild.  
\item In particular, if $(\trivvb, \xi)$ is a \emph{Deligne model}, i.e., the real parts of the eigenvalues of its residue take values in $[0,1)$, then $(\trivvb, \xi)$ is mild.
\end{itemize}
From the third example, one can easily deduce that generic representations of punctured curve fundamental groups cannot be realized by non-mild logarithmic connections.

The isomorphism class of a transversal model is called a \emph{transversal type}.  Accordingly, a \emph{mild transversal type} is the transversal type of a mild transversal model.

 Let $X$ be a  $\upomega$-manifold, and let $D$ be a reduced divisor in $X$. Denote $(D^i)_{i\in I}$ the irreducible components of $D$.
Let $$\rho \in \Hom(\pi_1(X\setminus D), \GL)$$ be a representation and $\mathcal{L}$ be a locally constant sheaf over $X\setminus D$ with monodromy $\rho$.  For each $i\in I$, choose a holomorphic embedding  $f_i : \mathbb{D}\hookrightarrow X\setminus \cup_{j\neq i} D^j$ such that $f_i (\mathbb{D})$ intersects $D^i$ transversely exactly once, at $f_i(0)$, a smooth point of $D$. We say that a transversal model $(\trivvb, \xi_i)$ is \emph{compatible} with $\rho$ at $D^i$ if its monodromy is  isomorphic to the one of $f_i^*\mathcal{L}$. This is a well-defined notion, independant of the choice of $f_i$. By isomorphism invariance, this adapts to a notion of \emph{compatible transversal type}. Compatible mild transversal models always exist, $e.g.$ one can choose Deligne models.

Assume we have a flat $\upomega$-logarithmic connection $\nabla$ over $X$, with polar locus in $D$.  By \cite[Prop. $3.2.1$]{cousinisom}, the transversal type defined by $f_i^*\nabla$ is independant of the choice of $f_i$, it depends only on $D^i$ and $\nabla$.
It is called the \emph{transversal type of $\nabla$ at $D^i$}.  The connection $\nabla$  is said to be \emph{mild} if for every component $D^i$, the transversal type of $\nabla$ at $D^i$ is mild. 
\begin{thm}[Logarithmic Riemann-Hilbert]\label{thmRH} 
Let $X$ be a $\omega$-manifold, let $D$ be a smooth reduced divisor in $X$ and let $\rho : \pi_1(X\setminus D)\to \mathrm{GL}_r\C$ be a representation. Let $(D^i)_{i\in I}$ be the irreducible components of $D$. For each $i\in I$, let  $(\trivvb, \xi_i)$ be a mild transversal model compatible with $\rho$. Then up to isomorphism there is a unique flat $\upomega$-logarithmic connection $(E, \nabla)$ over $X$ with polar locus $D$  such that 
\begin{itemize}
\item the monodromy of $(E, \nabla)$ is given by $[\rho]$ and\vspace{.1cm}
\item for each $i\in I$, the transversal type of $\nabla$ at $D^i$ is given by $(V,\xi_i)$;\vspace{.1cm}
\item if ``$\upomega=$ algebraic'', then $(E, \nabla)$ is regular.
\end{itemize}
\end{thm}
 
 \begin{proof}
The proof of this theorem in the analytic category can be found in \cite[Section 3.2]{cousinisom}. We only need to check that it also holds in the algebraic category. So assume now $X$ is a smooth irreducible quasiprojective variety. 
By definition, there exists  a smooth irreducible projective variety $\widehat{X}$ containing $X$  as  a Zariski open subset. Denote by $\widehat{D}^{j}, j\in J$, the irreducible components of $\widehat{X} \setminus X$ and by $\widehat{D}^{i}$ the Zariski closure of $D^i$ in $\widehat{X}$ for each $i\in I$. By Hironaka's desingularization, we may suppose that $\widehat{D}:=\sum_{i\in I\cup J}\widehat{D}^i$ is a normal crossing divisor. 
Moreover, since $X\setminus D=\widehat{X}\setminus \widehat{D}$, the representation $\rho$ defines $$\widehat{\rho}=\rho \in \Hom(\pi_1(\widehat{X}\setminus \widehat{D}), \GL)\, .$$
For each $j\in J$, choose a Deligne model $(\trivvb, \xi_j)$ on $(\mathbb{D},0)$ compatible with $\widehat{\rho}$.
Then again by \cite[Section 3.2]{cousinisom}, there exists an analytic logarithmic connection $(\widehat{E}^\mathrm{an},\widehat{\nabla}^\mathrm{an})$ over $\widehat{X}$ with polar divisor $\widehat{D}$ and the prescribed transversal types. Since $\widehat{X}$ is projective, this connection is however analytically isomorphic to an algebraic logarithmic connection $(\widehat{E},\widehat{\nabla})$ on $\widehat{X}$ by GAGA  \cite[Prop. $18$]{GAGA}. Since any logarithmic connection on $\widehat{X}$ restricts to a regular connection on $X$ (see \cite[Thm $4.1$]{MR0417174}), $(E,\nabla):=(\widehat{E},\widehat{\nabla})|_X$ has the desired properties. It remains to show uniqueness up to algebraic isomorphism. By the analytic statement, we already know that the (analytic) isomorphism class of $(E^{\mathrm{an}}, \nabla^{\mathrm{an}})$ is unique. Yet any analytic isomorphism between  {regular} algebraic logarithmic connections $\nabla_1, \nabla_2$ over $X$  is algebraic, for the isomorphism can be seen as a horizontal section of $\nabla_1\otimes \nabla_2^{\vee}$, which is regular by \cite[Prop. $4.6$]{MR0417174}.
\end{proof}
\begin{rem}
Using the above regularity argument, one also obtains that two flat logarithmic connections with the same monodromy are bimeromorphically equivalent. In particular, every logarithmic connection over a curve is mild up to a meromorphic gauge transformation.
\end{rem}
 
 
\subsection{Proof of the universal property}\label{Sec proof UP}
 Lemma \ref{LemStandardiser}, stated in Section \ref{SecDefDef}, implying that under suitable generic conditions, algebraic universal isomonodromic deformations, if they exist, may be chosen to be regular  is now an immediate consequence of the logarithmic Riemann-correspondence. Moreover, we are now able to prove their universal property, also stated in Section \ref{SecDefDef}. 
     \begin{proof}[Proof of Lemma \ref{LemStandardiser}] Let $\mathcal{I}_{(C, \initialBundle, \nabla_0)}=(\mathcal{F}_{(C, D)}, \familyBundle, \nabla , \Psi)$ with $\mathcal{F}_{(C,D)}=(\kappa : \familyCurve\to T,\mathcal{D}, t_0, \psi)$ be an algebraic isomonodromic deformation of $(C, \initialBundle, \nabla_0)$. Let $\rho\in \Hom (\pi_1(\familyCurve\setminus \mathcal{D}, x_0), \GL)$ be a representative of the monodromy $[\rho]$ of $(\familyBundle, \nabla)$. For $i\in\intint{1}{n}$, let $\mathcal{D}^i$ be the component of $\mathcal{D}$ passing through $\psi(x_i)$.  Since by assumption $(C, \initialBundle, \nabla_0)$ is mild, Theorem \ref{thmRH} yields a regular algebraic connection $(\familyBundle', \nabla')$ over $\familyCurve$ with polar divisor $\mathcal{D}$, monodromy  $[\rho]$ and the same transversal types as $\nabla$ at the components $(\mathcal{D}^i)_{i\in\intint{1}{n}}$. Moreover, also by Theorem \ref{thmRH}, there is an isomorphism $\widetilde{\Psi} : (\familyBundle, \nabla)|_{\familyCurve_{t_0}}\stackrel{\sim}{\to }(\familyBundle', \nabla')|_{\familyCurve_{t_0}}$. Then 
   $\mathcal{I}'_{(C, \initialBundle, \nabla_0)}:=(\mathcal{F}_{(C, D)}, \familyBundle', \nabla' ,\widetilde{\Psi}\circ \Psi)$ is a regular algebraic   isomonodromic deformation of $(C, \initialBundle, \nabla_0)$ and there is an analytic isomorphism $(\familyBundle^\mathrm{an}, \nabla^\mathrm{an})\simeq ({\familyBundle'}^\mathrm{an}{\nabla'}^\mathrm{an})$. In particular, the analytification of $\mathcal{I}_{(C, \initialBundle, \nabla_0)}$ is  isomorphic to the analytification of $\mathcal{I}'_{(C, \initialBundle, \nabla_0)}$.
   \end{proof} 
  
      \begin{proof}[Proof of Proposition \ref{propuniviso}]   Let  $\mathcal{I}^{\mathrm{univ, alg}}_{(C, \initialBundle, \nabla_0)}=(\mathcal{F}^{\mathrm{Kur}}_{(C, D)}, \familyBundle, \nabla , \Psi)$ be a regular algebraic universal isomonodromic deformation of $(C, \initialBundle, \nabla_0)$ with parameter space $(T,t_0)$ and  let $\mathcal{I}'_{(C, \initialBundle, \nabla_0)}=(\mathcal{F}_{(C, D)}', \familyBundle', \nabla' , \Psi')$ be a regular algebraic isomonodromic deformation of $(C, \initialBundle, \nabla_0)$ with parameter space $(T',t_0')$. 
     By Lemma \ref{lemsection},   there is an  \'etale base change $\tilde{p} : (\widetilde{T} , \tilde{t}_0) \to (T, t_0)$, such that for $\widetilde{\mathcal{F}}^{\mathrm{Kur}}_{(C, D)} :=\tilde{p}^*\mathcal{F}^{\mathrm{Kur}}_{(C, D)}, $
 there is a section $\sigma : \widetilde{T}\to \widetilde{\familyCurve}$ avoiding the marked points. 
  Since $\widetilde{\mathcal{F}}^{\mathrm{Kur}}_{(C, D)}$ is still Kuranishi, by the universal property of Kuranishi families,  we have an   \'etale base change $ {p} : (T'', t_0'')\to (T' , {t}_0')$, a morphism $\tilde{q} : (T'',t_0'')\to (\widetilde{T}, \tilde{t}_0)$ and an isomorphism
$$\tilde{\morph} : \mathcal{F}''_{(C,D)}:={p}^* \mathcal{F}'_{(C,D)} \stackrel{\sim}{\longrightarrow} \tilde{q}^*\widetilde{\mathcal{F}}^{\mathrm{Kur}}_{(C,D)}\, .$$
In particular, $\sigma$ lifts to a section $\sigma'':=\tilde{f}^*\tilde{q}^*\sigma : T''\to \familyCurve''$ avoiding the marked points of $\mathcal{F}_{(C, D)}''$. Denote by $$\rho''\, , \tilde{\rho} \, \in \, \Hom(\pi_1(\familyCurve''\setminus \mathcal{D}''\, , \sigma''(t_0'')),  \GL)$$ 
representatives of the conjugacy classes of the monodromy representations of $(\familyBundle''; \nabla'')$ and $\tilde{f}^*\tilde{q}^* \tilde{p}^*(\familyBundle,\nabla) $ respectively (with respect to the identity). 
By the Splitting Lemma \ref{LemSplit}, we have $$\pi_1(\familyCurve''\setminus \mathcal{D}'' , \sigma''(t_0''))= \psi''_* \pi_1(C\setminus D, x_0) \rtimes_{\eta} \sigma''_*\pi_1(T'',t_0'')\, .$$
Moreover, if $\rho_{\nabla_0}$ denotes a representative of the monodromy representation of $(\initialBundle, \nabla_0)$ (with respect to the identity), then $\rho''$ and $\tilde{\rho}$ could be chosen so that 
$$\rho''|_{ \psi''_* \pi_1(C\setminus D, x_0) }= \tilde{\rho}|_{ \psi''_* \pi_1(C\setminus D, x_0) }=\psi''_*\rho_{\nabla_0}.$$
Since $\rho_{\nabla_0}$ is irreducible, by Lemma \ref{LemExt}  there is a representation $\lambda \in \Hom(\pi_1(T''\, ,\, t_0'')\, ,  \C^*)$ such that  $\lambda \otimes (\sigma'')^*\rho'' = (\sigma'')^*\tilde{\rho}$.  By the Riemann-Hilbert correspondence, there is  a regular flat algebraic connection $(L, \xi)$  of rank $1$ over $T''$, without poles, whose monodromy representation is $\lambda^{-1}$. The monodromy representation of its lift ${\kappa''}^*(L, \xi)$ is the trivial extension of ${\sigma''}^*\lambda^{-1}$ to a representation $\psi''_* \pi_1(C\setminus D, x_0) \rtimes_{\eta} \sigma''_*\pi_1(T'',t_0'')\, \rightarrow \C^*$. Now the monodromy representations of $(\familyBundle''; \nabla'')\otimes {\kappa''}^*(L, \xi)$ and $\tilde{f}^*q^*p^*(\familyBundle,\nabla) $ coincide. Both connections are regular, have same monodromy representations and same transversal models, given by $(\initialBundle, \nabla_0)$. Hence they are isomorphic by the logarithmic Riemann-Hilbert correspondence. 
\end{proof} 

  \section{The monodromy of the monodromy}\label{SecMonMon}
In this section, we introduce the so-called group of  mapping classes of a $\upomega$-family, which is the image of a canonical morphism from the fundamental group of the parameter space of the family to the mapping class group of the central fiber.  For an isomonodromic deformation, the action on the monodromy representation of the initial connection by the group of  mapping classes of the underlying family of curves corresponds to the monodromy of the monodromy representation. 
Under suitable conditions,  this group can be canonically translated into a subgroup of $\MCG$.  

\subsection{Mapping classes of the central fiber}\label{Sec map}
As usual, let $(C,D)$ be a stable $n$-pointed genus $g$ curve. Let
$\mathcal{F}_{(C,D)}$ be a $\upomega$-family with parameter space $(T,t_0)$.  Let  $\beta : [0,1]\rightarrow T$ be a closed path with endpoint $t_0$, \textit{i.e.} a continous map such that $\beta(0)=\beta(1)=t_0$. 
By \cite[Cor. $10.3$]{MR1249482}, the pullback bundle $\beta^*(\familyCurve,\mathcal{D}) \rightarrow [0,1]$ possesses a topological trivialization $\Phi :(C, D)\times [0,1] \stackrel{\sim}{\to}\beta^*(\familyCurve,\mathcal{D})$.
For $s\in [0,1]$, we denote $$\Phi_s:=\Phi|_{(C, D)\times \{s\}} \, $$ and deduce a homeomorphism from the central fiber seen over $\{1\}$ to the central fiber seen over $\{0\}$ given by 
 $$\psi^{-1}\circ \Phi_0\circ \Phi_1^{-1} \circ \psi: (C, D) \stackrel{\sim}{\to}  (C, D)\, .$$ Its isotopy class
 shall be called the \emph{mapping class associated to $\beta$ and $\mathcal{F}_{(C,D)}$} and denoted  $$\mathrm{map}_{\mathcal{F}_{(C,D)}}(\beta)\, .$$ 

 \begin{lem} The mapping class $\mathrm{map}_{\mathcal{F}_{(C,D)}}(\beta)$ is well-defined, \textit{i.e.} it does not depend on the choice of a trivialization $\Phi$. Moreover,  $\mathrm{map}_{\mathcal{F}_{(C,D)}}(\beta)$ only depends on the homotopy class of $\beta$. 
 \end{lem}
 
 \begin{proof}
 For fixed $\beta$, take two trivializations : $\Phi,\widetilde{\Phi}: (C, D)\times [0,1] \stackrel{\sim}{\to}\beta^*(\familyCurve,\mathcal{D})$. The family  $\widetilde{\Phi}_0 \circ\widetilde{\Phi}_s^{-1}\circ \Phi_s\circ  \Phi_1^{-1}$ gives an isotopy from $\Phi_0\circ \Phi_1^{-1}$ to $\widetilde{\Phi}_0\circ \widetilde{\Phi}_1^{-1}$.\\
  Consider now two paths $\beta_1$ and $\beta_2$ that are homotopic relative to their endpoints. By definition, there exists  a continuous map $\theta : \overline{\D} \rightarrow T$, where $ \overline{\D}$ denotes the closed unit disc, such that $\beta_2(s)=\theta(\mathrm{e}^{i\pi (1+s)})$ and $\beta_1(s)=\theta(\mathrm{e}^{i\pi (1-s)})$. Since  $ \overline{\D}$ is contractible, by \cite[Cor. $10.3$]{MR1249482}, there is a trivialization $\Phi$ of $\theta^*(\familyCurve,\mathcal{D})$. It induces trivializations $\Phi^i$ of $\beta_i^*\familyCurve$ for $i=1,2$. Since they are both induced by $\Phi$, we have $\Phi^1_0=\Phi^2_0=\Phi_{-1}$ and $\Phi^1_1=\Phi^2_1=\Phi_{1}$.  \end{proof}
  
   \begin{prop} \label{lem relation 2 actions}
Let
$\mathcal{F}_{(C,D)}=(\kappa : \familyCurve\to T,\mathcal{D}, t_0, \psi)$ be a $\upomega$-family as in Section \ref{SecSetup}. Assume that none of the fibers $(\familyCurve_t, \mathcal{D}_{t})$ has exceptional automorphisms. Let $\mathbf{x}$ be a labelling of $D$ and denote  
$\clF : T\rightarrow \Mgn\setminus \Bgn$ the corestriction of the induced classifying map  $\class (\mathcal{F})$ (see Section \ref{Sec univ fam}).  Let $\varphi : (\Sigma_g,y^n)\stackrel{\sim}{\to} (C,\mathbf{x})$ be an orientation preserving homeomorphism and denote by $\hat{\star}:=[C,D,\varphi]$ the corresponding point in $\Tgn$.  Then  for all $ \beta \in \pi_1({T},t_0)$, the following equation holds in $\MCG/K_{g,n}$:
$$\varphi^{-1}\circ \mathrm{map}_{\mathcal{F}_{(C,D)}}(\beta)\circ \varphi=\tautgn(\clF_*\beta) \, ,$$
where $\tautgn$ is the tautological morphism 
$\tautgn :\pi_1( \Mgn\setminus \Bgn, \star)\to \MCG/K_{g,n}$ (see \eqref{def phi}
) and $\star:=[C,\mathbf{x}]\in \Mgn$.
 \end{prop}
 
 \begin{proof}
 Denote $\mathcal{F}^{+}_{g,n}=(\mathcal{F}_{g,n}, \Phi_{g,n})$ the universal Teichm\"uller curve 
 $\mathcal{F}_{g,n}=(\kappa_{g,n} : \mathcal{X}\to \Tgn, \mathcal{Y})$ endowed with the Teichm\"uller structure 
 $ \Phi_{g,n}: (\Sigma_g,Y^n)\times \Tgn \stackrel{\sim}{\to} (\mathcal{X}, \mathcal{Y})$. For any point $t \in \Tgn$, we shall denote $$\Phi^{g,n}_{t}:=\Phi_{g,n}|_{(\Sigma_g,Y^n)\times \{t \}} : (\Sigma_g,Y^n)\times \{t \}  \stackrel{\sim}{\to} (\mathcal{X}_{t}, \mathcal{Y}_t )\, .$$
 Let $p : (\widetilde{{T}}, \tilde{t}_0) \to ({T},t_0)$ be a universal cover and consider the pulled-back family
 $$\widetilde{\mathcal{F}}=(\widetilde{\kappa} : \widetilde{\mathcal{C}}\to \widetilde{{T}},\widetilde{\mathcal{D}}):=p^*(\kappa : \mathcal{C}\to {T},\mathcal{D})\, .$$
Now for any contractible analytic submanifold $\widetilde{\Delta} \subset \widetilde{{T}}$ containing $\tilde{t}_0$, there is a topological trivialization $${\Phi}: (C,D)\times \widetilde{\Delta} 
\stackrel{\sim}{\longrightarrow} (\widetilde{\mathcal{C}} , \widetilde{\mathcal{D}})|_{\widetilde{\Delta}}
$$ of $\widetilde{\mathcal{F}}|_{\widetilde{\Delta}}$, unique up to isotopy, such that ${\Phi}_{\tilde{t}_0}=\psi$ with respect to the identification  $$(\widetilde{\mathcal{C}}_{\tilde{t}_0}, \widetilde{\mathcal{D}}_{\tilde{t}_0}){=}(\mathcal{C}_{t_0},\mathcal{D}_{t_0})$$ provided by pullback. 
We denote $\widetilde{\Phi}:={\Phi}\circ ({\varphi}\times \mathrm{id})$.  
Setting 
$\widetilde{\mathcal{F}}^+:=(\widetilde{\mathcal{F}}|_{\widetilde{\Delta}}, \widetilde{\Phi})$ defines an analytic family of compact Riemann surfaces with marked points and   Teichm\"uller structure. 
 By the universal property of the  Teichm\"uller curve, up to modifying $\widetilde{\Phi}$ by a fiber-preserving isotopy, there exists a unique isomorphism $\morph$ of complex manifolds fitting into the following commutative diagram:
     \[ \xymatrix{
  (\Sigma_g, Y^n)\times \widetilde{\Delta} \ar@{=}[rr]\ar[d]_{\widetilde{\Phi} }&& (\Sigma_g, Y^n)\times \widetilde{\Delta}\ar[d]^{\classan(\widetilde{\mathcal{F}}^+)^*\Phi_{g,n}}\\
  (\widetilde{\mathcal{C}}, \widetilde{\mathcal{D}})|_{\widetilde{\Delta}}\ar[rr]^{\morph}_{\sim}\ar[d]_{\widetilde{\kappa}}&& \classan(\widetilde{\mathcal{F}}^+)^*(\mathcal{X}, \mathcal{Y})\ar[d]^{\classan(\widetilde{\mathcal{F}}^+)^*\kappa_{g,n}}\\
  \widetilde{\Delta} \ar@{=}[rr]&&\widetilde{\Delta}\, .
     } \] 
       Now let $[\beta]\in \pi_1({T},t_0)\setminus\{1\}$ and consider $\widetilde{\beta}: [0,1]\to \widetilde{{T}}$, the lift of $\beta$ with starting point $\tilde{t}_0$. If the representative $\beta$ of the homotopy class $[\beta]$ is well chosen, then $\widetilde{\beta}$ is a $\mathcal{C}^{\infty}$-embedding. By existence of tubular neighborhoods, there is a contractible neighborhood $\widetilde{\Delta}$ of $\tilde{t}_0$ as above, containing $\widetilde{\beta}$. 
       We claim that, up to isotopy, 
   \begin{equation}\label{idee1} \mathrm{map}_{\mathcal{F}_{(C,D)}}(\beta)=  {\Phi}_{\widetilde{\beta}(1)}^{-1}\circ  {\Phi}_{\tilde{t}_0}\, .
   \end{equation}
   Indeed, we have $\beta^*(\mathcal{C}, \mathcal{D})=(p\circ \widetilde{\beta})^*(\mathcal{C}, \mathcal{D})=\widetilde{\beta}^*p^*(\mathcal{C}, \mathcal{D})=\widetilde{\beta}^*(\widetilde{\mathcal{C}},\widetilde{\mathcal{D}})$.
    The claim then follows from  the fact that $ \psi^{-1}\circ  {\Phi}_{\tilde{t}_0}$ is the identity and from
the definition of the mapping class.

     Denote $\widehat{\beta}:=\classan(\widetilde{\mathcal{F}}^+)_*\widetilde{\beta}$, which is a path in $\Tgn$ with starting point $\hat{\star}$.
   By our definitions, the  following diagram is commutative if we remove the dotted arrow.    
     \[ \xymatrix{
  (\Sigma_g, Y^n)\times \{\tilde{t}_0\} \ar[rr]^{\widetilde{\Phi}_{\tilde{t}_0}}_{\sim}
  \ar@{=}[d]   &&(\widetilde{\mathcal{C}}_{\tilde{t}_0}, \widetilde{\mathcal{D}}_{{\tilde{t}_0}})\ar@{=}[r] \ar[d]_{\morph_{\tilde{t}_0}}^{\hskip3pt \sim 
   }
&(\widetilde{\mathcal{C}}_{\widetilde{\beta}(1)}, \widetilde{\mathcal{D}}_{{\widetilde{\beta}(1)}})  \ar[d]^{\morph_{\widetilde{\beta}(1)}}_{\hskip3pt \sim
}
   &&  (\Sigma_g, Y^n)\times \{\widetilde{\beta}(1)\} \ar[ll]_{\widetilde{\Phi}_{\widetilde{\beta}(1)}}^{\sim} \ar@{=}[d]    \\
 (\Sigma_g, Y^n)\times \{\hat{\star}\}\ar[rr]_{\Phi^{g,n}_{\hat\star}}^{\sim}&& (\mathcal{X}_{\hat{\star}}, \mathcal{Y}_{{\hat{\star}}}) \ar@{..>}
 [r]^{
 \hskip-15pt\sim}_{
  \hskip-15pt \widehat{\psi}} &  (\mathcal{X}_{\widehat{\beta}(1)}, \mathcal{Y}_{{\widehat{\beta}(1)}})&&(\Sigma_g, Y^n)\times \{\widehat{\beta}(1)\} \ar[ll]^{\Phi^{g,n}_{\widehat{\beta}(1)}}_{\sim}
   } \]  
   We define the induced isomorphism of pointed curves  \begin{equation}\label{def psi}\widehat{\psi}=\morph_{\tilde{\beta}(1)}\circ(\morph_{t_0})^{-1},\end{equation} so that adding the dotted arrow maintains this commutativity.     
We have 
     \begin{equation}\label{idee2}\left\{\begin{array}{ccccc} 
   \Phi^{g,n}_{\hat{\star}} &=&\morph_{\tilde{t}_0}\circ \widetilde{\Phi}_{\tilde{t}_0}&=&\morph_{\tilde{t}_0}\circ  {\Phi}_{\tilde{t}_0}\circ  \varphi\vspace{.1cm}\\
    \Phi^{g,n}_{\widehat{\beta}(1)}&=&
    \morph_{\widetilde{\beta}(1)}\circ  \widetilde{\Phi}_{\widetilde{\beta}(1)}&=&
    \morph_{\widetilde{\beta}(1)}\circ  {\Phi}_{\widetilde{\beta}(1)} \circ \varphi. \vspace{.1cm}\\
     \end{array}\right.
   \end{equation}
 
On the other hand, $\clF_*\beta$ is a closed path in $\Mgn\setminus \Bgn$ with end point $\star$. By construction, it lifts, with respect to the forgetful map $\forggn$, to $\widehat{\beta}$, with $\widehat{\beta}(0)=\hat{\star}$. By definition of the tautological morphism $\tautgn$, we thus have, for $[h]:= \tautgn (\clF_*\beta) \in \MCG/K_{g,n}$: 
$$[h]\cdot \left[\mathcal{X}_{\hat{\star}}, \mathcal{Y}_{{\hat{\star}}}, \Phi^{g,n}_{\hat{\star}}\right]=\left[\mathcal{X}_{\widehat{\beta}(1)}, \mathcal{Y}_{{\widehat{\beta}(1)}}, \Phi^{g,n}_{\widehat{\beta}(1)} \right]\, .$$

\noindent By the definition of the action of the mapping class group on $\Tgn$, we now have
$$[h]\cdot  \left[\mathcal{X}_{\hat{\star}}, \mathcal{Y}_{{\hat{\star}}}, \Phi^{g,n}_{\hat{\star}}\right] = [h]\cdot \left[\mathcal{X}_{\widehat{\beta}(1)}, \mathcal{Y}_{{\widehat{\beta}(1)}},\widehat{\psi}\circ \Phi^{g,n}_{\hat{\star}} \right]= \left[\mathcal{X}_{\widehat{\beta}(1)}, \mathcal{Y}_{{\widehat{\beta}(1)}},\widehat{\psi}\circ \Phi^{g,n}_{\hat{\star}} \circ h^{-1}\right]\, .$$

\noindent Hence there is an element $[k]\in K_{g,n}$ such that, up to isotopy, 
$$\widehat{\psi}\circ \Phi^{g,n}_{\hat{\star}} =  \Phi^{g,n}_{\widehat{\beta}(1)} \circ h\circ k\, .$$
Combined with \eqref{def psi} and \eqref{idee2}, this implies, up to isotopy,
 $$  {\Phi}_{\widetilde{\beta}(1)}^{-1} \circ   {\Phi}_{\tilde{t}_0} =  \varphi\circ h \circ k\circ \varphi^{-1}\, ,$$
 which by \eqref{idee1} and the definitions of $h$ and $k$ yields the desired result.
 \end{proof}

\subsection{Splitting and the mapping class group}\label{SecNotBul}
   
 Let  $\mathcal{F}_{(C,D)}=(\kappa : \familyCurve \to T,\mathcal{D}, t_0, \psi)$  be a $\upomega$-family of $n$-pointed genus $g$ curves as in Section~\ref{SecSetup}. Assume there is a section $\sigma : T\to \familyCurve^0:=\familyCurve\setminus \mathcal{D}$ of $\kappa$. Then we can define a  $\upomega$-family of $n+1$-pointed genus $g$ curves
  $\mathcal{F}^{\bullet}_{(C,D^\bullet)}:=(\kappa : \familyCurve\to T,\mathcal{D}^\bullet, t_0, \psi)$ by setting $\mathcal{D}^\bullet:=\mathcal{D}+\sigma(T)$ and $D^\bullet:=D+x_0$, where $x_0:=\psi^{-1}(\sigma(t_0))$.  To a labelling $\mathbf{x} =(x_1, \ldots, x_n)$   of $D$ we can associate a labelling $\mathbf{x}^\bullet:=(x_1, \ldots, x_n,x_0)$   of $D^\bullet$. Note that if a fiber of $\mathcal{F}^\bullet$ has exceptional automorphisms, then the corresponding fiber of $\mathcal{F}$ also has exceptional automorphisms.
  
   If none of the fibers of $\mathcal{F}^\bullet$  has exceptional automorphisms, we may corestrict the classifying map $\class (\mathcal{F}^\bullet)$ to obtain a morphism
  $$\clF^{\bullet}: \mathcal{T}\to \Mgnplusun \setminus \Bgnplusun \, .$$
  Let  $\varphi : (\Sigma_g,y^n, y_0)\stackrel{\sim}{\to} (C,\mathbf{x}, x_0)$ be an orientation preserving homeomorphism and denote $\hat{\star}^{\bullet}:=[C,D^\bullet,\varphi]\in \mathcal{T}_{g,n+1}$ and ${\star}^{\bullet}:=[C,\mathbf{x}^\bullet]\in \Mgnplusun$. 
  Note that since we assumed $2g-2+n>0$, we have  $K_{g,n+1}=\{1\}$ according to Lemma \ref{lemme Kgn}. 
 We obtain a tautological morphism 
 $$\taut_{\hat{\star}^{\bullet}}: \pi_1(\Mgnplusun \setminus \Bgnplusun, {\star}^{\bullet})\to  \Gamma_{g,n+1}$$
 as in \eqref{def phi}.

 \begin{prop}\label{corSplit} Let $\mathcal{F}^{\bullet}_{(C,D^\bullet)}=(\kappa : \familyCurve\to T,\mathcal{D}^\bullet, t_0, \psi)$  be a $\upomega$-family of $n+1$-pointed genus $g$ curves as above.    Assume that none of the fibers of $\mathcal{F}^\bullet$ has exceptional automorphisms. Let $\varphi : (\Sigma_g,y^n, y_0)\stackrel{\sim}{\to} (C,\mathbf{x}, x_0)$ be an orientation preserving homeomorphism. Then
    \begin{equation}\label{eqSplit2} \pi_1(C^0,\sigma(t_0)) =(\psi\circ \varphi)_*\gf\rtimes_{\eta} \sigma_*\pi_1({T},t_0)\, ,\end{equation}
 where for all $\alpha \in \gf$ and $ \beta \in \pi_1({T},t_0)$, we have $$\begin{array}{rcl}\eta(\sigma_*\beta ) ((\psi\circ \varphi)_*\alpha) &=& \sigma_*\beta \cdot (\psi\circ \varphi)_*\alpha  \cdot \sigma_*\beta^{-1}\vspace{.1cm}\\
 &=& (\psi\circ \varphi)_* \, \mathfrak{a}\left( \taut_{\hat{\star}^{\bullet}}({\clF^{\bullet}}_*\, \beta)\right)(\alpha) \, . \end{array}$$ 
Here we adopt the notation above and denote $\mathfrak{a}(h)(\alpha)=h_*\alpha$ for all $h\in \Gamma_{g,n+1}$ and $\alpha \in \gf$ as in the introduction.   \end{prop}
  \begin{proof} 
By Proposition \ref{lem relation 2 actions},   the following equation holds in $\Gamma_{g,n+1}=\Gamma_{g,n+1}/K_{g,n+1}$ for every $\beta \in \pi_1(\mathcal{T},t_0)$:
 \begin{equation}\label{mapettaut}\varphi^{-1}\circ \mathrm{map}_{\mathcal{F}^\bullet_{(C,D^\bullet)}}(\beta)\circ \varphi=\taut_{\hat{\star}^{\bullet}}({\clF^\bullet}_*\beta)\, .\end{equation}
  Denote $C^0:=C\setminus D$. We claim that for any $\gamma \in  \pi_1(C^0,x_0) $ and any $\beta \in \pi_1(T,t_0)$, the following equation  holds  in $\pi_1(\familyCurve^0,\sigma(t_0))$:
\begin{equation}\label{mapetgf} \psi_*\mathrm{map}_{\mathcal{F}^\bullet_{(C,D^\bullet)}}(\beta)_*\gamma = \sigma_* \beta \cdot \psi_*\gamma   \cdot \sigma_* \beta^{-1} \, .
\end{equation}
Indeed, let $\gamma : [0,1]\to C^0$ be a closed path with endpoint $x_0$.
For any $s_0\in [0,1]$, we have a closed path $\gamma_{s_0}:=\gamma \times \{s_0\}$ in the product space $C^0\times [0,1]$. We also have a path $\theta : [0,1] \to C^0\times [0,1]\, ; \, s\mapsto (x_0,s)$. The path $\theta \cdot \gamma_1 \cdot \theta^{-1}$ is closed and homotopic to $\gamma_0$. Now let $\beta \in \pi_1(T,t_0)$ and let $\Phi : (C^0, x_0)\times [0,1] \stackrel{\sim}{\to} \beta^*(\familyCurve^0, \sigma(T))$ be a trivialization commuting with the natural projections to $[0,1]$. Define 
the homeomorphism $$\widetilde{\Phi}:=\Phi\circ ((\Phi_{1}^{-1}\circ \psi) \times \mathrm{id}_{[0,1]}): (C^0,x_0)\times [0,1] \stackrel{\sim}{\to} \beta^*(\familyCurve^0, \sigma(T))\, , $$
which is another trivialization, satisfying $\widetilde{\Phi}_1=\psi$ and $ \widetilde{\Phi}_0=\psi_*\mathrm{map}_{\mathcal{F}^\bullet_{(C,D^\bullet)}}(\beta).$ Since $\widetilde{\Phi}$ is continuous,
the closed paths $\widetilde{\Phi}_*\gamma_0$ and $\widetilde{\Phi}_*\theta\cdot \widetilde{\Phi}_*\gamma_1 \cdot \widetilde{\Phi}_*\theta^{-1}$ are homotopic in $ \beta^*(\familyCurve^0, \sigma(T))$. Considering the natural projection 
$\kappa:   \beta^*(\familyCurve^0, \sigma(T)) \to (\familyCurve^0, \sigma(T))$, we have $\kappa_*\widetilde{\Phi}_{*}\gamma_0=\widetilde{\Phi}_{0*}\gamma$ and $\kappa_*\widetilde{\Phi}_*\gamma_1=\widetilde{\Phi}_{1*}\gamma$. Since moreover $\kappa_*\widetilde{\Phi}_*\theta=\sigma_*\beta$, we have \eqref{mapetgf}.

  Since $\varphi$ is a homeomorphism, the induced map 
  $\varphi_*:\gf \to \pi_1(C^0,x_0)$ is an isomorphism. 
 The statement then follows from \eqref{mapettaut}, \eqref{mapetgf} and the Splitting Lemma \ref{LemSplit} . \end{proof}

\section{Necessary and sufficient conditions for algebraizability}  We shall see in Section \ref{SecFinPreuve} that Theorem  \ref{algebrization thm} is a corollary of the juxtaposition of \hyperlink{thA1}{Theorem A1}, showing that our algebraizability criterion for germs of universal isomonodromic deformations is necessary, and \hyperlink{thA2}{Theorem A2}, showing that it is also sufficient.
We have already established the main ingredients for the proofs of both theorems.
For \hyperlink{thA2}{Theorem A2}, we moreover need a representation-theoretical result developed in Section \ref{SecExt}.

\subsection{Extensions of representations}\label{SecExt}
We shall now consider the problem of extending a representation of the fundamental group of a fiber of a family of pointed curves to a representation for the whole family in  light of Lemma \ref{LemExt}  and Proposition \ref{corSplit}. We begin with the elementary case of non-semisimple rank 2 representations.

Let $A,B$ be groups. Consider a representation $\rho\in \Hom(A ,\triang)$, where $\triang$ is the group of invertible  upper triangular matrices of rank $2$. To such a representation, we may associate two other ones :
the \emph{scalar part}  $\rho_{\C^*} :\alpha \mapsto \rho(\alpha)_{2,2}$ and the \emph{affine part} $\rho_{\mathrm{Aff}}:=\rho_{\C^*}^{-1} \otimes \rho\label{aff/scalar}$. The latter takes values in $$\mathrm{Aff}(\C):=\{(a_{i,j})\in \triang~|~a_{2,2}=1\}$$ which is   isomorphic to the affine group of the complex line. 

\begin{lem} \label{lemred}
Let $\rho =\rho_{\C^*}\otimes \rho_{\mathrm{Aff}}$ and $\rho' =\rho'_{\C^*}\otimes \rho'_{\mathrm{Aff}}$ as above, and assume that they are not semisimple. We have 
$[\rho]=[\rho']  \in \Hom(A , \mathrm{GL}_{2}\C)/\mathrm{GL}_{2}\C$ if and only if $\rho_{\C^*}=\rho'_{\C^*}$ and 
$[\rho_{\mathrm{Aff}}]=[\rho'_{\mathrm{Aff}}]  \in \Hom(A , \Aff)/\Aff$. 
   \end{lem}
\begin{proof} The ''if''-part is trivial. Assume $[\rho]=[\rho']$. 
Since they take values in $\triang$, both representations $\rho$ and $\rho'$ leave the line $\mathrm{span}(e_1)$ of $\C^2$ invariant. 
By non-semisimplicity, for each of the representations, there is no other  globally invariant line. Let $M=(m_{i,j})\in\mathrm{GL}_2\C$ conjugate both representations. Then $M$ must  leave $\mathrm{span}(e_1)$ invariant, \textit{i.e.} $M\in \triang$. As the scalars are central in $\mathrm{GL}_2\C$, the element $M/m_{2,2}\in \Aff$  conjugates both representations. In particular $\rho_{\C^*}=\rho'_{\C^*}$ and $M/m_{2,2}$ conjugates $\rho_{\mathrm{Aff}}$ and $\rho'_{\mathrm{Aff}}$.  \end{proof}

\begin{lem} \label{lemextrep} Let $\rho_A \in \Hom(A,\mathrm{GL}_{2}\C)$ be non-semisimple. Let $\theta \in \Hom(B, \mathrm{Aut}(A))$ such that for all $h\in \mathrm{Im}(\theta)$, we have $[\rho_A]=h\cdot [\rho_A]:=[\rho_A\circ h^{-1}]$. Then there exists a representation $\rho_B \in \Hom(B,\mathrm{GL}_{2}\C)$ such that 
$$\rho_A(\theta(\beta)^{-1}(\alpha))=\rho_B(\beta)^{-1} \rho_A(\alpha) \rho_B(\beta) \quad \forall \alpha \in A\, , \beta \in B\, .$$
    \end{lem}
\begin{proof} We may assume that $\rho_A$ takes values in $\triang$. By assumption, for each $\beta \in B$, there exists a matrix $M_\beta\in\mathrm{GL}_{2}\C$ such that $\rho_A(\theta(\beta)^{-1}(\bullet))=M_\beta^{-1} \rho_A(\bullet)M_\beta$. By Lemma \ref{lemred}, we may assume $M_\beta\in \Aff$. If $\mathrm{Im}(\rho_A)\subset \triang$ is nonabelian, then it has trivial centralizer and the matrices $M_\beta\in \Aff$ are uniquely defined. Otherwise, we have $ \mathrm{Im}(\rho_A)\subset \left\{ \lambda \left( \begin{smallmatrix} 1 & \tau\\0&1\end{smallmatrix}\right)~\middle|~\lambda \in \C^*\, , \tau \in \C\right\}$ and the matrices $M_\beta$ are uniquely defined if we impose 
$M_\beta \in  \left\{  \left( \begin{smallmatrix} \mu & 0 \\0&1\end{smallmatrix}\right)~\middle|~\mu \in \C^*\right\}$. It is now straightforward to check that given these choices, the well-defined map $\beta \mapsto M_\beta$ is a morphism of groups.  \end{proof}

For a similar result for semisimple representations $\rho_A$ (of arbitrary rank), the group $B$, which in our case will be the fundamental group of a parameter space, might have to be modified, in order to take into account the non-unicity of the matrices $M_\beta$ due to possible permutations of irreducible components. 

\begin{prop}\label{propextrep}
Let $\rho_A \in \Hom(A, \GL)$ be semisimple. Let $  (T,t_0)$ be a smooth connected quasi-projective variety, and let $\theta\in \Hom(\pi_1(T,t_0),{\mathrm{Aut}(A)})$  such that $H:=\mathrm{Im}(\theta)$  stabilizes $[\rho_A]$.  Then there is an \'etale base change $p : (T',t_0') \to (T,t_0)$ and 
    a representation $\rho_B \in \Hom(\pi_1(T',t_0'), \GL)$ such that
$$  \rho_A (\theta(p_*\beta)^{-1} (\alpha) )= \rho_B(\beta)^{-1} \cdot \rho_A(\alpha) \cdot \rho_B(\beta) \quad \forall  \alpha\in A\, , \beta \in \pi_1(T',t_0')\, . $$
 \end{prop}
\begin{proof}
  Let $\rho_A = \bigoplus_{i\in I} \rho^i_A $ be a decomposition such that each $\rho^i_A$ is irreducible. The subgroup $$\bigcap_{i\in I} \mathrm{Stab}_{{\mathrm{Aut}(A)}}[\rho^i_A] \subset \mathrm{Stab}_{{\mathrm{Aut}(A)}}[\rho_A]\, , $$stabilizing the conjugacy class $[\rho^i_A]$ for each $i\in I$, is of finite index  (see for example \cite[Lemma $3$]{cousinisom}).  Hence the subgroup $\widetilde{H}:=H \cap_{i\in I} \mathrm{Stab}_{{\mathrm{Aut}(A)}}[\rho^i_A]$ is of finite index in $H$. Consider now the finite connected 
   unramified covering $\tilde{p}: (\widetilde{T},\tilde{t}_0)\to (T,t_0)$ characterized by $\tilde{p}_*\pi_1(\widetilde{T},\tilde{t}_0)=\theta^{-1}(\widetilde{H})$. Note that $\tilde{p}$ induces a structure of   smooth quasi-projective variety on $\widetilde{T}$. Since $\widetilde{H}$ stabilizes $[\rho^i_A]$, for every $h\in \widetilde{H}$ and every $i\in I$, there is a matrix $M^i_{h}\in \mathrm{GL}_{r_i}\C$ such that 
\begin{equation} \label{eqtrouverrhoB} (M^i_h)^{-1} \cdot \rho^i_A \cdot  M_h^i= [h]\cdot \rho^i_A\, .\end{equation}
Given $i$ and $h$, the choice of $M_h^i$ is unique up to an element of the centralizer of $\rho^i_A$. Since $\rho^i_A$ is irreducible, this centralizer is given by the set of scalar matrices. 
Denote by $\overline{M_h^i}\in  \mathrm{PGL}_{r_i}\C$ the projectivization of   $ {M_h^i}\in  \mathrm{GL}_{r_i}\C$.  Then $\overline{\rho_B}^i: \beta \mapsto \overline{M_{\theta_*\tilde{p}_*\beta}^i}$ is a well-defined element of $\Hom(\pi_1(\widetilde{T},\tilde{t}_0), \mathrm{PGL}_{r_i}\C)$. According to the Lifting Theorem \cite[Th. $3.1$]{MR3300949}, there exists a Zariski closed subset $\widetilde{Z}$ of $\widetilde{T}$ not containing $\tilde{t}_0$, a finite morphism of smooth quasi-projective varieties  $$p' : (\widetilde{T}',t_0') \to ( \widetilde{T}\setminus \widetilde{Z}, \tilde{t}_0)\, , $$ \'etale in a neighborhood $T'$ of $t_0'$, and a representation $\rho^i_B \in  \Hom(\pi_1(\widetilde{T}',t_0'), \mathrm{GL}_{r_i}\C)$  whose projectivization is ${p'}^*\overline{\rho_B}^i$. For a convenient choice of $p'$, this property is satisfied for all $i\in I$ at once. We obtain a representation $\rho_B:= \bigoplus_{i\in I} \rho^i_B  $ in $\Hom(\pi_1(T',t_0'), \mathrm{GL}_{r_i}\C)$ satisfying the required properties with respect to $p:=\tilde{p}\circ p'|_{T'}$
.  
 \end{proof}

 \subsection{Finiteness and algebraization}\label{SecFiniteness}\label{SecFinPreuve}

\begin{thmA1}\hypertarget{thA1}{}
Let $(C,D)$ be a stable $n$-pointed genus  $g$ curve as in Section \ref{SecSetup}. 
Let $\varphi : (\Sigma_g, Y^n) \stackrel{\sim}{\to} (C,D)$ be an orientation preserving homeomorphism.   
Let $(\initialBundle, \nabla_0)$ be an algebraic logarithmic connection over $C$ with polar divisor $D$ and denote by $[\rho_{\nabla_0}]\in \chign{\GL}$ its monodromy with respect to $\varphi$.  
Let $\mathcal{I}_{(C,\initialBundle, \nabla_0)}=(\mathcal{F}_{(C,D)}, \familyBundle, \nabla, \Psi)$ be an algebraic isomonodromic deformation of $(C,\initialBundle, \nabla_0)$ with parameter space $T$ as in Section \ref{SecDefDef}. Assume that 
\begin{itemize}
\item 
the classifying map $\class(\mathcal{F}) : T\to \Mgn$  is dominant (see Section \ref{Sec univ fam}).
\end{itemize}
Then 
\begin{itemize}
\item  the $\Gamma_{g,n}$-orbit of $[\rho_{\nabla_0}]$ in $\chign{\GL}$ is finite.
\end{itemize}
 \end{thmA1}
  
\begin{proof} The orbit $\Gamma_{g,n}\cdot [\rho_{\nabla_0}]$ does not depend on the choice of $\varphi$. Moreover, it is canonically identified, for any $t_1\in T$, with the orbit $\Gamma_{g,n}\cdot [\rho_{t_1}]$ of the monodromy of the connection $(\familyBundle, \nabla)$ restricted to the fiber over $t_1$ of the family $\mathcal{F}$. Since $\class(\mathcal{F} )$ is dominant we may assume, without loss of generality, that
 $\star:=\class(\mathcal{F} )(t_0)\in  \Mgn\setminus \Bgn$. Moreover, up to restricting $\mathcal{I}_{(C,\initialBundle, \nabla_0)}$ to a Zariski open neighborhood $\Delta$ of $t_0$ in $T$, we may assume that  
$\class(\mathcal{F} )(T)\, \cap\,   \Bgn\, = \emptyset.$ Notice that this property, as well as the assumption of $\class(\mathcal{F} )$ being dominant is not altered by finite covers and further excision of strict subvarieties not containing $t_0$. 
 According to Lemma~\ref{lemsection}, up to such a manipulation, we may assume that $\mathcal{F}_{(C,D)}=(\kappa : \familyCurve\to T, \mathcal{D}, t_0, \psi)$ admits a section $\sigma : T\to \familyCurve$ of $\kappa$ with values in $\familyCurve^0:=\familyCurve\setminus \mathcal{D}$   such that $\sigma(t_0)=\psi\circ \varphi(y_0)$. Denote by $\rho$ a representative of the monodromy representation of $(\familyBundle, \nabla)$ with respect to the identity such that the restriction of 
$\rho$ to the subgroup $(\psi\circ \varphi)_*\gf$ of $\pi_1(\familyCurve^0, \sigma(t_0))$, given by the inclusion of the central fiber, is identical to $(\psi\circ \varphi)_*\rho_{\nabla_0}$. Such a representative exists, as implies for example Theorem \ref{thmRH}. 
According to Proposition \ref{corSplit}, we  then have a semi-direct product decomposition  
 $$ \pi_1(\familyCurve^0,\sigma(t_0)) =(\psi\circ \varphi)_*\gf\rtimes_{\eta} \sigma_*\pi_1(T,t_0)\, ,$$
 where we have two different expressions for its structure morphism $\eta$, 
 proving that $${H}:= \tautgnplusun({\clF^\bullet}_*\pi_1(T,t_0))\subset \Gamma_{g,n+1} 
\, $$ acts on $\rho_{\nabla_0}\in \Hom(\gf, \GL)$ by conjugation
  .
 More precisely, for all $\alpha\in \gf$ and $[h]= \tautgnplusun({\clF^\bullet}_*\, \beta)\, \in {H}$, we have
 $$\rho_{\nabla_0}\left(\mathfrak{a}(h)(\alpha)\right)=\rho(\sigma_*\beta)\cdot \rho_{\nabla_0}(\alpha)\cdot \rho(\sigma_*\beta^{-1})$$
and in particular $[h^{-1}]\cdot [\rho_{\nabla_0} ]= [\rho_{\nabla_0} ].$ 
In other words, $H$ is a subgroup of the stabilizer of $[\rho_{\nabla_0}]$ in $\Gamma_{g, n+1}$. 
  By definition of the mapping class group action, we then have  $$\pi(H)\subset \mathrm{Stab}_{\Gamma_{g,n}}[\rho_{\nabla_0}]\, ,$$ where $\pi :\Gamma_{g, n+1}\to \MCG$ is the projection  forgetting the marking $y_0$.  Since the size of the orbit $\Gamma_{g,n}\cdot [\rho_{\nabla_0}]$ equals the index of $ \mathrm{Stab}_{\Gamma_{g,n}}[\rho_{\nabla_0}]$ in $\Gamma_{g,n}$, it now suffices to prove that $\pi(H)$ has finite index in $\Gamma_{g,n}$. 
Denote by $q:\MCG\to \MCG/K_{g,n}$ the quotient by the normal subgroup $K_{g,n}$, which, by Lemma \ref{lemme Kgn}, has order at most $2$. Hence for the indices, we have $$
[\MCG : \pi(H)]  \leq 2\cdot [\MCG/K_{g,n} : q(\pi(H))]\, .$$
We have a commutative diagram 
$$\xymatrix{\Gamma_{g,n+1}\ar[d]_{\pi}&&&& \pi_1(T,t_0)\ar[llll]_{\tautgnplusun\, \circ \, {\clF^\bullet}_*}\ar[d]^{\tautgn \, \circ\,  \clF_*}\\
\MCG \ar[rrrr]^{q}&&&& \MCG/K_{g,n}\, ,
 }$$
 where $\clF: T\to \Mgn \setminus \Bgn$ denotes the corestriction of $\class(\mathcal{F})$. 
On the other hand, by  the dominance assumption and  \cite[Lemma $4.19$]{MR1841091},  the subgroup   $\clF_*\, \pi_1(T,t_0)$ of $\pi_1\left(\Mgn \setminus \Bgn,\star \right)$ is of finite index.
In particular, since the tautological morphism $\tautgn  : \pi_1\left(\Mgn \setminus \Bgn,\star\right) \twoheadrightarrow \MCG/K_{g,n}$ is onto, the subgroup $q(\pi(H))=\tautgn (\clF_*\pi_1(T,t_0))$ of $\MCG /K_{g,n}$ has finite index as well.  
 \end{proof}


\begin{thmA2}\hypertarget{thA2}{}
Let $\mathcal{F}_{(C,D)}=(\kappa : \familyCurve\to T, \mathcal{D}, t_0, \psi)$ be an algebraic family of stable $n$-pointed genus $g$ curves with central fiber $(C,D)$ as in Section \ref{SecSetup}. 
 Let $(\initialBundle, \nabla_0)$ be an algebraic logarithmic connection over $C$ with polar divisor $D$ and denote by $[\rho_{\nabla_0}]\in \chign{\GL}$ its monodromy  with respect to  an orientation preserving homeomorphism $\varphi : (\Sigma_g, Y^n) \stackrel{\sim}{\to} (C,D)$.  
Assume that 
\begin{itemize}
\item $(\initialBundle, \nabla_0)$ is mild,    \vspace{.1cm}
\item $r=2$ or $\rho_{\nabla_0}$ is semisimple, and  \vspace{.1cm}
\item the $\Gamma_{g,n}$-orbit of $[\rho_{\nabla_0}]$ in $\chign{\GL}$ is finite. 
\end{itemize}
Then there are 
\begin{itemize}
\item an \'etale base change $p: (T',t_0')\to (T,t_0)$ and  \vspace{.1cm}
\item a 
flat algebraic logarithmic connection $(\familyBundle,\nabla)$ over $\familyCurve':=p^*\familyCurve$ with polar divisor $p^*\mathcal{D}$, 
\end{itemize}
 such that $\psi^*(\familyBundle,\nabla)|_{\familyCurve'_{t_0'}}$ is isomorphic to $(\initialBundle, \nabla_0)$. 
\end{thmA2}

\begin{proof} Since $(C,D)$ is stable by assumption, it only admits a finite number of automorphisms. Let $x_0\in C\setminus D$ be a point fixed by no automorphism other than the identity. Up to isotopy, we may assume $\varphi(y_0)=x_0$. Let $\mathbf{x}^\bullet$ be the labelling of $D^\bullet=D+x_0$ induced by $\varphi$.
By construction, we have $\star:=[C,\mathbf{x}^\bullet]\in \Mgnplusun \setminus \Bgnplusun$. 
Up to an  \'etale base change, we may assume, by Lemma \ref{lemsection}, that there is a section   $\sigma : T\to \familyCurve$ of $\kappa$ with values in $\familyCurve^0:=\familyCurve\setminus \mathcal{D}$ and such that $\sigma(t_0)=\psi(x_0)$.   With the notation of Section \ref{SecNotBul}, we may consider the family of $n+1$-pointed genus $g$ curves  $\mathcal{F}^\bullet_{(C,D^\bullet)}=(\kappa : \familyCurve\to T, \mathcal{D}+\sigma(T), t_0, \psi)$.  
According to Proposition \ref{corSplit}, we have a semi-direct product decomposition  
 $ \pi_1(\familyCurve\setminus \mathcal{D},\sigma(t_0)) =(\psi\circ \varphi)_*\gf\rtimes_{\eta} \sigma_*\pi_1(T,t_0),$
where
 $$\eta(\sigma_*\beta)((\psi\circ\varphi )_*\alpha) =\sigma_*\beta \cdot (\psi\circ \varphi)_*\alpha \cdot \sigma_*\beta^{-1} = (\psi\circ \varphi)_*\mathfrak{a}(\theta_*\beta)(\alpha)\, $$  
 and $\theta : = \tautgnplusun \circ {\clF^\bullet}_* : \pi_1(T,t_0) \to \Gamma_{g,n+1}.$

Since the $\Gamma_{g,n+1}$-orbit of $[\rho_{\nabla_0}]$ in $\chign{\GL}$ is finite,    the stabilizer $$H:=\mathrm{Stab}_{\Gamma_{g,n+1}}[\rho_{\nabla_0}]$$ of the conjugacy class of $\rho_{\nabla_0}$ under the action of $\Gamma_{g,n+1}$  has finite index in $\Gamma_{g,n+1}$. Since the  tautological morphism is onto, the subgroup 
$\tautgnplusun^{-1}(H)$ of $\pi_1(\Mgnplusun \setminus \Bgnplusun, \star)$ then has also finite index. In particular, there is a finite connected \'etale cover
$q : (U,u_0)\to (\Mgnplusun \setminus \Bgnplusun, \star)$ such that 
$\pi_1(U,u_0)= \tautgnplusun^{-1}(H)$. 
Now consider the  fibered product 
\[\xymatrix{(T',t_0')\ar[r]^{p}\ar[d]&(T,t_0)\ar[d]^{\mathrm{class}(\mathcal{F}^\bullet)}\\(U,u_0)\ar[r]^{\hskip-15pt q}&(\Mgnplusun,\star).}
\] 
We denote the pullback family of curves by $\mathcal{F}_{(C,D^\bullet)}'=(\kappa' : \familyCurve'\to T', \mathcal{D}' +\sigma'(T'), t'_0, \psi'):=p^*\mathcal{F}^\bullet_{(C,D^\bullet)}$. We further denote  $\clF'=\clF^\bullet \circ p$, which is the corestriction of $\class (\mathcal{F}')$
.
By construction, the morphism $\theta' : =\theta\circ p= \tautgnplusun \circ {\clF'}_* : \pi_1(T',t_0') \to \Gamma_{g,n+1}$
takes values in~$H$. 

Again up to an \'etale base change of $(T',t_0')$, by Proposition  \ref{propextrep} and Lemma \ref{lemextrep}, there is a representation 
$\rho_B\in \Hom(\pi_1(T',t_0'),\GL)$ such that  for all $\beta \in \pi_1(T',t_0')$,  $\alpha \in \gf$, we have
$$\left([\theta'_* \beta]^{-1}\cdot \rho_{\nabla_0}\right) (\alpha ) = \rho_B(\beta) \cdot \rho_{\nabla_0}(\alpha) \cdot \rho_B(\beta^{-1}) \, . $$
Since by definition $\left([\theta'_*\beta]^{-1}\cdot \rho_{\nabla_0}\right) (\alpha )= \rho_{\nabla_0}(\mathfrak{a}(\theta'_*\beta)(\alpha))$, we obtain a well-defined representation 
$$\rho : \left\{\begin{array}{rcl} \pi_1(\familyCurve'\setminus \mathcal{D}',\sigma'(t_0'))  & \to & \GL\vspace{.1cm}\\
(\psi'\circ\varphi )_*\alpha \cdot \sigma'_*\beta & \mapsto & \rho_{\nabla_0}(\alpha) \cdot \rho_B(\beta)\end{array} \right.\, $$
(see Lemma \ref{LemExt}) with respect to the semi-direct product decomposition $ \pi_1(\familyCurve'\setminus \mathcal{D}',\sigma'(t_0')) =(\psi'\circ \varphi)_*\gf\rtimes_{\eta} \sigma'_*\pi_1(T',t_0')$. By construction, $\rho$ extends $ \rho_{\nabla_0}$. 
We  conclude by the logarithmic Riemann-Hilbert correspondence (see Theorem \ref{thmRH}). 
 \end{proof}

 \begin{proof}[Proof of Theorem  \ref{algebrization thm}]

 Let us first prove the implication  $(\ref{algitem alg})\Rightarrow (\ref{algitem finite})$. 
 Let $\mathcal{I}^{\mathrm{univ, alg}}_{(C,\initialBundle, \nabla_0)}=$ $(\mathcal{F}^{\mathrm{Kur}}_{(C,D)}, \familyBundle, $ $\nabla, \Psi)$ be an algebraic universal isomonodromic deformation of $(C,\initialBundle, \nabla_0)$ as in Section \ref{SecDefDef}. Then by definition, the family $\mathcal{F}^{\mathrm{Kur}}_{(C,D)}$ is Kuranishi. In particular,  the classifying map
  $\class(\mathcal{F}^{\mathrm{Kur}} ) : T\to \Mgn$  is dominant. Then by \hyperlink{thA1}{Theorem A1}, the $\Gamma_{g,n}$-orbit of $[\rho_{\nabla_0}]$ in $\chign{\GL}$ is finite. 

  Let us now prove the implication  $(\ref{algitem finite})\Rightarrow (\ref{algitem alg})$. Let $\mathcal{F}^{\mathrm{Kur}}_{(C,D)}=(\kappa : \familyCurve\to T, \mathcal{D}, t_0, \psi)$ be any algebraic Kuranishi family with central fiber $(C,D)$ as in Section \ref{SecSetup}. Note that such a family exists since $(C,D)$ is stable, and that it remains Kuranishi after pullback \textit{via} an \'etale base change. Up to such a manipulation, according to \hyperlink{thA2}{Theorem A2}, the family $\mathcal{F}^{\mathrm{Kur}}_{(C,D)}$ can be endowed with a flat algebraic logarithmic connection $(\familyBundle,\nabla)$ over $\familyCurve$ with polar divisor $\mathcal{D}$ such that there is an isomorphism
  $\Psi : (\initialBundle,\nabla_0)\to (\familyBundle, \nabla)|_{\familyCurve_{t_0}}$ commuting with $\psi$ via the natural projections to $(C,D)$ and $(\familyCurve_{t_0}, \mathcal{D}_{{t_0}})$ respectively. Now $\mathcal{I}^{\mathrm{univ, alg}}_{(C,\initialBundle, \nabla_0)}:=(\mathcal{F}^{\mathrm{Kur}}_{(C,D)}, \familyBundle, \nabla, \Psi)$ defines an algebraic universal isomonodromic deformation of $(C,\initialBundle,\nabla_0)$ (see Section \ref{SecDefDef}).
\end{proof}


 \part{Dynamics}\label{partDyn}
  
  \section{Effective description of the mapping class group action}\label{Sec explicit}
In this section we describe the action of $\fMCG$ on $\gf$ in terms of specified generators for both groups.
\subsection{Presentation of the fundamental group}
To give an effective description of $\gf$ and how $\fMCG$ acts, we will assume that $\Sigma_g$ is the subsurface of genus $g$ of $\R^3$ depicted in Figure $\ref{figgens 1}$. On this surface we also depicted, in gray, an embedded closed disk $\bar{\Delta}\subset \Sigma_g$, we will denote $\Delta$ its interior. We fix $n$ and we consider a subset $Y^n=\{y_1, \ldots, y_n\} \subset \Delta$ of cardinality $n$, as well as a point $y_{0}\in \bar{\Delta}\setminus \Delta$. 
We have
$$\pi_1(\Sigma_g\setminus \Delta, y_0)=\left\langle \alpha_1, \beta_1, \ldots, \alpha_g, \beta_g, \delta ~\middle|~ [\alpha_1, \beta_1]\cdots[\alpha_g, \beta_g]=\delta^{-1} \right\rangle ,$$
where the mentioned generators correspond to the loops in Figure $\ref{figgens 1}$. Note that $\delta$ runs over the boundary of $\bar{\Delta}$.
  \begin{figure}[htbp] 
\begin{centering} {\large \scalebox{.4}{\input{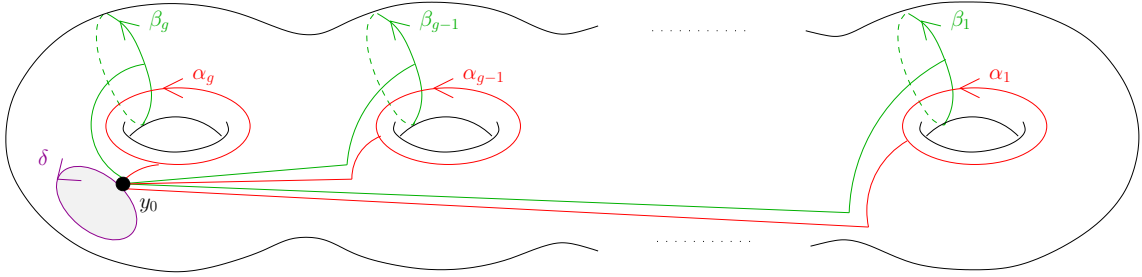}}}\\
 \end{centering}
 \caption{Preferred elements of the fundamental group, I \label{figgens 1}}
 \end{figure}
 
The loops in  Figure $\ref{figgens 2}$ correspond to the following presentation.
$$\pi_1(\bar{\Delta}\setminus Y^n, y_{0})=\left\langle \gamma_1, \ldots, \gamma_n, \delta ~\middle|~ \gamma_1\cdots\gamma_n=\delta \right\rangle .$$
 \vspace{-1cm}\\ \begin{figure}[htbp]
  \begin{centering} {\large \scalebox{.4}{\input{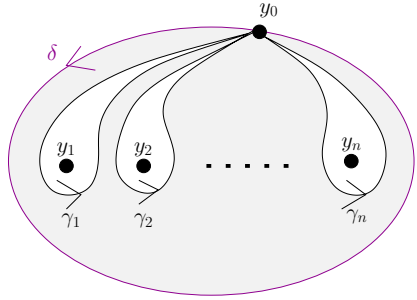}}}\\
 \end{centering}\caption{Preferred elements of the fundamental group, II\label{figgens 2}}
 \end{figure}
 
 By the Van Kampen theorem, we have
 \[\begin{array}{rcl}\gf&=&\pi_1(\Sigma\setminus \Delta, y_{0})*_\delta \pi_1(\bar{\Delta}\setminus Y^n, y_{0})\vspace{.2cm}\\
&=&\left\langle\alpha_1, \beta_1, \ldots, \alpha_g, \beta_g,\gamma_1,\ldots,\gamma_n ~\vert~ \gamma_1\cdots\gamma_n=\left([\alpha_1, \beta_1]\cdots[\alpha_g, \beta_g]\right)^{-1}\right\rangle.
 \end{array}
 \]
 In the sequel, writing ``the generators" of $\gf$, we will refer to the above $$(\alpha_i)_{i\in \intint{1}{g} }\, , \, (\beta_i)_{i\in \intint{1}{g} }\, , \, (\gamma_j)_{j\in \intint{1}{n} }\, .$$
 
 \subsection{Mapping class group generators}

We define $\MCGg$ to be the mapping class group of orientation preserving homeomorphisms of $\Sigma\setminus \Delta$ that restrict to the identity on $\partial \Delta$. Continuating such homeomorphisms by the identity on $\Delta$, we get a  morphism 
$$\varphi_{g}: \MCGg \to \fMCGplus\, .$$
After Lickorish \cite{MR0171269} (see also \cite[Th. $4.13$]{MR2850125}), the group $\MCGg$ is generated by the (right) Dehn-twists along the  loops  $\tau_1, \ldots,\tau_{3g-1}$ represented in Figure \ref{DehntwistPic}. \\

 \begin{figure}[htbp]
  \begin{centering} {\large \scalebox{.4}{\input{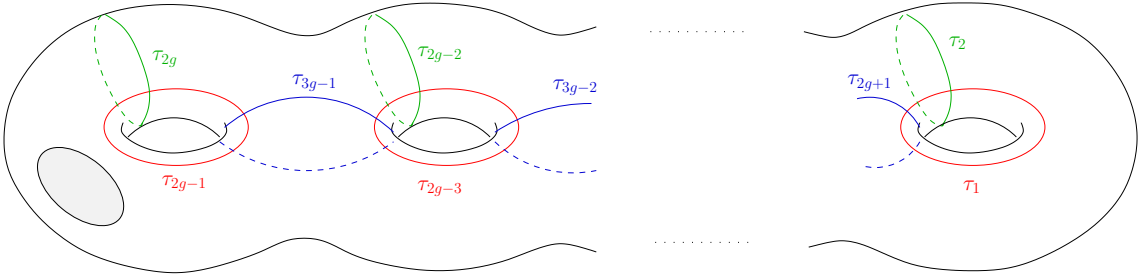}}} \end{centering}  
  \caption{Dehn-twists}\label{DehntwistPic}
 \end{figure}
 
 A right Dehn twist acts on paths which cross the corresponding Dehn curve as depicted in Figure \ref{FigDehnaction}. This action can be summarized as ``a path crossing the Dehn curve has to turn right''. A left Dehn twist is the inverse of a right Dehn twist.
\begin{figure}[htbp]
  \begin{centering} {\large \scalebox{.4}{\input{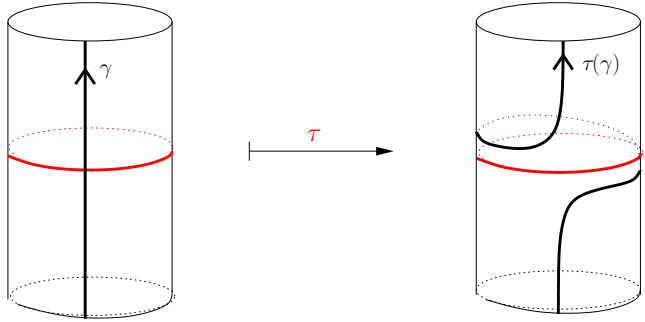}}} \end{centering}  
  \caption{Dehn-twist action }\label{FigDehnaction}
 \end{figure}
 
One can now easily check the following.   
\begin{lem}[Dehn-twists]
The action of the Dehn twists above on the fundamental group $\pi_1(\Sigma_g\setminus \Delta, y_{0})$ is given in Table~$\ref{DehnAction1}$, where the right hand side lists only the effect on those generators that are not fixed by the considered twist. Here for $\tau_{2k-1}$ we give the formula for the left Dehn twist. The other generators all correspond to right Dehn twists. Moreover, for $k\in\intint{1}{g-1}$, the element $\Theta_k$ described  in Table~$\ref{DehnAction1}$ is fixed by $\tau_{2g+k}$.
\begin{table}[htbp]
$\begin{array}{|ll| lcl |}
\hline
\tau_{2k}\,, &k\in\intint{1}{g} & \alpha_k & \mapsto& \alpha_k \beta_k\\
\hline
\tau_{2k-1}\,, & k\in\intint{1}{g} & \beta_k & \mapsto& \beta_k\alpha_k\\
\hline
\tau_{2g+k}\,,& k\in \intint{1}{g-1}&\alpha_{k+1}&\mapsto& \Theta_k^{-1} \alpha_{k+1}  \\&&  \alpha_{k} &\mapsto& \alpha_{k}\Theta_k\\
&&\beta_{k}&\mapsto& \Theta_k^{-1}\beta_{k}\Theta_k \\
&&\textrm{where }\Theta_k&=&\alpha_{k+1}\beta_{k+1}^{-1}\alpha_{k+1}^{-1}\beta_{k}\\
\hline
\end{array}
$\\$ $
\\
\caption{\label{DehnAction1}}
\end{table}
 \end{lem}
 
  On the other hand, one can define  the mapping class group of orientation preserving homeomorphisms of $\bar{\Delta}$ that preserve the set $Y^n$ and restrict to the identity on $\partial \Delta$. It is classically called the braid group on $n$ strands and denoted $\Bn$. Continuating such homeomorphisms by the identity on the complement of $\Delta$ in $\Sigma_g$, we get a morphism 
$$\varphi_{0}: \Bn \to \fMCGplus.$$
 After Artin \cite{MR3069440}, the group $\Bn$ is generated by half-twists  $\sigma_1, \ldots, \sigma_{n-1}$, whose action is depicted in Figure \ref{DehnhalftwistPic}.
 
    \begin{figure}[htbp]
  \begin{centering} {\large \scalebox{.4}{\input{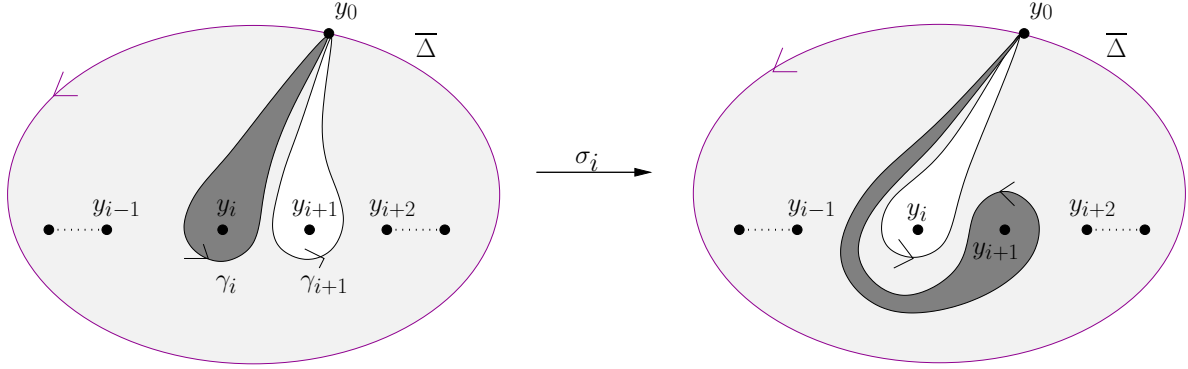}}\\
  }\end{centering}
  \caption{half-twists}\label{DehnhalftwistPic}
 \end{figure}
  
\begin{lem}[half-twists]
The action of $\Bn=\langle \sigma_1, \ldots, \sigma_{n-1}\rangle $ on the fundamental group $\pi_1(\bar{\Delta}\setminus Y^n, y_{0})$ is described in Table~$\ref{DehnAction2}$, where we only indicate the non-trivial actions on the generators. Moreover, Table~$\ref{DehnAction2}$ indicates the action  of $\sigma_{cycl}:=\sigma_{n-1}\circ \cdots \circ \sigma_{1} \in \Bn$ and some of its powers. 

 \begin{table}[htbp]$\begin{array}{|ll| lcl l|}
\hline
\sigma_{k}\,, &k\in\intint{1}{n-1} &\gamma_k & \mapsto& \gamma_k\gamma_{k+1}\gamma_k^{-1}&\\
 &&\gamma_{k+1}&\mapsto&\gamma_k&\\ \hline
\sigma_{cycl}  & & \gamma_1&\mapsto&\delta \gamma_n \delta^{-1}&\\
&  & \gamma_i&\mapsto&\gamma_{i-1}\,,&i\in \intint{2}{n}\\
\hline
\sigma_{cycl}^k \,,  &k\in \intint{1}{n}& \gamma_i& \mapsto & \delta \gamma_{n+i-k} \delta^{-1},&i\in \intint{1}{k}\\
&  & \gamma_j& \mapsto&\gamma_{j-k}\,,&j\in \intint{k+1}{n}\\
\hline

\end{array}$  \\$ $
\\  \caption{\label{DehnAction2}\label{tablecycl}}\end{table}
 \end{lem}

 \begin{rem}  Note that $\sigma_{cycl}$ is almost a cyclic permutation of the generators of $\pi_1(\bar{\Delta}\setminus Y^n, y_{0})$. More precisely,  it acts as such on the representations $\rho$ that satisfy $\rho(\delta)=\mathrm{id}$, \textit{e.g.} representations with abelian image.
  \end{rem}

By construction, the subgroups $\varphi_0( \Bn)$ and  $\varphi_g( \MCGg)$ of $\fMCGplus$
  commute, and we have a morphism 
 $$ \Bn\times  \MCGg \stackrel{\varphi_0\times\varphi_g}{\longrightarrow}  \fMCGplus \, .$$
 Composing with the canonical map  $\pi : \fMCGplus\rightarrow \fMCG$ (forgetting that $y_0$ is fixed) yields a morphism $\Bn\times  \MCGg\rightarrow \fMCG$, which is {not} 
 surjective. In order to generate the whole mapping class group $\fMCG$, it suffices to add $\min(0,n-1)$ Dehn twists, namely the ones corresponding to the loops $\tau_{3g},\ldots,\tau_{3g+n-2}$ of Figure \ref{FigMixing} (see \cite[Sec. $4.4.4$]{MR2850125}). We  call them \emph{mixing twists}.
 \begin{figure}[htbp]
  \begin{centering} {\large \scalebox{.5}{\input{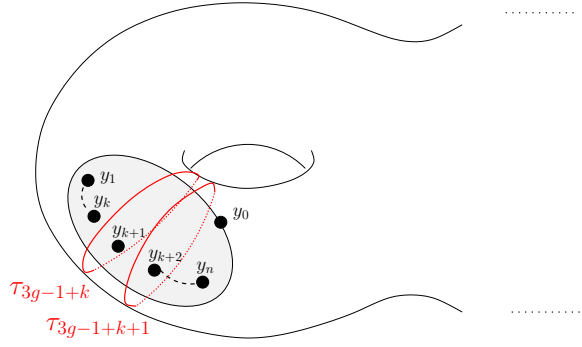}}} \end{centering}  
  \caption{Mixing twists \label{FigMixing}}
 \end{figure}

\begin{lem}[Mixing twists]
The action of the (right) mixing twists  $\tau_{3g},\ldots,\tau_{3g+n-2}$ on the fundamental group $\gf$ 
 is described in Table~$\ref{DehnAction3}$, where we only indicate the non-trivial actions on the generators. Moreover, for $k\in\intint{1}{n-1}$,
 the element $\Xi_k$ described there is fixed by $\tau_{3g-1+k}$.
 
 \begin{table}[htpb]$$\begin{array}{|ll|crll|}
\hline
\tau_{3g-1+k}\,,& k\in \intint{1}{n-1}   & \alpha_g & \mapsto& \alpha_g \Xi_k& \\
&&\beta_g & \mapsto& \Xi_k^{-1}\beta_g\Xi_k&\\
&&\gamma_i & \mapsto& \Xi_k^{-1} \gamma_i \Xi_k\,,&i\in \intint{1}{k}\\
&   && \textrm{for}&\Xi_k= (\gamma_1\ldots \gamma_k)^{-1}\beta_g&\\

\hline
\end{array}$$
\\  \caption{ \label{DehnAction3}}\end{table}
 \end{lem}

The twists, mixing twists and braids we introduced all fix $y_{0}$. We denote by $\liftMCG$ the subgroup of $\fMCGplus$ they generate. If $g=0$, then we have $\liftMCG = \Bn$. We are interested in the case $g>0$, where we have
$$\liftMCG := \left\langle \tau_i \, , \sigma_j~\middle|~i\in \intint{1}{3g-1+ \min(0, n-1)}\, , \, j\in \intint{1}{n-1}\right\rangle \, . $$ As mentioned, the image of $\liftMCG$ under $\pi : \fMCGplus\rightarrow \fMCG$ is $\fMCG$. 
\begin{rem}\label{MixingTwistetDelta} We did not call $\delta = \gamma_1\cdots\gamma_n=\left([\alpha_1, \beta_1]\cdots[\alpha_g, \beta_g]\right)^{-1}$ a generator of the fundamental group. It will  nevertheless be useful to notice that among our preferred generators of $\liftMCG$, only the mixing twists act non trivially on $\delta$. More precisely, for $k\in \intint{1}{n-1}$ we have 
$$\tau_{3g-1+k}(\delta)= [\Xi_k^{-1}, \beta_g]\delta\, .$$
\end{rem}
 
\section{Affine representations with finite orbit}\label{Sec Aff case}
We have now established an  explicit description of the full mapping class group action on $\gf$, which is resumed in Table \ref{A}. This description at hand, we will now classify affine representations $\rho \in \mathrm{Hom}(\gf,\mathrm{Aff}(\mathbb{C}))$ with finite orbit $\fMCG\cdot[\rho]$ in $\chign{\Aff}$ for $g>0$: 
\begin{itemize}
\item We establish that for those representations $\rho \in \mathrm{Hom}(\gf,\mathrm{Aff}(\mathbb{C}))$ such that the group $\mathrm{Im}(\rho) $ is abelian, 
the orbit $\fMCG\cdot[\rho]$ is finite if and only if $\mathrm{Im}(\rho) $ is finite (see Proposition $\ref{abcase}$).\vspace{.2cm}
\item We then consider representations $\rho \in \mathrm{Hom}(\gf,\mathrm{Aff}(\mathbb{C}))$ such that the group $\mathrm{Im}(\rho) $ is not abelian. We classify all finite orbits in this case in three steps.\vspace{.1cm}
 \begin{itemize}
\item We  give a necessary condition for the finiteness of $\fMCG\cdot[\rho]$ in Lemma \ref{Prep}.\vspace{.1cm}
\item We prove that in the genus one case, this necessary condition is also sufficient (see Proposition $\ref{SpecialCaseg1c}$). \vspace{.1cm}
\item We prove that in the higher genus case, this necessary condition can be enforced (see Lemma \ref{Prep0}), and this latter necessary condition cannot hold for every conjugacy class $[\rho']\in \fMCG\cdot[\rho]$. We conclude that in the higher genus case, there are no conjugacy classes of non-abelian$\mathrm{Aff}(\mathbb{C})$-representations with finite orbit under $\fMCG$ (see Proposition \ref{NABCaseg2}).
\end{itemize}
\end{itemize} 

The group $\mathrm{Aff}(\mathbb{C})=\left\{(a_{ij})\in \mathrm{GL}_{2}(\mathbb{C})~\middle|~a_{21}=0\, , a_{22}=1 \right\}$  identifies with the group $\{z\mapsto az+b~\vert~ a\in\C^*,b\in \C\}$ of affine transformations of $\C$. For shortness, its elements will be denoted as polynomials $az+b$.  Our explicit calculations are easier to check with the following formulas in mind.
$$\fbox{$\begin{array}{rcl}
(\lambda z)\circ (a z+b)\circ (\lambda z)^{-1}& = &a z+\lambda b\\
(z+c)\circ (az+b)\circ (z+c)^{-1}& = &a z+b-c(a-1)\\
{[}\lambda z+c, a z+b{]}   & = &z-c(a-1)+(\lambda-1) b 
\end{array}$}$$
Also, recall that by definition, for all $\tau \in \liftMCG~,~\rho\in \mathrm{Hom}(\gf,\mathrm{Aff}(\mathbb{C}))$ and $\alpha \in \gf$, we have
$$\fbox{$(\tau\cdot\rho)(\alpha)=\rho(\tau^{-1}_*\alpha) \, .
$}$$

\subsection{Abelian case}
\begin{lem}[Finding a non-trivial subgroup] \label{lemnontriv}  Let $g>0$, $n\in \N$. Let $G$ be a group with identity element $\mathrm{id}$ and let
 $\rho  :  \gf   \to G$ be a representation. Assume that for any $\rho' \in  \liftMCG\cdot\rho$, we have $$\rho'(\alpha_g) = \mathrm{id}\, . $$   Then $\rho$ is the trivial representation, \textit{i.e.} $\mathrm{Im}(\rho)  =\{\mathrm{id}\}$. 
\end{lem}
\begin{proof} To each element $\rho' \in \liftMCG\cdot \rho$, we associate the following two groups: $$R':=\langle \rho'(\alpha_g), \rho'(\beta_g), \ldots, \rho'(\alpha_1), \rho'(\beta_1)\rangle \, , \quad S':=\langle \rho'(\gamma_1),  \ldots, \rho'(\gamma_n)\rangle.$$

\begin{itemize} 
\item First step: \emph{for any $\rho' \in  \liftMCG\cdot\rho$, the associated group $R'$ is trivial.}\\
For $k\in \intint{1}{g}$, define the following property, which we shall denote $H(k)$:
$$\hspace{1cm} \fbox{ \emph{For any $\rho' \in  \liftMCG\cdot\rho$, the group $R_k':=\langle \rho'(\alpha_g), \rho'(\beta_g), \ldots, \rho'(\alpha_k), \rho'(\beta_k)\rangle$ is trivial.}}$$
Reasoning by decreasing induction, let us first prove that our assumption implies $H(g)$. 
Consider $\tau:=\tau_{2g}^{-1}$ and $\rho'=\tau\cdot\rho$. Then $\rho'(\alpha_g)=\rho(\alpha_g\beta_g)=\rho(\beta_g)$.  We  have $\rho'(\alpha_g)=\rho(\alpha_g)=\mathrm{id}$, hence $\rho(\beta_g)=\rho(\alpha_g)=\mathrm{id}$.
However, $\rho$ satisfies the assumption of the statement if and only if any $\tilde \rho \in  \liftMCG\cdot\rho$ does. 
Hence we have $H(g)$.\\
Let now $\rho$ be a representation satisfying $H(k)$. In particular, we have $$ \rho(\alpha_i)=\rho(\beta_i)=\mathrm{id}\quad \forall i\in \intint{k}{g}\, .$$ 
For $\rho'=\tau\cdot\rho$, with $\tau=\tau_{2g+k-1}^{-1}$ we have
$
 \rho'(\alpha_{k})= \rho( \beta_{k-1}^{-1} \alpha_k \beta_{k})= \rho( \beta_{k-1}) ^{-1}
$. \\
 For $\rho'=\tau\cdot\rho$, with $\tau=(\tau_{2k-3}\circ \tau_{2g+k-1})^{-1}$, we have 
$
 \rho'(\alpha_{k})= \rho(\beta_{k-1}\alpha_{k-1}) ^{-1}
$. 
\\ Hence $\rho$ satisfying $H(k)$ implies $$ \rho(\alpha_i)=\rho(\beta_i)=\mathrm{id}\quad \forall i\in \intint{k-1}{g}\, .$$  Yet again $\rho$ satisfies $H(k)$ if and only if any $\tilde \rho  \in  \liftMCG\cdot\rho$ does.  This yields $H(k-1)$. We conclude by noticing $R'=R_1'$.
\vspace{.2cm}
\item Second step: \emph{for any $\rho' \in  \liftMCG\cdot\rho$, the associated group $S'$ is trivial.}\\
If $n=0$ or $n=1$, there is nothing to prove. Assume $n>1$.
We have already proven that $R'$ is trivial for any $\rho' \in  \liftMCG\cdot\rho$. In particular $\rho'(\delta)=\mathrm{id}$.  Considering, for $i\in \intint{1}{n}$,  the action of $\tau=(\sigma_{cycl}^{n-i} \circ \tau_{3g+n-2})^{-1}$  on $\alpha_g$  then shows that  for $\rho'=\tau\cdot\rho$ we have $\mathrm{id}=\rho'(\alpha_g)=\rho(\gamma_i)$ (see Table~\ref{tablecycl} page~\pageref{tablecycl}). Hence
$$\langle \rho(\gamma_1),  \ldots, \rho(\gamma_n)\rangle =\{\mathrm{id}\}\, .$$
   Since the assertion is $\liftMCG$-invariant, we have proven that $S'$ is trivial for any $\rho' \in  \liftMCG\cdot\rho$.  \vspace{-.2cm} \end{itemize}
   We conclude that $\mathrm{Im}(\rho)=\mathrm{Im}(\rho')=\langle S',R'\rangle =\{\mathrm{id}\}$. 
\end{proof}

\begin{prop}[Abelian case] \label{abcase} Let $g>0$. Let $\rho  : \gf \to \mathrm{Aff}(\mathbb{C})$ be a representation such that the group $\mathrm{Im}(\rho)$ is abelian. Then the orbit of the conjugacy class $[\rho]$ under the action of $\fMCG$ is finite if and only if $\mathrm{Im}(\rho)$ is finite. 
\end{prop}

\begin{proof} If  $\mathrm{Im}(\rho)$ is finite, then the orbit $\liftMCG\cdot \rho$ is finite. \textit{A fortiori},  the orbit $\liftMCG \cdot [\rho]$ is finite. \\Assume now that $\rho$ is abelian and the orbit of $[\rho]$ is finite. Since $\mathrm{Im}(\rho)$ is an abelian subgroup of $\mathrm{Aff}(\mathbb{C})$ it is, up to conjugation, either a non-trivial subgroup of the translation group $$\{z\mapsto z+c ~|~c\in \mathbb{C}\}\subset \mathrm{Aff}(\mathbb{C})\, , $$ or it is a subgroup of the linear group
$$\{z\mapsto \lambda z ~|~\lambda \in \mathbb{C}^*\}\subset \mathrm{Aff}(\mathbb{C})\, .$$
\begin{itemize} 
\item \emph{$\mathrm{Im}(\rho)$ cannot be a non-trivial translation group.} \\
Indeed, if it would be the case, by Lemma \ref{lemnontriv}, we might assume $\rho(\alpha_g)\neq \mathrm{id}$. Up to conjugation, we would then have 
$$\rho \left(\begin{array}{l}\alpha_g\\
\beta_g 
\end{array}\right)=\left(\begin{array}{l}
 z +1\\
z+c\end{array}\right)$$
for a certain $c\in \C$. Considering the action of $\tau^{-m}$ with  $\tau :=\tau_{2g-1}$ :
$$\tau^{-m}\cdot \rho \left(\begin{array}{l}\alpha_g\\
\beta_g 
\end{array}\right)=\rho \left(\begin{array}{l}\alpha_g\\
\beta_g \alpha_g^m
\end{array}\right)=\left(\begin{array}{l}
 z +1\\
z+c+m\end{array}\right),$$
we would deduce that, for $m\neq m'$, the conjugacy classes of $\tau^m\cdot \rho$ and $\tau^{m'}\cdot \rho$ are distinct. Hence $\liftMCG\cdot [\rho]$ would be infinite, yielding a contradiction.
\vspace{.2cm}
\item \emph{If $\mathrm{Im}(\rho)$ is a subgroup of the linear group, then it is finite.}\\
Note that two distinct linear representations are not conjugated.
For any $i\in \intint{1}{g}$, finiteness of the orbit under $\langle\tau_{2i}\rangle$ yields that $\rho(\beta_i)$ is torsion. Similarly, considering $\langle\tau_{2i-1}\rangle$ yields that
 $\rho(\alpha_i)$ is torsion for all $i\in \intint{1}{g}$. For $j\in \intint{1}{n-1}$, finiteness  of the orbit under $\langle\tau_{3g-1+j}\rangle$ yields that $\rho(\gamma_1\ldots\gamma_j)$ is torsion. Consequently,   $\gamma_j$ is torsion for all $j\in \intint{1}{n-1}$.
 Hence $$\mathrm{Im}(\rho)=\left\langle \rho(\alpha_i),\rho(\beta_i),\rho(\gamma_j)~\middle|~i\in \intint{1}{g}\, ,~ j\in \intint{1}{n-1}\right\rangle$$ is an abelian group generated by finitely many torsion elements, whence the conclusion.
\end{itemize}
\end{proof}
\subsection{Preparation lemmata}

\begin{lem}[Finding a non-abeliansubgroup]\label{normalize} Let $g>0$.
Let $\rho  :  \gf \to \mathrm{Aff}(\mathbb{C})$ be a representation.  Assume that for any $\rho' \in  \liftMCG\cdot\rho$, the subgroup $$\langle \rho'(\alpha_g), \rho'(\beta_g)\rangle$$ of $\mathrm{Im}(\rho)$ is abelian.  Then $\rho$  is  an abelian representation, \textit{i.e.} $\mathrm{Im}(\rho)$ is abelian. 
\end{lem}

\begin{proof} We follow the same proof scheme we used for Lemma~\ref{lemnontriv}.
For any $k \in \intint{1}{g}$, to any $\rho' \in  \liftMCG\cdot\rho$ we may associate the following groups: $$R'_k:=\langle \rho'(\alpha_g), \rho'(\beta_g), \ldots, \rho'(\alpha_k), \rho'(\beta_k)\rangle \, , \quad S':=\langle \rho'(\gamma_1),  \ldots, \rho'(\gamma_n)\rangle\, .$$ 
\begin{itemize} 
\item First step: \emph{For any $\rho' \in  \liftMCG\cdot\rho$, the group $R_g'$ is contained in the center of $R_1'$.}\\
For $k \in \intint{1}{g}$, define the following property.\vspace{.2cm}\\
$H(k) :~\fbox{ \emph{For any $\rho' \in  \liftMCG\cdot\rho$, the group $R_g'$ is a subgroup of the center of $R_k'$.}}$
\vspace{.2cm}
\\
 By assumption, we have $H(g)$. Reasoning by decreasing induction, assume now $H(k)$ is proven. In particular, $R_g:=\langle \rho(\alpha_g), \rho(\beta_g)\rangle$ is   a subgroup of the center of $R_{k}:=\langle \rho(\alpha_g), \rho(\beta_g), \ldots, \rho(\alpha_k), \rho(\beta_k)\rangle$. \\
 Note that $\rho$ satisfies  $H(k)$ if and only if any $\tilde \rho \in  \liftMCG\cdot\rho$ satisfies  $H(k)$. Hence in order to prove $H(k-1)$, is suffices to prove that $R_g $ is also  a subgroup of the center of $R_{k-1}$.   \\
For $\rho'=\tau\cdot\rho$, with $\tau=\tau_{2g+k-1}^{-1}$, only one of the generators of $R_k$ is modified, namely  
$$ \rho'(\alpha_{k})= \rho( \beta_{k-1}^{-1} \alpha_k \beta_{k})= \rho( \beta_{k-1})^{-1}\rho( \alpha_k \beta_{k})\, .
$$ 
In particular, we have $\rho'(\beta_g)=\rho(\beta_g)$. 
Then $H(k)$ implies that $\rho(\beta_g)$ belongs to the center of 
$\langle R_k, R_k'\rangle=\langle R_k, \rho(\beta_{k-1})\rangle\, . $
 For $\rho''=\tau'\cdot\rho$, with $\tau'=\tau\circ \tau_{2k-3} ^{-1}$, we have 
$$
 \rho''(\alpha_{k})=  \rho( \beta_{k-1}\alpha_{k-1})^{-1}\rho( \alpha_k \beta_{k})\, .
$$
Then $H(k)$ implies that   $\rho''(\beta_g)=\rho(\beta_g) $ belongs to the center of $\langle R_k , R_k''\rangle=$ \linebreak $\langle R_k,  \rho(\beta_{k-1}\alpha_{k-1})\rangle  .$ 
We have now proven that for any representation $\rho$ such that $H(k)$ holds, $\rho(\beta_g)$  is an element of  the center of $$ R_{k-1}=\langle R_k, R_k',R_k''\rangle\, .$$
This assertion applied to $\tau_{2g-1}^{-1} \cdot \rho$ shows that $\rho(\beta_g\alpha_g)$ is an element of  the center of $R_{k-1}$.
Hence $R_g=\langle\rho(\beta_g),\rho(\beta_g\alpha_g) \rangle$ is a subgroup of the center of $R_{k-1}$. 
\vspace{.2cm}
  \item Second step: \emph{For any $\rho' \in  \liftMCG\cdot\rho$, the group $R':=R_1'$ is abelian.}\\
 If $R'$ is trivial, then in particular it is abelian.
 If $R'$ is non-trivial, then by the first step in the proof of Lemma \ref{lemnontriv}, we can find $\tau' \in \liftMCG$ such that for the induced representation $\rho''=\tau'\cdot \rho$ we have $R''=R'$ and $R_g''$ is non-trivial. Hence by the first step of the current lemma, $R'$ has a non-trivial center. Yet any subgroup of $\mathrm{Aff}(\mathbb{C})$ with non-trivial center is abelian. 
 \vspace{.2cm}
\item Third step:  \emph{For any $\rho' \in  \liftMCG\cdot\rho$, the group $R_g'$ is a subgroup of the center of $\mathrm{Im}(\rho')$.}\\
We have now proven that under our assumption, $R'$ is abelian  for any $\rho' \in  \liftMCG\cdot\rho$.
In particular, $\rho(\delta)=\mathrm{id}$.  Recall, from Remark~\ref{MixingTwistetDelta}, the action of the mixing twist $\tau_{3g+n-2}$
on $\delta$. It is given by 
$ \delta  \mapsto   [\beta_g^{-1}\delta\gamma_n^{-1}, \beta_g]\delta. $ \\
Hence, for $\rho'=\tau\cdot \rho$ with $\tau=(\sigma_{cycl}^{n-i}\circ  \tau_{3g+n-2})^{-1}$, we have
 $$\rho'(\delta)=[\rho(\beta_g^{-1} \gamma_i^{-1}),\rho( \beta_g)] =[\rho(\beta_g)^{-1}, \rho(\gamma_i)^{-1}]\, $$
 (see Table~\ref{tablecycl} page~\pageref{tablecycl}). 
Consequently, $\rho(\beta_g)$ centralizes $S:=\langle \rho(\gamma_i)~|~i\in\intint{1}{n}\rangle$. 
Yet we could have applied the same argument to $\rho''=\tau'\cdot \rho$, where $\tau'=\tau_{2g-1}^{-1}$ is the inverse of the Dehn-twist 
$\beta_g\mapsto \beta_g\alpha_g$, and we would have obtained that  $\rho(\beta_g\alpha_g)$ centralizes $S$. It follows that $R_g$ centralizes $S$. By $\liftMCG$-invariance of the statement, we deduce that   for any $\rho' \in  \liftMCG\cdot\rho$,  the associated group $R_g'$ centralizes  $\mathrm{Im}(\rho')=\langle R',S'\rangle.$
\vspace{.2cm}
\item Fourth step:  \emph{ $\mathrm{Im}(\rho)$ is abelian.}\\
If $\rho$ is the trivial representation, there is nothing to prove. Otherwise, 
by Lemma \ref{lemnontriv}, there is a representation $\rho' \in  \liftMCG\cdot\rho$ in the orbit of $\rho$ such that 
$R_g'=\langle \rho'(\alpha_g), \rho'(\beta_g)\rangle$ is not the trivial group. On the other hand, we have proven that $R_g'$ is a subgroup of the center of  $\mathrm{Im}(\rho')$. Hence $\mathrm{Im}(\rho)=\mathrm{Im}(\rho')$ is abelian. 
 \end{itemize}
 \end{proof}

\begin{lem}[Prepared form]\label{Prep}
Let $g>0$.
Let $\rho  : \gf\to \mathrm{Aff}(\mathbb{C})$ be a representation.  Assume that $\mathrm{Im}(\rho)$ is nonabelian
and $\fMCG\cdot[\rho]$ is finite.
Then up to the action of a certain element of the mapping class group and up to conjugation, $\rho$ is of the following ``prepared form'' \begin{equation}\label{eqPrep0}
\rho \left(\begin{array}{l}\alpha_g\\
\beta_g\\
\alpha_i\\
\beta_i\\
\gamma_j
\end{array}\right)=\left(\begin{array}{l}
\mu^{m_g} z\\
z+1\\
\mu^{m_i}z +a_i\\
z+b_i\\
z+c_j 
\end{array}\right)\, ,  \end{equation}
for $i\in \intint{1}{g-1}$ and  $j\in \intint{1}{n}$, where $\mu\in \mathbb{C}^*\setminus \{1\}$ is a root of unity, $m_g, m_i\in \mathbb{Z}$, $a_i, b_i, c_j\in \mathbb{C}$ and $\mu^{m_g}\neq 1$.  \end{lem}

\begin{proof} 

According to Lemma \ref{normalize}, up to the action of an element of the mapping class group, we may assume 
$\rho([\alpha_g, \beta_g])\neq \mathrm{id}$. Since $\fMCG\cdot[\rho]$ is finite, the linear part $\rho_{lin}$ of $\rho$ also has finite orbit. After Proposition \ref{abcase}, $\rho_{lin}$ takes values in a finite cyclic group $\langle \mu\rangle\subset\C^*$. Hence for each $i\in \intint{1}{g}$, we have  
$$\rho(\alpha_i)=\mu^{m_i}z+a_i, \quad \rho(\beta_i)=\mu^{n_i}z+b_i $$
   for integers $m_i, n_i\in \mathbb{Z}$ and 
complex numbers $a_i, b_i\in \mathbb{C}$.
Consider the actions of $\tau_{2i}^{-1}$ and $\tau_{2i-1}^{-1}$ on  $({m_i}, n_i)$ (the other exponents are not altered) :
$$\begin{array}{|l| lcl |}
\hline
\phantom{\begin{array}{l}\!\!\!\! \vspace{.3cm}\\ \!\!\!\! \end{array}}\tau_{2i-1}^{-1} &\begin{pmatrix}
m_i\\
n_i
\end{pmatrix}\ &\mapsto&\begin{pmatrix}
1&0\\
1&1
\end{pmatrix} \begin{pmatrix}
m_i\\
n_i
\end{pmatrix}\\ 
\hline
\phantom{\begin{array}{l}\!\!\!\!  \vspace{.3cm} \\ \!\!\!\! \end{array}} \tau_{2i }^{-1}  &\begin{pmatrix}
m_i\\
n_i
\end{pmatrix} & \mapsto& \begin{pmatrix}
1&1\\
0&1
\end{pmatrix} \begin{pmatrix}
m_i\\
n_i
\end{pmatrix}\\
\hline
\end{array}$$
These actions generate the action of $\SLtwoZ$ on $(m_i, n_i)\in \mathbb{Z}^2$. 
If $(m_i, n_i)\neq (0,0)$, then $\widetilde{m}_i:=\gcd (m_i,n_i)$ is a well-defined positive integer. 
Let $p_i$ and $q_i$ be integers such that
$p_i m_i+q_i n_i = \widetilde{m}_i.$ The matrix $$\begin{pmatrix}
p_i&q_i\\
-\frac{n_i}{\widetilde{m}_i}&\frac{m_i}{\widetilde{m}_i}
\end{pmatrix}\in \SLtwoZ$$ then sends
$(m_i,n_i)$ to $(\widetilde{m}_i, 0)$.
Hence, up to the action of a word in the twists $(\tau_{2i})_{i \in \intint{1}{g}}$, $(\tau_{2i-1})_{i \in \intint{1}{g}}$, we may assume  $n_i= 0$ for each ${i \in \intint{1}{g}}$. The property $\rho([\alpha_g, \beta_g])\neq \mathrm{id}$ is not altered by such a word, hence $\mu^{m_g}\neq 1$.
Up to conjugation by an element of $\mathrm{Aff}(\mathbb{C})$, we may moreover assume
\begin{equation}\label{RepNormal}\rho\left(  \begin{array}{l}
\alpha_g\\
\beta_g
\end{array}\right)=\left(\begin{array}{l}
\mu^{m_g}z\\
z+1
\end{array}\right).\end{equation}
Note that since $\mu^{m_g}\neq 1$, $\rho$ is the unique representative of the conjugacy class $[\rho]$ satisfying \eqref{RepNormal}.
For $j\in \intint{1}{n}$, let $c_j,  d_j \in \mathbb{C}$ be defined by 
$\rho(\gamma_j)=d_jz+c_j $.
For $k \in \mathbb{Z}$, consider the action of $\tau_2^{-k}$   :
\begin{equation}\label{normtripel}\tau_2^{-k}\cdot\rho \left(\begin{array}{l} \alpha_g\\
\beta_g\\
\gamma_j
\end{array}\right)=\left(\begin{array}{l}
\mu^{m_g} z+k\mu^{m_g}\\
z+1\\
d_jz+c_j\\
\end{array}\right)\approx \left(\begin{array}{l}
\mu^{m_g} z\\
z+1\\
d_jz+c_j-k\frac{(d_j-1)\mu^{m_g}}{\mu^{m_g}-1} 
\end{array}\right)\, .\end{equation}
Here the equivalence sign stands for being conjugated by an element of $\mathrm{Aff}(\mathbb{C})$.  
For the sequence, parametrized by $k \in \mathbb{Z}$, of normalized triples  \eqref{normtripel} to be finite, we must have $d_j=1$.
 \end{proof}

\subsection{Non-abeliancase in genus one}

\begin{prop}[Non-abelianrepresentations for $g=1$]\label{SpecialCaseg1c}
Assume    $g=1$. Let $\rho : \gf \to \mathrm{Aff}(\mathbb{C})$ be a representation with non-abelianimage (in particular $n\geq 1$). Then the orbit $\fMCG\cdot[\rho]$ is finite if and only if there is a root of unity $\mu\neq 1$ and $\mathbf{c}:=(c_1, \ldots, c_{n})\in \mathbb{C}^{n}$ with $\sum_{i=1}^nc_i=1$ such that $[\rho]\in \fMCG\cdot[\rho_{\mu, \mathbf{c}}]$, where $\rho_{\mu, \mathbf{c}}$ is the representation given by
\[ \rho_{\mu, \mathbf{c}} (\alpha_1)=\mu z\, ; ~~ \rho_{\mu, \mathbf{c}} (\beta_1)=z-\frac{1}{\mu-1}\, ;  ~~ \rho_{\mu, \mathbf{c}} (\gamma_i)=z+c_i ~~  \forall i\in \intint{1}{n}\, .\]

   \end{prop}
 
\begin{proof}
 Recall that for $g=1$,  the fundamental group $\gf$ has the following presentation
  $$\gf=\langle \alpha_1, \beta_1, \gamma_1, \ldots, \gamma_{n}~|~\gamma_1\cdots \gamma_n= [\alpha_1, \beta_1]\rangle\, ,$$
 and the mapping class group $\fMCG $ is generated by the elements of Table $\ref{table genus 1}$.   
 \begin{table}[htbp]
  $$\begin{array}{|ll| rcl l|}
\hline
\tau_{1} &  &  \beta_1 & \mapsto&\beta_1 \alpha_1&\\
\hline
\tau_{2} &  &  \alpha_1 & \mapsto& \alpha_1 \beta_1&\\\hline
\widetilde{\tau}_{2+k}:=\tau_2^{-1}\circ {\tau}_{2+k}\, ,&k\in \intint{1}{n-1}  & \alpha_1 & \mapsto& \alpha_1\beta_1^{-1} \Xi_k &\\
&&\beta_1 & \mapsto& \Xi_k^{-1}\beta_1\Xi_k&\\
&&\gamma_i & \mapsto& \Xi_k^{-1} \gamma_i \Xi_k\,,& i\in \intint{1}{k}\\
 & &where&  \Xi_k =&(\gamma_1\ldots \gamma_k)^{-1}\beta_1&\\
\hline
 \sigma_{i}\,, & i\in \intint{1}{n-1} &\gamma_i & \mapsto& \gamma_i\gamma_{i+1}\gamma_i^{-1}& \\
  &&\gamma_{i+1}&\mapsto& \gamma_i&\\ 
 \hline
\end{array}$$
\caption{\label{table genus 1}}
\end{table}
 Assume $[\rho]$ has finite orbit, then by Lemma \ref{Prep}, we have $\fMCG \cdot [\rho]=\fMCG \cdot [\rho_{\mu , \mathbf{c}}]$ for a convenient choice of $\mathbf{c}\in \mathbb{C}^n$ and a root of unity $\mu\neq 1$. 
 Let us now prove that $[\rho_{\mu, \mathbf{c}}]$ has finite orbit. Denote $$N:=\mathrm{order}(\mu)\, ; \quad \quad D_{\mathbf{c}}:=\mu^{\mathbb{Z}}c_1+\ldots+\mu^{\mathbb{Z}}c_n\, .$$
Denote the following sets of tuples of affine transformations  
$$\begin{array}{rcl}S_{\mu,d}^1&:=&\left\{ \left(\begin{array}{l} \mu^{k_1}z\\\mu^{k_2}z-\frac{d}{\mu^{k_1}-1}\end{array}\right)  ~\middle|~\begin{array}{l}k_1, k_2 \in\mathbb{Z}, \quad  k_1\not \in N\mathbb{Z},\vspace{.1cm}\\
\gcd (k_1,k_2,N)=1  \end{array}\right\}
\vspace{.3cm}\\
S_{\mu, d}^2&:=&\left\{ \left(\begin{array}{l} \mu^{k_1}z+\frac{d}{\mu^{k_2}-1}\\\mu^{k_2}z\end{array}\right)  ~\middle|~\begin{array}{l}k_1, k_2 \in\mathbb{Z},\quad 
  k_2\not \in N\mathbb{Z},\vspace{.1cm}\\ \gcd (k_1,k_2,N)=1 \end{array}\right\}
 \vspace{.3cm} \\
 R_{\mu,\mathbf{c}, d}&:=&\left\{ \left(\begin{array}{c} z+\tilde c_1 \\ \vdots \\ z+\tilde c_n \end{array}\right)  ~\middle|~\begin{array}{l}
(\hat{c}_1, \ldots, \hat{c}_n)\, \in \, \mathfrak{S}_n\cdot (c_1, \ldots, c_n), \vspace{.1cm}
 \\
  \tilde{c}_i\in \mu^\mathbb{Z} \hat c_i\quad \quad  \forall i\in \intint{1}{n},\vspace{.1cm}
  \\
   d=\sum_{i=1}^n \tilde c_i
 \end{array}\right\}\, .\end{array}$$
   Moreover, we set $S_{\mu, d}:=S_{\mu, d}^1\cup S_{\mu,d}^2 $. Then by definition, we have $$\rho_{\mu, \mathbf{c}} \left(\begin{array}{l} \alpha_1\\\beta_1\\\gamma_1 \\ \vdots \\\gamma_n
\end{array}\right)\in  O_{\mu,\mathbf{c}}:= \bigcup_{d\in D_{\mathbf{c}}}
\left\{
\left(
\begin{array}{l}
\varphi_\alpha\\
\varphi_\beta \\
 \varphi_{1}\\ 
 \vdots \\
 \varphi_{n} 
 \end{array}
 \right)
 ~\middle|~
 \begin{array}{l}
  \left(\begin{array}{rcl}\varphi_\alpha\\\varphi_\beta   \end{array}\right)
   \in S_{\mu,d} 
    \, , \quad
      \left(\begin{array}{l} \varphi_{1}\\ \vdots \\\varphi_{n} \end{array}\right)
       \in R_{\mu,\mathbf{c},d} \end{array} \right\}
\, .$$ Note that $O_{\mu , \mathbf{c} }$ is a finite set, and we will prove that each conjugacy class in the orbit of $\rho_{\mu, \mathbf{c}}$ under the action of the mapping class group has a representative in $O_{\mu , \mathbf{c} }$. We shall denote $[O_{\mu , \mathbf{c} }]$ the image of $O_{\mu , \mathbf{c} }$ in $\chign{\Aff}$.
 \begin{itemize}
 \item \emph{The set $[O_{\mu , \mathbf{c}}]$  is stable under the inverses of   $\tau_1$ and $\tau_2$.}\\
In order to prove this first assertion, it is enough to prove that the sets $ S_{\mu,d}^1$ and $S_{\mu,d}^2$ are stable under the action of $\tau_1^{-1}$ and $\tau_2^{-1}$ modulo conjugation by translations.  Let $$\rho\left(\begin{array}{l} \alpha_1\\\beta_1 \end{array}\right)= \left(\begin{array}{l} \mu^{k_1}z\\\mu^{k_2}z-\frac{d}{\mu^{k_1}-1}\end{array}\right) \in S_{\mu,d}^1.$$ Then 
 $$\tau_1^{-1} \cdot\rho\left(\begin{array}{l} \alpha_1\\\beta_1 \end{array}\right)= \left(\begin{array}{l} \mu^{k_1}z\\\mu^{k_1+k_2}z-\frac{d}{\mu^{k_1}-1}  \end{array}\right) \in S_{\mu, d}^1$$
and 
 $$\tau_2^{-1} \cdot\rho\left(\begin{array}{l} \alpha_1\\\beta_1 \end{array}\right) = \left(\begin{array}{l} \mu^{k_1+k_2}z
 -\frac{\mu^{k_1}d}{\mu^{k_1}-1}\\\mu^{k_2}z-\frac{d}{\mu^{k_1}-1} \\
  \end{array}\right) \, .$$
To see that, up to conjugation by a translation, the latter image also belongs to $S_{\mu,d}$, we need to distinguish two cases. Firstly, if $k_1+k_2\in N\mathbb{Z}$, then $k_2\not \in N\mathbb{Z}$ and we obtain 
$$ \left(\begin{array}{l} \mu^{k_1+k_2}z
 -\frac{\mu^{k_1}d}{\mu^{k_1}-1}\\\mu^{k_2}z-\frac{d}{\mu^{k_1}-1} \\
  \end{array}\right)= \left(\begin{array}{l}  z
 +\frac{d}{\mu^{k_2}-1}\\\mu^{k_2}z-\frac{d}{\mu^{k_1}-1} \\
  \end{array}\right)\approx \left(\begin{array}{l}  z+\frac{d}{\mu^{k_2}-1}\\\mu^{k_2}z  
 \end{array}\right) \in S_{\mu , d}^2$$
Secondly, if $k_1+k_2\not \in  N\mathbb{Z}$, then we obtain 
$$ \left(\begin{array}{l} \mu^{k_1+k_2}z
 -\frac{\mu^{k_1}d}{\mu^{k_1}-1}\\\mu^{k_2}z-\frac{d}{\mu^{k_1}-1} \\
  \end{array}\right) \approx 
 \left(\begin{array}{l} \mu^{k_1+k_2}z\\\mu^{k_2}z -\frac{d}{\mu^{k_1+k_2}-1} \\
\end{array}\right)  \in S_{\mu, d }^1$$
In a similar way, one can show that up to conjugation by translations, we have $\tau_1^{-1}\cdot S_{\mu, d}^2\subset S_{\mu, d}$ and $\tau_2^{-1}\cdot S_\mu^2\subset   S_\mu^2$.     
\vspace{.2cm}
 \item \emph{The set $[O_{\mu,\mathbf{c}}]$  is stable under the inverses of $\sigma_1, \ldots, \sigma_{n-1}$.} \\
 Indeed,  for every $\rho\in O_{\mu,\mathbf{c}}$, the group
 $\langle \rho(\gamma_1), \ldots, \rho(\gamma_n)\rangle$ is a translation group. In particular, it is abelian. Hence the elements $\sigma_i$ act as permutations. But permutations stabilize the set $R_{\mu,\mathbf{c}, d}$.
 \vspace{.2cm}
  \item\emph{The set $[O_{\mu,\mathbf{c}}]$  is stable  under the inverse of the modified mixing twist  $\widetilde{\tau}_{2+k}$ for every $k\in \intint{1}{n-1}$.} \\
Note that for $k\in \intint{1}{n-1}$, up to a common conjugation by $\rho(\Xi_k)$, the representation $\rho':=\widetilde{\tau}_{2+k}^{-1}\cdot\rho $ may be described as follows, where $\Xi_k = (\gamma_1\ldots \gamma_k)^{-1} \beta_1$.
 \[
 \left \lbrace \begin{array}{lr}
 \rho'(\alpha_1)=\rho(\Xi_k\alpha_1 \beta_1^{-1} )~~&\\
\rho'(\beta_1)=\rho(  \beta_1 )~&\\
  \rho'(\gamma_i)= \rho(\gamma_i)~~& i\in \intint{1}{k}; \\
\rho'(\gamma_j)=\rho(\Xi_k \gamma_j\Xi_k^{-1})~~& j\in \intint{k+1}{n}.\\
 \end{array}
 \right.
 \]
  In the following calculations, $i$ represents an index less or equal to $k$ (if such an index exists) and $j$ represents an index greater than $k$.\\
   Assume first that $ \rho\left(   {\alpha_1}\, ,   {\beta_1}  \right)   \in S_{\mu,d}^1.$ Then
 $$\rho\left(\begin{array}{l}   {\alpha_1}\\\,   {\beta_1} \\ \gamma_i \\ \gamma_j \end{array}\right)=   \left(\begin{array}{l} \mu^{k_1}z\\\mu^{k_2}z-\frac{d}{\mu^{k_1}-1}\\z+\tilde c_i\\z+\tilde c_j\end{array}\right)\quad\textrm{and}\quad \rho(\Xi_k) = \mu^{k_2}z-\frac{d}{\mu^{k_1}-1}-\sum_{i=1}^k \tilde c_i\, . $$  Hence
   $$   \rho'\left(\begin{array}{l} 
     {\alpha_1}\vspace{.1cm}\\\,   {\beta_1}\vspace{.1cm} \\ \gamma_i \vspace{.1cm}\\ \gamma_j \end{array}\right)= \left(\begin{array}{l}
 \mu^{k_1}z+d-\sum_{h=1}^k \tilde c_h
    \vspace{.1cm}
   \\\mu^{k_2}z-\frac{d}{\mu^{k_1}-1}\vspace{.1cm}\\z+\tilde c_i\vspace{.1cm}\\z+\mu^{k_2}\tilde c_j\end{array}\right)
  \approx 
  \left(\begin{array}{l}
  \mu^{k_1}z   
    \vspace{.1cm}
   \\\mu^{k_2}z 
    -\frac{d'}{\mu^{k_1}-1} 
   \vspace{.1cm}\\z+\tilde c_i\vspace{.1cm}\\z+\mu^{k_2}\tilde c_j\end{array}\right)\, , 
  $$
  where $$d'=\mu^{k_2}d-(\mu^{k_2}-1)\sum_{i=1}^k \tilde c_i=\sum_{i=1}^k \tilde c_i+ \sum_{j=k+1}^n \mu^{k_2}\tilde c_j\, ,$$
  since $d=\sum_{i=1}^k\tilde c_i+\sum_{j=k+1}^n \tilde c_j$. 
  In other words,  up to conjugation by a translation, we have $\rho' \in O_{\mu, \mathbf{c}}$. 
   By an almost identical argumentation, we show that if $\rho \in  O_{\mu, \mathbf{c}}$ with $ \rho\left(   {\alpha_1}\, ,   {\beta_1}  \right)   \in S_{\mu, d}^2$, then $\widetilde{\tau}_{2+k}^{-1}\cdot\rho$ is also in  $O_{\mu, \mathbf{c}}$ modulo conjugation. \\
  \end{itemize} 
Since every element of $\liftMCG$ induces a bijection of $\chign{\mathrm{Aff}(\mathbb{C})}$ and we have proven that $[O_{\mu, \mathbf{c}}]$ is stable under $\tau_i^{-1}$ for every  $i\in \intint{1}{n+1}$ and  $\sigma_j^{-1}$  for every  $j\in \intint{1}{n-1}$, these generators of  $\liftMCG$  induce bijections of $$[O_{\mu, \mathbf{c}}]\subset \chign{\mathrm{Aff}(\mathbb{C})}\, .$$ Hence $[O_{\mu, \mathbf{c}}]$ is also stable under $\tau_i$ for every  $i\in \intint{1}{n+1}$ and  $\sigma_j$  for every  $j\in \intint{1}{n-1}$. We conclude that the orbit $\fMCG\cdot [\rho_{\mu, \mathbf{c}}]=\liftMCG\cdot [\rho_{\mu, \mathbf{c}}]$  is contained in the finite set $[O_{\mu, \mathbf{c}}]$.  
\end{proof}

\subsection{Non-abeliancase in higher genus}
We are now considering the case $g>1$, and arbitrary $n\geq 0$.
Recall that $\gf$ then contains the group
$$G:=\langle \alpha_{g-1}, \beta_{g-1}, \alpha_g, \beta_g \rangle \subset \gf$$
and $\liftMCG$ contains a subgroup
 $$H:=\langle \tau_{2g-3}, \tau_{2g-2}, \tau_{2g-1}, \tau_{2g}, \tau_{3g-1}\rangle \subset \liftMCG\, $$
 acting on $G$ as summarized by Table~$\ref{DehnAction1sous}$. Here, as usual, we only indicate the effect of the generators of $H$ on those generators of $G$ which are not left invariant by the element of $H$ under consideration.  \begin{table}[htbp]
$\begin{array}{|ll| lcl |}
\hline
\tau_{2k} &k\in\intint{g-1}{g} & \alpha_k & \mapsto& \alpha_k \beta_k\\
\hline
\tau_{2k-1} & k\in\intint{g-1}{g} & \beta_k & \mapsto& \beta_k\alpha_k\\
\hline
\tau_{3g-1}
&&\alpha_{g}&\mapsto& \Theta^{-1} \alpha_{g}  \\
&  &  \alpha_{g-1} &\mapsto& \alpha_{g-1}\Theta\\
&&\beta_{g-1}&\mapsto& \Theta^{-1}\beta_{g-1}\Theta \\
&&where&\Theta=&\alpha_{g}\beta_{g}^{-1}\alpha_{g}^{-1}\beta_{g-1}\\
\hline
\end{array}
$\\$ $
\\
\caption{\label{DehnAction1sous}}
\end{table}

\begin{lem}[Elimination criterion]\label{Prep0}
Let $g\geq 2$.
Let $\rho  : \gf \to \mathrm{Aff}(\mathbb{C})$ be a representation of the following ``weak prepared form''
   \begin{equation}\label{eqPrep1}
\rho \left(\begin{array}{l}\alpha_g\\
\beta_g\\
\alpha_{g-1}\\
\beta_{g-1}\\
\end{array}\right)=\left(\begin{array}{l}
\mu^{m_g} z\\
z+1\\
\mu^{m_{g-1}}z +a\\
z+b\\
\end{array}\right)\, , \end{equation}
where $\mu$ is a root of unity, $a, b\in \mathbb{C}$, $m_g, m_{g-1}\in \mathbb{Z}$ and $\mu^{m_g}\neq 1$. If  $\fMCG\cdot[\rho]$ is finite, then the conditions of Table~$\ref{elimcrit}$ are fulfilled. 
\vspace{-.3cm}
\begin{table}[htbp]
$$\fbox{$\begin{array}{lcl}
 \mu^{m_{g-1}} &=& \frac{1}{\mu^{m_g}}\\
a&=&0\\
b&=&0
\end{array}$}$$
 \caption{\label{elimcrit}}
\end{table}
 \end{lem}
 \vspace{-.6cm}
 \begin{proof}
Note that if two representations $\rho, \rho'$ of the form \eqref{eqPrep1} are conjugated, then they their restrictions to $G$ are equal. Assume $$\rho\left(\begin{array}{l}
\alpha_{g}\\
\beta_g\\
\alpha_{g-1}\\
\beta_{g-1}
\end{array}\right)= \left(\begin{array}{l}
\mu^{m_{g}} z\\
z+1\\
\mu^{m_{g-1}}z +a\\
z+b
\end{array}\right)$$
Now consider the action of $\tau_{2g-2}^{-k}$ for $k\in \mathbb{Z}$:  
$$\tau_{2g-2}^{-k}\cdot \rho\left(\begin{array}{l}
\alpha_{g}\\
\beta_g\\
\alpha_{g-1}\\
\beta_{g-1}
\end{array}\right)=     \left(\begin{array}{l}
\mu^{m_{g}} z\\
z+1\\
\mu^{m_{g-1}}z +a+k\cdot \mu^{m_{g-1}}b\\
z+b
\end{array}\right)$$
Since the suborbit  $(\tau_{2g-2}^{-k}\cdot [\rho])_k$ is supposed to take finitely many values, we have
$\fbox{$b=0$} .$\\
Now consider the action of $\tau_{3g-1}^{-k}$.  We have
$$\begin{array}{rlllllll}
\tau_{3g-1}^{-k}\cdot \rho\left(\begin{array}{l}
\alpha_{g}\\
\beta_g\\
\alpha_{g-1}\\
\beta_{g-1}
\end{array}\right)= \left(\begin{array}{l}
\mu^{m_{g}} z+k\mu^{m_{g}}\\
z+1\\
\mu^{m_{g-1}}z +a-k\mu^{m_g+m_{g-1}}\\
z
\end{array}\right)\approx   \left(\begin{array}{l}
\mu^{m_{g}} z \\
z+1\\
\mu^{m_{g-1}}z +a-k\cdot \frac{\mu^{2m_g+m_{g-1}}- \mu^{m_{g}}}{\mu^{m_{g}}-1}\\
z
\end{array}\right).\end{array}$$
As the corresponding suborbit is supposed to be finite, we have $\fbox{$\mu^{m_{g-1}}=\mu^{-m_g}$} .$\\
 In order to conclude,  consider  $\tilde{\tau}_{3g-1}=\tau^{-1}\circ \tau_{3g-1}\circ  \tau$, where
 $\tau:= \tau_{2g-3}\circ\tau_{2g}\circ \tau_{2g-1}^{-1}\circ \tau_{2g} $.
 We have 
 $$  {\tilde{\tau}_{3g-1}^k}_*:\left(\begin{array}{l}
\alpha_g\vspace{.1cm}\\
\beta_g\vspace{.1cm}\\
\alpha_{g-1}\vspace{.1cm}\\
\beta_{g-1}
\end{array}\right)\stackrel{  }{\mapsto}\left(\begin{array}{l}
\widetilde{\Theta}^{-k} \alpha_g\widetilde{\Theta}^{k}\vspace{.1cm}\\
\beta_g\widetilde{\Theta}^{k}\vspace{.1cm}\\
\alpha_{g-1}\widetilde{\Theta}^{k}\vspace{.1cm}\\
\widetilde{\Theta}^{-k}\beta_{g-1}\alpha_{g-1}^{-1}\widetilde{\Theta}^{k}\alpha_{g-1}\widetilde{\Theta}^{k}
\end{array}\right), $$
  where $\widetilde{\Theta}:=\tau^{-1}_*{\Theta}=\alpha_g^{-1}\beta_{g-1}\alpha_{g-1}^{-1}$.
  We have $$\rho\left(\widetilde{\Theta}^k\right)=z-k\cdot a\, .$$
  Hence, modulo conjugation by $\rho\left(\widetilde{\Theta}^k\right)$, we have
  $$ \tilde{\tau}_{3g-1}^{-k}\cdot \rho\left(\begin{array}{l}
\alpha_{g}\vspace{.1cm}\\
\beta_g\vspace{.1cm}\\
\alpha_{g-1}\vspace{.1cm}\\
\beta_{g-1}
\end{array}\right)\approx \rho \left(\begin{array}{l}
\alpha_g\vspace{.1cm}\\
\widetilde{\Theta}^k\beta_g\vspace{.1cm}\\
\widetilde{\Theta}^{k}\alpha_{g-1}\vspace{.1cm}\\
\beta_{g-1}\alpha_{g-1}^{-1}\widetilde{\Theta}^{k}\alpha_{g-1}
\end{array}\right) 
= \left(\begin{array}{l}
\mu^{m_{g}} z\vspace{.1cm}\\
z+1-k\cdot a\vspace{.1cm}\\
\frac{1}{\mu^{m_{g}}}z+(1-k)\cdot  a \vspace{.1cm}\\
  z-k\cdot  a\mu^{m_{g}} 
\end{array}\right) .$$
Provided $1-k\cdot a\neq 0$ (which is the case for an infinite number of $k\in \mathbb{Z}$ anyway), we obtain
$$\tilde{\tau}_{3g-1}^{-k}\cdot \rho\left(\begin{array}{l}
\alpha_{g}\vspace{.1cm}\\
\beta_g\vspace{.1cm}\\
\alpha_{g-1}\vspace{.1cm}\\
\beta_{g-1}
\end{array}\right)\approx \left(\begin{array}{l}
\mu^{m_{g}} z\vspace{.1cm}\\
z+1\vspace{.1cm}\\
\frac{1}{\mu^{m_{g}}}z+\frac{(1-k)\cdot  a}{1-k\cdot a} \vspace{.1cm}\\
  z-\frac{k\cdot  a}{1-k\cdot a}\mu^{m_{g}}
\end{array}\right).$$
Again by finiteness,  we have $\fbox{$a=0$} .$   \end{proof}

  \begin{prop}[Non-abelianrepresentations for $g>1$]\label{NABCaseg2}
Assume $g\geq 2$ and $n\geq 0$. Let $\rho : \gf \to \mathrm{Aff}(\mathbb{C})$ be a  representation with non-abelianimage. Then the orbit $\fMCG\cdot [\rho]$ is infinite. 
   \end{prop}
   \begin{proof}
   Let $g\geq 2$ and let $\rho$ be a representation with finite orbit modulo conjugation.
Let us assume for a contradiction that $\rho(\gf)$ is non-abelian. We may then assume that $\rho$ is of ``prepared form'' as in Lemma 
 $\ref{Prep}$. In particular, we may assume that  $\rho$ is of ``weak prepared form'' and hence, by Lemma \ref{Prep0}, $\rho$ satisfies the conditions of Table~\ref{elimcrit}. In other words, we may assume that $\rho$ is of the following form.
$$\rho\left(\begin{array}{l}
\alpha_{g}\\
\beta_g\\
\alpha_{g-1}\\
\beta_{g-1}
\end{array}\right)=\left(\begin{array}{l}
\mu z\\
z+1\\
\frac{1}{\mu}z\\
z
\end{array}\right)\, , $$
where $\mu\neq 1$ is a root of unity.
We have $$\tau_{3g-1}^{-1}\cdot \rho\left(\begin{array}{l}
\alpha_{g}\\
\beta_g\\\alpha_{g-1}\\
\beta_{g-1}
\end{array}\right)= \left(\begin{array}{l}
\mu z+\mu\\
z+1\\\frac{1}{\mu}z-1\\
z
\end{array}\right)\, ; \quad \quad \quad \tau_{2g}\cdot (\tau_{3g-1}^{-1}\cdot \rho)\left(\begin{array}{l}
\alpha_{g}\\
\beta_g\\\alpha_{g-1}\\
\beta_{g-1}
\end{array}\right)= \left(\begin{array}{l}
\mu z\\
z+1\\\frac{1}{\mu}z-1\\
z
\end{array}\right).$$
Now $\tau_{2g}\circ\tau_{3g-1}^{-1}\cdot\rho$ is also of weak prepared form, but is not compatible with the elimination criterion of Table~\ref{elimcrit}, whence the contradiction. 
 \end{proof}

\section{Reducible rank 2 representations with finite orbit}
 Theorem \ref{mainthm dynamics} concerns representations $\rho: \gf \rightarrow  \GLtw$ that are reducible, \textit{i.e.} that globally fix a line in $\C^2$. A particular case of reducible rank 2 representations are those that are totally reducible,  \textit{i.e.}  that globally fix two distinct lines in $\C^2$. In \hyperlink{cortotred}{Theorem B1}, we will prove the statement in the totally reducible case, and in \hyperlink{corred}{Theorem B2}, we will prove it in the reducible but not totally reducible case. The juxtaposition of these two results yields Theorem \ref{mainthm dynamics}. First, we are going to estimate the size of finite orbits of conjugacy classes of affine representation under the pure mapping class group and prove the reduction to the affine case. 

\subsection{The size of some finite orbits}\label{SecCalcul}
Note that since $\C^*$ is abelian, we have a natural identification between scalar representations and their conjugacy classes: $\chign{\C^*}=\mathrm{Hom}(\gf, \C^*)$. In particular, the pure mapping class group $\MCG$ acts on $\mathrm{Hom}(\gf, \C^*)$.

\begin{prop}\label{propscalar} Let $g>0, n\geq 0$. Let $\lambda \in \Hom(\gf,\mathbb{C}^*)$ be a scalar representation with finite image. Then \begin{equation}\label{eqcountlambda}  \cardn(\mathrm{Im}(\lambda))^{2g-1}  \leq  \cardn (\MCG \cdot \lambda)\leq \cardn(\mathrm{Im}(\lambda))^{2g}\end{equation}
\end{prop}
\begin{proof}

Since $\mathrm{Im}(\lambda)$ is finite,  there is a root of unity $\mu\in \C^*$ such that $\mathrm{Im}(\lambda)=\mu^{\mathbb{Z}}$.  For each $j\in \intint{1}{n}$, choose an integer $m_j\in \mathbb{Z}$ such that $\lambda(\gamma_j)=\mu^{m_j}.$ Denote $N:=\mathrm{order}(\mu)$ and
$$O_\lambda := \left\{(\mu^{k_g}, \mu^{\ell_g}, \ldots, \mu^{k_1}, \mu^{\ell_1})~\middle|~\begin{array}{l}\mathbf{k}:=(k_g, \ldots, k_1)\in\mathbb{Z}^g\, ,\quad \,\mathbf{\ell}:=(\ell_g, \ldots, \ell_1)\in\mathbb{Z}^g \vspace{.1cm}\\ \gcd(k_g, \ldots, k_1, \ell_g, \ldots, \ell_1, m_1, \ldots, m_n, N)=1\end{array} \right\}\, .$$
Note that to any element $(\mu^{k_g}, \mu^{\ell_g}, \ldots, \mu^{k_1}, \mu^{\ell_1})\in O_\lambda$ we can associate a 
well-defined representation $\lambda' \in \Hom(\gf,\mathbb{C}^*)$ by setting $\lambda'(\alpha_i)=\mu^{k_i}\, ; \, \lambda'(\beta_i)=\mu^{\ell_i}$ for all $i\in \intint{1}{g}$ and $\lambda'(\gamma_j)=\mu^{m_j}$ for all $j\in \intint{1}{n}$.
In that sense, we can see $O_\lambda$ as a subset of $ \Hom(\gf,\mathbb{C}^*)$.

We claim that $\MCG\cdot \lambda=O_\lambda$. Notice that this claim implies \eqref{eqcountlambda}. Indeed, the second inequality is obvious, and the first one follows from the fact that if we set for example $k_g=1$, then we can choose all other exponents freely.

Let us now prove the claim.  We clearly have $\lambda \in O_\lambda$.
Each pure element $\tau$ of $\liftMCG$ transforms the generators $\gamma_i$ into conjugates $\zeta_i^{-1} \gamma_i\zeta_i$.
Since $\C^*$ is abelian, this implies that for any representation $\lambda'$ corresponding to an element of $O_\lambda$, we have $(\tau\cdot \lambda')(\gamma_i)=\lambda'(\gamma_i)=\mu^{m_i}$.  Consequently, $ \MCG\cdot O_\lambda=O_\lambda$ and in particular $\MCG\cdot \lambda \subset  O_\lambda$\,.
 
The orbits of $\fMCG$ on $\chign{\C^*}=\Hom(\gf,\mathbb{C}^*)$ are the ones of $\liftMCG$. 
Note that the subgroup $H:=\langle \tau_{i}~\vert~i\in \intint{1}{3g-1+ \min(0, n-1)}\rangle \subset \liftMCG$ is generated by pure elements. 
 Translating Table \ref{A} into an action of $\liftMCG$ on the powers of $\mu$ corresponding to the generators of $\gf$ then yields the following.
\begin{enumerate}[(a)] 
\item For a given $(\tilde{k}_g, \ldots, \tilde{k}_1)\in \{1, \ldots, N\}^g$ such that $\gcd(\tilde{k}_g, \ldots, \tilde{k}_1, m_1, \ldots, m_n, N)=1$, the subgroup $\langle \tau_{2i}\, ,\,  \tau_{2i-1}~|~i\in \intint{1}{g}\rangle \subset H$ acts transitively on those elements of $O_\lambda$ satisfying $\gcd (k_i, \ell_i)=\tilde{k}_{i}$ for all $i \in \intint{1}{g}$ (see also the proof of Lemma \ref{Prep}).\vspace{.1cm}
\item  For all  $\widetilde{\mathbf{k}}:=(\tilde{k}_g, \ldots, \tilde{k}_1)\in \{1, \ldots, N\}^g$ such that $\gcd(\tilde{k}_g, \ldots, \tilde{k}_1, m_1, \ldots, m_n, N)=1$, there is an element of the subgroup $\langle \tau_{2i}\, , \, \tau_{2i-1}\, , \, \tau_{2g+i'}~|~i\in \intint{1}{g}\, , \, i'\in \intint{1}{g-1}\rangle \subset H$, which sends the element of $O_\lambda$ given by $\mathbf{k}=  (\gcd(\tilde{k}_g, \ldots, \tilde{k}_1), 0, \ldots, 0)$ and $\ell=\vec{0}$ to the element of $O_\lambda$ given by $\mathbf{k}=\widetilde{\mathbf{k}}$ and $\ell=\vec{0}$.  \vspace{.1cm}
\item The subgroup $\langle \tau_{3g-1+j} \, ~|~j\in \intint{1}{\min(0,n-1)}\rangle \subset H$ acts transitively on those elements of $O_\lambda$ satisfying $\ell=\vec{0}$ and $k_i=0$ for all $i \in \intint{1}{g-1}$.
\end{enumerate}
Consequently,  the pure subgroup $H$ acts transitively on $O_\lambda$. This implies $\MCG\cdot \lambda=O_\lambda$. 
 \end{proof}

Recall that we denote \begin{equation}\label{AffCommeMatrice}\mathrm{Aff}(\mathbb{C})=\left\{ \begin{pmatrix}a & b \\0&1 \end{pmatrix}~\middle|~a, b\in \mathbb{C},~a\neq 0\right\}\, . \end{equation}
 
\begin{prop}\label{lemcomptermuc} Let $g=1$ and let $n>0$. Let $\mu\in \C^*$ be a root of unity of order $N>1$ and let  $\mathbf{c}=(c_1, \ldots, c_n) \in \C^n$ with $\sum_{i=1}^{n}c_i=1$. Consider the representation $\rho_{\mu, \mathbf{c}}\in \Hom(\gf,\Aff)$ defined by $$\rho_{\mu, \mathbf{c}} (\alpha_1):=\begin{pmatrix}\mu & 0\\ 0 & 1\end{pmatrix}\, ;  \quad  \rho_{\mu, \mathbf{c}}(\beta_1):=\begin{pmatrix} 1& -\frac{1}{\mu-1}  \\0 & 1\end{pmatrix} \, ;  \quad  \rho_{\mu, \mathbf{c}}(\gamma_i):=\begin{pmatrix}1&c_i\\0&1\end{pmatrix}   \quad \forall i\in \intint{1}{n}\,  .$$ 
Recall from Proposition  \ref{SpecialCaseg1c} that the orbit $\MCG \cdot [\rho_{\mu, \mathbf{c}}]$ is finite in $\chign{\Aff}$. The size of this orbit can be estimated as follows:  \begin{equation}\label{eqcounrrhoaff}   \phi(N)(2N-\phi(N))\cdot N^{n'-1}\leq \cardn (\MCG \cdot [\rho_{\mu, \mathbf{c}}])\leq  (N^2-1)N^{n'-1}\, , \end{equation}
where $n':=\cardn \{i\in \intint{1}{n}~|~c_i\neq 0\}$ and $\phi$ denotes the Euler totient function.
 \end{prop}
 \begin{rem}
Observe that  the estimate \eqref{eqcounrrhoaff} yields an equality if $N$ is a prime number.
 \end{rem}
\begin{proof}For convenience we shall represent the elements of $\mathrm{Aff}(\mathbb{C})$ by degree one polynomials $az+b$, as in page~$\pageref{Sec Aff case}$. Denote $$ D_{\mathbf{c}}:=\mu^{\mathbb{Z}}c_1+\ldots+\mu^{\mathbb{Z}}c_n\,;  \quad \quad S_{\mu, d}:=S_{\mu, d}^1\cup S_{\mu,d}^2\,;  \quad \quad R_{\mu, \mathbf{c}, d}\,;  \quad \quad O_{\mu, \mathbf{c}} \, $$ as in the proof of Proposition  \ref{SpecialCaseg1c}.
Moreover, denote  
  $$R^{pure}_{\mu,\mathbf{c}, d}:=\left\{ \left(\begin{array}{c} z+\tilde c_1 \\ \vdots \\ z+\tilde c_n \end{array}\right)  ~\middle|~\begin{array}{l}
  \tilde{c}_i\in \mu^\mathbb{Z}  c_i\quad \quad  \forall i\in \intint{1}{n}\vspace{.3cm}
  \\
   d=\sum_{i=1}^n \tilde c_i
 \end{array}\right\}$$
$$  O^{pure}_{\mu,\mathbf{c}}:= \bigcup_{d\in D_{\mathbf{c}}}
\left\{
\left(\varphi_\alpha \, , \varphi_\beta \, , \varphi_1\, , \ldots\, , \varphi_n\right)
 ~\middle|~
 \begin{array}{l}
  \left(\begin{array}{rcl}\varphi_\alpha\\\varphi_\beta   \end{array}\right)
   \in S_{\mu,d} 
    \, , \quad
      \left(\begin{array}{l} \varphi_{1}\\ \vdots \\\varphi_{n} \end{array}\right)
       \in R^{pure}_{\mu,\mathbf{c},d} \end{array} \right\}
\, .$$ 
We shall denote by $[O_{\mu , \mathbf{c} }]$ and $[O^{pure}_{\mu , \mathbf{c} }]$ the respective images of $O_{\mu , \mathbf{c} }$ and $O^{pure}_{\mu , \mathbf{c} }$ in $\chign{\Aff}$. By a slight refinement of the proof of Proposition  \ref{SpecialCaseg1c}, we have 
\begin{equation}\label{eqcurlyI} \MCG \cdot [\rho_{\mu, \mathbf{c}}]\subset   [O^{pure}_{\mu , \mathbf{c} }] \, .\end{equation}
Indeed, recall that the pure subgroup $\MCG$ of $\fMCG$ is the subgroup that respects the labellings of the punctures. Each pure element $\tau$ of $\liftMCG$ transforms the generators $\gamma_i$ into conjugates $\zeta_i^{-1} \gamma_i\zeta_i$. As we have $\rho(\zeta_i)=\mu^m z+d$ for suitable $m\in \Z$, $d\in\C$, we deduce $(\tau\cdot\rho)(\gamma_i)=\mu^m c_i$.
This proves  the inclusion \eqref{eqcurlyI}.

 Moreover, using Table $\ref{table genus 1}$ page \pageref{table genus 1},
we can check successively:
\begin{enumerate}[(a)] 
\item \label{point 1} as observed in the proof of Lemma $\ref{Prep}$, any element $\rho=[*_1,*_2,z+\tilde{c}_1,\ldots z+\tilde{c}_n]$ of $[O_{\mu,c}]$ can be transformed into an element $\rho'=[
z+d/(\mu-1),\mu z,z+\tilde{c}_1,\ldots, z+\tilde{c}_n]$ by an element of $\langle \tau_1,\tau_2\rangle$, where $d=\sum_{i=1}^n\tilde{c}_i$;  \vspace{.1cm}
\item  for any $j\in \intint{1}{n}$, by the action of an element of $\Bn=\langle \sigma_i~\vert~i \in \intint{1}{n-1}\rangle$, the element $\rho'$ can be transformed into $[
z+d/(\mu-1)\,  , \mu z,z+{c}'_{1},\ldots, z+ {c}'_n]$, where ${c}'_1=\tilde{c}_j$, ${c}'_j=\tilde{c}_1$ and $ {c}'_i=\tilde{c}_i$ for $i\neq 1,j$; \vspace{.1cm}
\item for any $m_j\in \Z$, using a power of $\tilde{\tau}_3$, we transform this latter element into $[
z+{d'/(\mu-1)} ,\mu z ,$ $ {z+\mu^{m_j}{c}_{1}'} ,\ldots, z+ {c}_n'],$ where $d' =d+(\mu^{m_j}-1)\tilde{c}_j$. \vspace{.1cm}
\item reusing an element of $\Bn$
, one gets $\rho'$ altered only by replacing $\tilde{c}_j$ by $\mu^{m_j}\tilde{c}_j$ and $d$ by $d'$. \vspace{.1cm}
\end{enumerate}
This allows to infer that  any element $\rho=[*_1,*_2,z+\tilde{c}_1,\ldots z+\tilde{c}_n]$ of $[O_{\mu,c}]$ can be transformed into $[z+1/(\mu-1),\mu z,z+c_1,\ldots ,  z+c_n]$ by a suitable element of $\fMCG$. Reusing $(\ref{point 1})$, we deduce $\fMCG \cdot [\rho_{\mu, \mathbf{c}}]=[O_{\mu , \mathbf{c} }]$.
The conjunction of this equality and the inclusion \eqref{eqcurlyI} yields
$$\MCG \cdot [\rho_{\mu, \mathbf{c}}]= [O^{pure}_{\mu , \mathbf{c} }] \, .$$
Denote by $[S_{\mu, d}]_t$ and $ [O^{pure}_{\mu , \mathbf{c} }]_t$ the set of equivalence classes of $S_{\mu, d}$ and $ O^{pure}_{\mu , \mathbf{c} }$ respectively modulo conjugation by translations.  
For each $d\in D_{\mathbf{c}}$, the cardinality of $[S_{\mu, d}]_t$ equals the cardinality of 
$$
K_N:=\left\{ (k_1, k_2) \in \intint{1}{N}^2~\middle|~\gcd (k_1, k_2, N)=1 \right\} \, .$$
Indeed, for $\{i,j\}=\{1,2\}$, the elements of $S_{\mu, d}^i$ that are not conjugated by a translation to an element of $S_{\mu, d}^j$ are precisely those corresponding to $k_j=0$ and $\gcd (k_i, N)=1$.  
We can estimate
$$\phi(N)(2N-\phi(N)) \leq \cardn([S_{\mu, d}]_t) \leq N^2-1\,. $$
These inequalities are readily derived from the inclusions 
\[\left\{ (k_1, k_2) \in \intint{1}{N}^2~ \middle| ~\gcd (k_1, N)=1~or~\gcd (k_2, N)=1\right\}\subset K_N\subset \intint{1}{N}^2\setminus \{(0,0)\}.\]

On the other hand,  conjugations by translations act trivially on $R^{pure}_{\mu,\mathbf{c}, d}$. By definition of $n'$, we have 
$$\cardn \left(  \bigcup_{d\in D_{\mathbf{c}}}R^{pure}_{\mu,\mathbf{c}, d} \right) = N^{n'}\, .$$
We deduce 
$$\phi(N)(2N-\phi(N))N^{n'}\leq \cardn [O^{pure}_{\mu , \mathbf{c} }]_t \leq (N^2-1)N^{n'}  \, .$$
The condition $\sum_{i=1}^{n}c_i=1$ ensures $n'>0$. In particular, there is an index $i_0\in \intint{1}{n}$ such that $c_{i_0}\neq 0$. 
Up to conjugation by powers of the linear transformation $\mu z$, we can normalize $\tilde{c}_{i_0}=c_{i_0}$ for each element in $[O^{pure}_{\mu , \mathbf{c} }]_t$, which yields 
$\cardn [O^{pure}_{\mu , \mathbf{c} }]=\frac{1}{N}\cardn [O^{pure}_{\mu , \mathbf{c} }]_t \, .$
\end{proof}

\subsection{Reduction to the affine case}\label{SecReduc}
Consider the natural inclusion $$\iota : (\C^*)^2\to  \GLtw\, ; \quad  (a_1, a_2)\mapsto \left(\begin{smallmatrix}a_1&0\\0& a_2\end{smallmatrix}\right)\, . $$

\begin{lem}\label{lemtotred} Let $g\geq 0, n\geq 0$.
Let $\rho \in \mathrm{Hom}(\gf,    \GLtw)$ be a totally reducible representation. Then there are scalar representations $\lambda_1, \lambda_2 \in \Hom( \gf ,  \mathbb{C}^*)$ such that $[\rho]=[\iota_*(\lambda_1, \lambda_2)]\, \in \, \chign{ \GLtw}.$
Moreover, we have \begin{equation}\label{eqcomparerscalaire}\frac{1}{2}\max \{\cardn(\MCG\cdot  \lambda_i)~|~i\in \{1,2\}\} \leq \cardn( \MCG\cdot [\rho])\leq\cardn(\MCG\cdot  \lambda_1)\cdot \cardn(\MCG\cdot  \lambda_2)\, .\end{equation}
In particular, the following are equivalent:
\begin{itemize}
\item $\MCG \cdot [\rho]$ is a finite subset of $\chign{ \GLtw}$. \vspace{.2cm}
\item $\MCG \cdot \lambda_i$  is a finite subset of $\mathrm{Hom}(\gf, \C^*)$ for $i=1,2$.
\end{itemize} 
\end{lem}
\begin{proof} 
The image of the map $\iota_*$ from $\mathrm{Hom}(\gf, (\C^*)^2)$ to $\chign{ \GLtw}$ is obviously the set of conjugacy classes of totally reducible representations.  

Note that by definition, the action of $\MCG$ on  $\iota_*\mathrm{Hom}(\gf, (\C^*)^2)$, induced by the action on $\mathrm{Hom}(\gf, (\C^*)^2)$,
coincides with the action of $\MCG$ on $\chign{ \GLtw}$. 
Moreover, we have 
\begin{equation}\label{eqcomparerscalaire1}
 \frac{1}{2}\cdot \cardn(\MCG\cdot (\lambda_1\, ,  \lambda_2))\leq \cardn( \MCG\cdot [\iota_*(\lambda_1\, ,  \lambda_2)])\leq  \cardn(\MCG\cdot (\lambda_1\, ,  \lambda_2))
\, .\end{equation}
Indeed, the second inequality is obvious, and the first one follows from the fact that if $[\iota_*(\lambda_1\, ,  \lambda_2)]=[\iota_*(\lambda_1'\, ,  \lambda_2')]$ then either $(\lambda_1\, ,  \lambda_2)=(\lambda_1'\, ,  \lambda_2')$ or 
$(\lambda_1\, ,  \lambda_2)=(\lambda_2'\, ,  \lambda_1').$

On the other hand, we have \begin{equation}\label{eqcomparerscalaire2}
 \max \{\cardn(\MCG\cdot  \lambda_i)~|~i\in \{1,2\}\} \leq \cardn( \MCG\cdot (\lambda_1\, ,  \lambda_2))\leq\cardn(\MCG\cdot  \lambda_1)\cdot \cardn(\MCG\cdot  \lambda_2) 
\, .
\end{equation}

We conclude by noticing that \eqref{eqcomparerscalaire1} and \eqref{eqcomparerscalaire2} imply \eqref{eqcomparerscalaire}.
\end{proof}

\begin{rem} The equality $[\rho]=[\iota_*(\lambda_1, \lambda_2)] \in \chign{ \GLtw}$ in the above Lemma is commonly written as $\rho = \lambda_1\oplus \lambda_2$. We adopted this notation in the statement of Theorem \ref{mainthm dynamics}, and we will use it in its proof. 
\end{rem}

Consider the natural inclusion  $$\iota\iota :~\mathrm{Aff}(\mathbb{C})=\left\{ \begin{pmatrix}a & b \\0&1 \end{pmatrix}~\middle|~a, b\in \mathbb{C},~a\neq 0\right\} {\hookrightarrow} ~\GLtw\, .$$  

\begin{lem}\label{lemred2} Let $g>0, n\geq 0$ and let  $\rho \in \mathrm{Hom}(\gf,    \GLtw)$ be a reducible but not totally reducible representation. Then there is a unique $\lambda \in \mathrm{Hom}(\gf, \C^*)$ and a unique conjugacy class 
$[\rho_{\mathrm{Aff}}] \in \chign{\mathrm{Aff}(\mathbb{C})}$ such that 
$[\rho]  = [\lambda \otimes  \iota\iota_*\rho_{\mathrm{Aff}}] \, \in \, \chign{ \GLtw}.$
Moreover, we have 
 \begin{equation}\label{eqcompareraff} \max \{\cardn(\MCG\cdot  \lambda)\, , \cardn(\MCG\cdot [\rho_{\mathrm{Aff}}])\} \leq \cardn( \MCG\cdot [\rho])\leq\cardn(\MCG\cdot  \lambda)\cdot \cardn(\MCG\cdot [\rho_{\mathrm{Aff}}])\, .\end{equation}
In particular, the following are equivalent. \begin{itemize} \item $\MCG \cdot [\rho]$ is a finite subset of $\chign{ \GLtw}$. \vspace{.2cm} \item  $\MCG \cdot \lambda$  is a finite subset of $\mathrm{Hom}(\gf, \C^*)$ and   $\MCG \cdot [\rho_{\mathrm{Aff}}]$  is a finite subset of $ \chign{\mathrm{Aff}(\mathbb{C})}$. 
 \end{itemize}
\end{lem}
\begin{proof}
The unique decomposition statement has been proven in Lemma \ref{lemred}  in Part A of the present paper. This Lemma also yields \eqref{eqcompareraff}. 
 \end{proof}

 \subsection{Proof of Theorem \ref{mainthm dynamics}}\label{Sec Proof A}
 
 \begin{thmB1}\hypertarget{cortotred}{} Let $g\geq 0, n\geq 0$.
Let $\rho \in \mathrm{Hom}(\gf,    \GLtw)$ be totally reducible, \textit{i.e.} $\rho = \lambda_1\oplus \lambda_2$ is a direct sum of scalar representations. The following are equivalent:
\begin{itemize}
\item the orbit $\MCG \cdot [\rho]$ in $\chign{ \GLtw}$ is finite. \vspace{.1cm}
\item the subgroup $\mathrm{Im}(\rho)$ of $\GLtw$ has finite order.
\end{itemize} 
Moreover, if  the orbit $\MCG \cdot [\rho]$ is finite, then its size can be estimated as follows:
\begin{equation}\label{eqconcltotred} \frac{1}{2}\max \{\cardn(\mathrm{Im}(\lambda_i))^{2g-1}~|~i\in \{1,2\}\} \leq \cardn( \MCG\cdot [\rho])\leq\cardn(\mathrm{Im}(\rho))^{2g}\, .\end{equation}
 \end{thmB1}
 
 \begin{proof}
 From Lemma $\ref{lemtotred}$, finiteness of the orbit $\MCG \cdot [\rho]$ in $\chign{ \GLtw}$ is tantamount to the finiteness of the orbits $\MCG \cdot \lambda_i\subset\mathrm{Hom}(\gf, \C^*)$ for $i=1,2$. Since $[\fMCG : \MCG]=n!$ is finite, finiteness of $\MCG \cdot \lambda_i$   is equivalent to the finiteness of $\fMCG \cdot \lambda_i$. Proposition $\ref{abcase}$ establishes that $\fMCG \cdot \lambda_i$ is finite if and only if  $\mathrm{Im}(\lambda_i)$ is finite. 
This proves the equivalence in the statement.

 The left inequality in \eqref{eqconcltotred} follows from Lemma \ref{lemtotred} and Proposition \ref{propscalar}.

Each pure element $\tau$ of $\liftMCG$ transforms the generators $\gamma_i$ into conjugates. By abelianity, for $\rho'=\tau\cdot \rho$ and any $i\in \intint{1}{n}$, we get $\rho'(\gamma_i)=\rho(\gamma_i)$.  We deduce the right inequality in \eqref{eqconcltotred}. \end{proof}
 
 \begin{thmB2}\hypertarget{corred}{} Let $g>0, n\geq 0$ and let  $\rho  \in \mathrm{Hom}(\gf,    \GLtw)$ be a reducible but not totally reducible representation. The following are equivalent:
\begin{itemize}
\item the orbit $\MCG \cdot [\rho]$ in $\chign{ \GLtw}$ is finite. \vspace{.1cm}
 \item $g=1\, , \, n>0$, there are  a scalar representation $\lambda \in \mathrm{Hom}(\gf,    \C^*)$ and an affine representation $\rho_{\mu, \mathbf{c}}\in  \mathrm{Hom}(\gf,    \Aff)$ as in Proposition  \ref{lemcomptermuc} , such that 
 $$[\rho] \in \MCG \cdot [\lambda\otimes \rho_{\mu, \mathbf{c}}]\, . $$  \end{itemize} 
Moreover, if  the orbit $\MCG \cdot [\rho]$ is finite, then its size can be estimated as follows:
\begin{equation}\label{eqcountrhoed}   \max\left\{ N_2\, , \phi(N)(2N-\phi(N))N^{n'-1}\right\}\leq \cardn (\MCG \cdot [\rho])\leq  (N^2-1)N^{n'-1}N_2^{2}\, , \end{equation}
where $n':=\cardn \{i\in \intint{1}{n}~|~\rho(\gamma_i)\not \in \C^*I_2\}$, $N:=\mathrm{order}(\mu)$, $N_2=\cardn(\mathrm{Im}(\lambda))$ and $\phi$ is the Euler totient function. 
 \end{thmB2}
 \begin{proof}   From Lemma \ref{lemred2}, we know that $[\rho]$ admits a unique decomposition $[\rho]=[\lambda \otimes \rho_{\mathrm{Aff}}]$, where $\lambda$ is a scalar representation and $\rho_{\mathrm{Aff}}$ is an affine representation. Moreover, since $\rho$ is not totally reducible, the affine representation  $\rho_{\mathrm{Aff}}$ has non-abelianimage. Still by Lemma \ref{lemred2}, the orbit $\MCG \cdot [\rho]\subset \chign{ \GLtw}$ is finite if and only if the orbits 
$\MCG \cdot \lambda \subset \Hom(\gf, \C^*)$ and $\MCG \cdot [\rho_{\mathrm{Aff}}]\subset \chign{ \Aff}$ are finite. 
From Proposition $\ref{abcase}$, the orbit $\MCG \cdot \lambda \subset \Hom(\gf, \C^*)$ is finite if and only if $\lambda$ has finite image. Since $[\fMCG : \MCG]=n!$ is finite, finiteness of the orbit $\MCG \cdot [\rho_{\mathrm{Aff}}]\subset \chign{ \Aff}$   is equivalent to the finiteness of the orbit $\fMCG \cdot [\rho_{\mathrm{Aff}}]\subset \chign{ \Aff}$. Since $\rho_{\mathrm{Aff}}$ has non-abelianimage, by the Propositions \ref{NABCaseg2} and  \ref{SpecialCaseg1c}, the finiteness of the latter orbit is equivalent to  $g=1\, , \, n>0$ and $[\rho_{\mathrm{Aff}}]\in \fMCG \cdot [\rho_{\mu, \mathbf{c}'}]$ for a convenient choice of a non-trivial root of unity $\mu$ and
$\mathbf{c}'=(c_1', \ldots, c_n') \in \C^n$ with $\sum_{i=1}^{n}c_i'=1$.  Composing with a suitable element of $B_n$ shows this is also equivalent to $[\rho_{\mathrm{Aff}}]\in \MCG \cdot [\rho_{\mu, \mathbf{c}}]$, for some $\mathbf{c}\in \mathfrak{S}_n\cdot \mathbf{c}'$. This proves the equivalence in the statement. 

The estimate \eqref{eqcountrhoed} follows from Lemma \ref{lemred2}, Proposition \ref{propscalar} and Proposition \ref{lemcomptermuc}, taken into account that 
\[\cardn \{i\in \intint{1}{n}~|~\rho(\gamma_i)\not \in \C^*I_2\} =\cardn \{i\in \intint{1}{n}~|~\rho_{\mathrm{Aff}}(\gamma_i)\neq \mathrm{id}\}=\cardn \{i\in \intint{1}{n}~|~c_i\neq 0\}.\] \end{proof}

\bibliographystyle{abbrv}
\bibliography{biblio}

\begin{thebibliography}{10}

\bibitem{MR2807457}
E.~Arbarello, M.~Cornalba, and P.~A. Griffiths.
\newblock {\em Geometry of algebraic curves. {V}olume {II}}, volume 268 of {\em
  Grundlehren der Mathematischen Wissenschaften [Fundamental Principles of
  Mathematical Sciences]}.
\newblock Springer, Heidelberg, 2011.
\newblock With a contribution by Joseph Daniel Harris.

\bibitem{MR3069440}
E.~Artin.
\newblock Theorie der {Z}\"opfe.
\newblock {\em Abh. Math. Sem. Univ. Hamburg}, 4(1):47--72, 1925.

\bibitem{MR0142552}
W.~L. Baily, Jr.
\newblock On the automorphism group of a generic curve of genus {$>2$}.
\newblock {\em J. Math. Kyoto Univ.}, 1(101--108; correction):325, 1961/1962.

\bibitem{2017arXiv170700071B}
I.~{Biswas}, S.~{Gupta}, M.~{Mj}, and J.~P. {Whang}.
\newblock {Surface group representations in ${\rm SL}_2({\mathbb C})$ with
  finite mapping class orbits}.
\newblock {\em arXiv e-prints}, page arXiv:1707.00071, Jun 2017.

\bibitem{MR2346015}
M.~Cornalba.
\newblock Erratum: ``{O}n the locus of curves with automorphisms'' [{A}nn.
  {M}at. {P}ura {A}ppl. (4) {\bf 149} (1987), 135--151; mr0932781].
\newblock {\em Ann. Mat. Pura Appl. (4)}, 187(1):185--186, 2008.

\bibitem{MR3300949}
G.~Cousin.
\newblock Projective representations of fundamental groups of quasiprojective
  varieties: a realization and a lifting result.
\newblock {\em C. R. Math. Acad. Sci. Paris}, 353(2):155--159, 2015.

\bibitem{cousinisom}
G.~Cousin.
\newblock Algebraic isomonodromic deformations of logarithmic connections on
  the {R}iemann sphere and finite braid group orbits on character varieties.
\newblock {\em Math. Ann.}, 367(3-4):965--1005, 2017.

\bibitem{cousin2016finite}
G.~Cousin and D.~Moussard.
\newblock Finite braid group orbits in {$\mathrm{Aff}(\mathbf{C})$}-character
  varieties of the punctured sphere.
\newblock {\em Int. Math. Res. Not. IMRN}, 2018(11):3388--3442, 2018.

\bibitem{MR1841091}
O.~Debarre.
\newblock {\em Higher-dimensional algebraic geometry}.
\newblock Universitext. Springer-Verlag, New York, 2001.

\bibitem{MR0417174}
P.~Deligne.
\newblock {\em \'{E}quations diff\'erentielles \`a points singuliers
  r\'eguliers}.
\newblock Lecture Notes in Mathematics, Vol. 163. Springer-Verlag, Berlin,
  1970.

\bibitem{MR2850125}
B.~Farb and D.~Margalit.
\newblock {\em A primer on mapping class groups}, volume~49 of {\em Princeton
  Mathematical Series}.
\newblock Princeton University Press, Princeton, NJ, 2012.

\bibitem{MR2667785}
V.~Heu.
\newblock Universal isomonodromic deformations of meromorphic rank 2
  connections on curves.
\newblock {\em Ann. Inst. Fourier (Grenoble)}, 60(2):515--549, 2010.

\bibitem{Hubbard}
J.~H. Hubbard.
\newblock {\em Teichm\"uller theory and applications to geometry, topology, and
  dynamics. {V}ol. 1}.
\newblock Matrix Editions, Ithaca, NY, 2006.
\newblock Teichm{\"u}ller theory, With contributions by Adrien Douady, William
  Dunbar, Roland Roeder, Sylvain Bonnot, David Brown, Allen Hatcher, Chris
  Hruska and Sudeb Mitra, With forewords by William Thurston and Clifford
  Earle.

\bibitem{MR1249482}
D.~Husemoller.
\newblock {\em Fibre bundles}, volume~20 of {\em Graduate Texts in
  Mathematics}.
\newblock Springer-Verlag, New York, third edition, 1994.

\bibitem{Krichever}
I.~Krichever.
\newblock Isomonodromy equations on algebraic curves, canonical transformations
  and {W}hitham equations.
\newblock {\em Mosc. Math. J.}, 2(4):717--752, 806, 2002.

\bibitem{MR0171269}
W.~B.~R. Lickorish.
\newblock A finite set of generators for the homeotopy group of a
  {$2$}-manifold.
\newblock {\em Proc. Cambridge Philos. Soc.}, 60:769--778, 1964.

\bibitem{Malgrange}
B.~Malgrange.
\newblock Sur les d\'eformations isomonodromiques. {I} et {II}.
\newblock In {\em Mathematics and physics (Paris, 1979/1982)}, volume~37 of
  {\em Progr. Math.}, pages 401--438. Birkh\"auser Boston, Boston, MA, 1983.

\bibitem{MR2611511}
P.~Monsky.
\newblock {\em T{he} {automorphism} {groups} {of} {algebraic} {curves}}.
\newblock ProQuest LLC, Ann Arbor, MI, 1962.
\newblock Thesis (Ph.D.)--The University of Chicago.

\bibitem{MR1748288}
B.~Poonen.
\newblock Varieties without extra automorphisms. {I}. {C}urves.
\newblock {\em Math. Res. Lett.}, 7(1):67--76, 2000.

\bibitem{GAGA}
J.-P. Serre.
\newblock G\'eom\'etrie alg\'ebrique et g\'eom\'etrie analytique.
\newblock {\em Ann. Inst. Fourier, Grenoble}, 6:1--42, 1955--1956.

\end{thebibliography}

\end{document}